%% file: Bumping.tex
\documentclass[12pt]{amsart}

\usepackage{microtype}

\usepackage{times}
\usepackage{pifont}
\usepackage{datetime}

\usepackage{amsmath,amssymb,commath,mathtools,mathabx}
\usepackage{xspace}

\usepackage[T2A,T1]{fontenc}
\usepackage[utf8]{inputenc}
\usepackage[russian,english]{babel}

\usepackage{dsfont}

\usepackage{hyperref}
\usepackage[capitalize,noabbrev]{cleveref}
\usepackage{enumitem}
\usepackage{units}
\usepackage{subfiles}
\usepackage[labelfont={normalsize}]{caption,subfig}

\usepackage[backgroundcolor=yellow]{todonotes}

\newcommand{\russian}[2]{#2}

\usepackage{tikz}
\usetikzlibrary{calc}

\usepackage[backend=bibtex,doi=true,isbn=false,url=false,style=alphabetic]{biblatex}
\newcommand{\biblio}{\printbibliography}

\makeatletter
\@ifclasswith{amsart}{externalize}%
{
	\usetikzlibrary{external} 
	\tikzexternalize[
	optimize=false,
	prefix=tikz/,
	mode=list and make]
	\makeatletter
	\renewcommand{\todo}[2][]{\tikzexternaldisable\@todo[#1]{#2}\tikzexternalenable}
	\makeatother
}%
{
	\newcommand{\tikzexternaldisable}{}
	\newcommand{\tikzexternalenable}{}
}
\makeatother

\author[M.~Marciniak]{Mikołaj Marciniak} 

\address{Interdisciplinary Doctoral School “Academia~Copernicana”, Faculty of
    Mathematics and Computer Science, Nicolaus Copernicus University in Toruń,
    ul.~Chopina~12/18, 87-100 Toruń, Poland} \email{marciniak@mat.umk.pl}

\author[{\L}.~Maślanka]{{\L}ukasz Ma\'slanka}

\address{Institute of Mathematics, Polish Academy of Sciences,
    \mbox{ul.~\'Sniadeckich 8}, 00-656 Warszawa, Poland}

\email{lmaslanka@impan.pl}

\author[P.~\'Sniady]{Piotr \'Sniady}

\address{Institute of Mathematics, Polish Academy of Sciences,
    \mbox{ul.~\'Sniadec\-kich 8,} \linebreak 00-656 Warszawa, Poland }

\email{psniady@impan.pl}

\newcommand{\Marciniak}{the~first named author}
\newcommand{\Sniady}{the~last named author}

\theoremstyle{definition}
\newtheorem{definition}{Definition}[section]

\theoremstyle{plain}
\newtheorem{theorem}[definition]{Theorem}
\newtheorem{proposition}[definition]{Proposition}
\newtheorem{corollary}[definition]{Corollary}
\newtheorem{lemma}[definition]{Lemma}

\newtheorem{conjecture}[definition]{Conjecture}
\newtheorem{fact}[definition]{Fact}

\theoremstyle{remark}

\newtheorem{remark}[definition]{Remark}
\newtheorem{question}[definition]{Question}

\newcommand{\N}{\mathbb{N}}
\newcommand{\R}{\mathbb{R}}

\newcommand{\PP}{\mathbb{P}}
\newcommand{\PPcond}[2]{\PP\left(#1  \;\middle\vert\; #2 \right)}
\newcommand{\PPcondCurly}[2]{\PP\left\{ #1  \;\middle\vert\; #2 \right\}}

\newcommand{\inrv}{\in}

\newcommand{\E}{\mathbb{E}}
\newcommand{\Z}{\mathbb{Z}}
\newcommand{\n}{^{[n]}}

\DeclareRobustCommand{\stirling}{\genfrac\{\}{0pt}{}}

\newcommand{\Sym}[1]{\mathfrak{S}_#1}
\newcommand{\tableau}{\mathcal{T}}
\newcommand{\diagrams}{\mathbb{Y}}
\newcommand{\Augmented}{\diagrams^*}

\DeclareMathOperator{\RSK}{RSK}

\DeclareMathOperator{\lazy}{lazy}
\DeclareMathOperator{\projective}{proj}

\DeclareMathOperator{\traj}{traj}
\DeclareMathOperator{\Pos}{Pos}

\DeclareMathOperator{\Plancherel}{Plan}

\begin{document}

\title[Poisson limit of bumping routes]%
{Poisson limit of bumping routes \\ in the~Robinson--Schensted correspondence}

\begin{abstract}
We consider the~Robinson--Schensted--Knuth algorithm applied to a~random input
and investigate the~shape of the~bumping route (in~the~vicinity of the~$y$-axis) 
when a~specified number is inserted into a~large
Plancherel-distributed random tableau. We show that after a~projective change of the
coordinate system the~bumping route converges in distribution to the~Poisson
process.
\end{abstract}



\maketitle

\section{Introduction}
\label{sec:preliminaries}

\subsection{Notations}

The~set of Young diagrams will be denoted by $\diagrams$; the~set of Young
diagrams with $n$ boxes will be denoted by $\diagrams_n$. The~set $\diagrams$
has a~structure of an~oriented graph, called \emph{Young graph}; a~pair
$\mu\nearrow\lambda$ forms an~oriented edge in this graph if the~Young diagram
$\lambda$ can~be created from the~Young diagram $\mu$ by addition of a~single
box.

We will draw Young diagrams and tableaux in the~French convention with the
Cartesian~coordinate system $Oxy$, cf.~\cref{fig:RSK,fig:tableauFrench}. We
index the~rows and the~columns of tableaux by \emph{non-negative} integers from
$\N_0=\{0,1,2,\dots\}$. In particular, if $\Box$ is a~box of a~tableau, we
identify it with the~Cartesian~coordinates of its \emph{lower-left corner}:
$\Box=(x,y)\in \N_0\times \N_0$. For a~tableau $\tableau$ we denote by
$\tableau_{x,y}$ its entry which lies in the~intersection of the~row $y\in\N_0$
and the~column $x\in\N_0$.
The position of the box $s$ in the tableau $\tableau$ 
will be denoted by $\Pos_s(\tableau)\in\N_0\times \N_0$. 

Also the~rows of any Young diagram $\lambda=(\lambda_0,\lambda_1,\dots)$ are
indexed by the~elements of~$\N_0$; in particular the~length of the~bottom row
of $\lambda$ is denoted by~$\lambda_0$.

\subsection{Schensted row insertion}

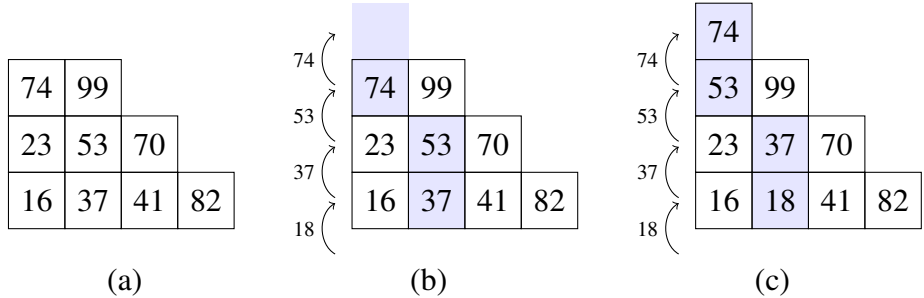
\begin{figure}[t]
    \centering
    \hfill
    \subfloat[]{
        \begin{tikzpicture}[scale=0.75]
            \clip (-0.1,-0.5) rectangle (4.1,4.5);

            \draw (0,0) rectangle +(1,1); 
            \node at (0.5,0.5) {16}; 
            \draw (1,0) rectangle +(1,1); 
            \node at (1.5,0.5) {37}; 
            \draw (2,0) rectangle +(1,1); 
            \node at (2.5,0.5) {41}; 
            \draw (3,0) rectangle +(1,1); 
            \node at (3.5,0.5) {82}; 
            \draw (0,1) rectangle +(1,1); 
            \node at (0.5,1.5) {23}; 
            \draw (1,1) rectangle +(1,1); 
            \node at (1.5,1.5) {53}; 
            \draw (2,1) rectangle +(1,1); 
            \node at (2.5,1.5) {70}; 
            \draw (0,2) rectangle +(1,1); 
            \node at (0.5,2.5) {74}; 
            \draw (1,2) rectangle +(1,1); 
            \node at (1.5,2.5) {99};

        \end{tikzpicture}
        \label{fig:RSKa}
    }
    \hfill
    \subfloat[]{
        \begin{tikzpicture}[scale=0.75]
            \clip (-1.5,-0.5) rectangle (4.1,4.5);
            \fill[blue!10] (1,0) rectangle +(1,1);
            \fill[blue!10] (1,1) rectangle +(1,1);
            \fill[blue!10] (0,2) rectangle +(1,1);
            \fill[blue!10] (0,3) rectangle +(1,1);

            \draw (0,0) rectangle +(1,1); 
            \node at (0.5,0.5) {16}; 
            \draw (1,0) rectangle +(1,1); 
            \node at (1.5,0.5) {37}; 
            \draw (2,0) rectangle +(1,1); 
            \node at (2.5,0.5) {41}; 
            \draw (3,0) rectangle +(1,1); 
            \node at (3.5,0.5) {82}; 
            \draw (0,1) rectangle +(1,1); 
            \node at (0.5,1.5) {23}; 
            \draw (1,1) rectangle +(1,1); 
            \node at (1.5,1.5) {53}; 
            \draw (2,1) rectangle +(1,1); 
            \node at (2.5,1.5) {70}; 
            \draw (0,2) rectangle +(1,1); 
            \node at (0.5,2.5) {74}; 
            \draw (1,2) rectangle +(1,1); 
            \node at (1.5,2.5) {99};

            \draw[->] (-0.3,-0.45) to[bend left=60] (-0.3,0.45);
            \draw[->] (0,1) +(-0.3,-0.45) to[bend left=60] +(-0.3,0.45);
            \draw[->] (0,2) +(-0.3,-0.45) to[bend left=60] +(-0.3,0.45);
            \draw[->] (0,3) +(-0.3,-0.45) to[bend left=60] +(-0.3,0.45);
            
            \tiny
            \node[] at (-0.85,0) {18};
            \node[] at (-0.85,1) {37};
            \node[] at (-0.85,2) {53};
            \node[] at (-0.85,3) {74};
            
        \end{tikzpicture}
        \label{fig:RSKb}
    }
    \hfill
    \subfloat[]{
        \begin{tikzpicture}[scale=0.75]
            \clip (-1.5,-0.5) rectangle (4.1,4.5);
            \fill[blue!10] (1,0) rectangle +(1,1);
            \fill[blue!10] (1,1) rectangle +(1,1);
            \fill[blue!10] (0,2) rectangle +(1,1);
            \fill[blue!10] (0,3) rectangle +(1,1);
            
            \draw (0,0) rectangle +(1,1); 
            \node at (0.5,0.5) {16}; 
            \draw (1,0) rectangle +(1,1); 
            \node at (1.5,0.5) {18}; 
            \draw (2,0) rectangle +(1,1); 
            \node at (2.5,0.5) {41}; 
            \draw (3,0) rectangle +(1,1); 
            \node at (3.5,0.5) {82}; 
            \draw (0,1) rectangle +(1,1); 
            \node at (0.5,1.5) {23}; 
            \draw (1,1) rectangle +(1,1); 
            \node at (1.5,1.5) {37}; 
            \draw (2,1) rectangle +(1,1); 
            \node at (2.5,1.5) {70}; 
            \draw (0,2) rectangle +(1,1); 
            \node at (0.5,2.5) {53}; 
            \draw (1,2) rectangle +(1,1); 
            \node at (1.5,2.5) {99}; 
            \draw (0,3) rectangle +(1,1); 
            \node at (0.5,3.5) {74};

            \draw[->] (-0.3,-0.45) to[bend left=60] (-0.3,0.45);
            \draw[->] (0,1) +(-0.3,-0.45) to[bend left=60] +(-0.3,0.45);
            \draw[->] (0,2) +(-0.3,-0.45) to[bend left=60] +(-0.3,0.45);
            \draw[->] (0,3) +(-0.3,-0.45) to[bend left=60] +(-0.3,0.45);
            
            \tiny
            \node[] at (-0.85,0) {18};
            \node[] at (-0.85,1) {37};
            \node[] at (-0.85,2) {53};
            \node[] at (-0.85,3) {74};
            
        \end{tikzpicture}
        \label{fig:RSKc}
    }
    
    \caption{\protect\subref{fig:RSKa} The original tableau $\tableau$.
        \protect\subref{fig:RSKb} We consider the Schensted row insertion of the number
        $18$ to the~tableau~$\tableau$. The~highlighted boxes form the corresponding
        bumping route. The small numbers on the left (next
        to the arrows) indicate the~inserted/bumped numbers. \protect\subref{fig:RSKc} The
        output $\tableau\leftarrow 18$ of the Schensted insertion. } \label{fig:RSK}
\end{figure}

\newcommand{\letter}{a}

\emph{The Schensted row insertion} is an algorithm which takes as an input a
tableau $\tableau$ and some number $\letter$. The number $\letter$ is inserted
into the first row (that is,~the bottom row, the row with the index~$0$) of
$\tableau$ to the~leftmost box which contains an entry which is strictly bigger
than $\letter$.

In the case when the row contains no entries which
are bigger than~$\letter$, the~number~$\letter$ is inserted into the leftmost
empty box in this row and the algorithm terminates. 

If, however, the number $\letter$ is inserted into a box which was not empty,
the~previous content $\letter'$ of the box is \emph{bumped} into the second row
(that is, the~row with the index $1$). This means that the algorithm is
iterated but this time the number $\letter'$ is inserted into the second row to
the leftmost box which contains a number bigger than $\letter'$. 
We repeat these steps of row insertion and bumping 
until some number is inserted into a previously empty box.
This process is illustrated on \cref{fig:RSKb,fig:RSKc}. The outcome of
the Schensted insertion is defined as the~result of the~aforementioned procedure;
it will be denoted by $\tableau \leftarrow \letter$. 

\medskip

Note that this procedure is well defined also in the~setup when $\tableau$ is
an~\emph{infinite} tableau (see \cref{fig:tableauFrench}  for an~example), even
if the~above procedure does not terminate after a~finite number of steps.

\subsection{Robinson--Schensted--Knuth algorithm} 

For the purposes of this article we consider a simplified version of \emph{the
    Robinson--Schensted--Knuth algorithm}; for this reason we should rather call it
\emph{the Robinson--Schensted algorithm}. Nevertheless, we use the first name
because of its well-known acronym RSK. The RSK algorithm associates to a finite
sequence $w=(w_1,\dots,w_\ell)$ a~pair of tableaux: \emph{the insertion tableau
    $P(w)$} and \emph{the recording tableau~$Q(w)$}.

The insertion tableau 
\begin{equation}
    \label{eq:insertion}
    P(w) = \Big( \big( (\emptyset \leftarrow w_1) \leftarrow w_2 \big) 
    \leftarrow  \cdots \Big) \leftarrow w_\ell
\end{equation}
is defined as the result of the iterative Schensted row insertion applied to the
entries of the sequence $w$, starting from \emph{the empty tableau $\emptyset$}.

The recording tableau $Q(w)$ is defined as the standard Young tableau of the
same shape as $P(w)$ in which each entry is equal to the number of 
the~iteration of \eqref{eq:insertion} in which the given box of~$P(w)$ stopped
being empty; in other words the entries of $Q(w)$ give the order in which the
entries of the~insertion tableau were filled.

The tableaux $P(w)$ and $Q(w)$ have the same shape; we will denote this common
shape by $\RSK(w)$ and call it \emph{the RSK shape associated to $w$}.

The RSK algorithm is of great importance in algebraic combinatorics, especially in
the context of the representation theory \cite{Fulton1997}. 

\subsection{Plancherel measure, Plancherel growth process}
\label{sec:pgp}

Let $\Sym{n}$ denote the symmetric group of order $n$. We will view each
permutation $\pi \in\Sym{n}$ as a sequence $\pi=(\pi_1,\dots,\pi_n)$ which has
no repeated entries, and such that $\pi_1,\dots,\pi_n\in\{1,\dots,n\}$. The
restriction of RSK to the symmetric group is a bijection which to a given
permutation from $\Sym{n}$ associates a pair $(P,Q)$ of standard Young tableaux
of the same shape, consisting of $n$ boxes. A~fruitful area of study concerns
the RSK algorithm applied to a uniformly random permutation from $\Sym{n}$,
especially asymptotically in the limit $n\to\infty$, see \cite{Romik2015a} and
the references therein.

\emph{The Plancherel measure} on $\diagrams_{n}$, denoted $\Plancherel_n$, is
defined as the probability distribution of the random Young diagram $\RSK(w)$
for a uniformly random permutation $w\inrv\Sym{n}$.

\medskip

An \emph{infinite standard Young tableau} 
\cite[Section~2.2]{Kerov1999}
is a filling of the boxes
in a subset of the upper-right quarterplane with positive integers, such that
each row and each column is increasing, and each positive integer is used
exactly once.
There is a natural bijection between the set of infinite standard Young tableaux and the set of infinite paths in the Young graph
\begin{equation} 
    \label{eq:PGP0} 
    \lambda^{(0)} \nearrow \lambda^{(1)} \nearrow \cdots \qquad
    \text{with } \lambda^{(0)}=\emptyset;
\end{equation}
this bijection is given by setting $\lambda^{(n)}$ to be the set of boxes of a
given infinite standard Young tableau which are $\leq n$. 

If
$w=(w_1,w_2,\dots)$ is an \emph{infinite} sequence, the recording tableau $Q(w)$
is defined as the infinite standard Young tableau in which each non-empty entry
is equal to the number of the iteration in the infinite sequence of Schensted row insertions
\[   \big( (\emptyset \leftarrow w_1) \leftarrow w_2 \big) 
\leftarrow  \cdots 
\]
in which the corresponding box stopped being empty, see \cite[Section~1.2.4]{RomikSniady-AnnPro}. Under the aforementioned bijection, the recording tableau $Q(w)$ corresponds to the sequence \eqref{eq:PGP0}
with 
\[ \lambda^{(n)}= \RSK( w_1,\dots, w_n). \]

Let $\xi=(\xi_1,\xi_2,\dots)$ be an infinite sequence of independent, identically
distributed random variables with the uniform distribution $U(0,1)$ on the unit
interval $[0,1]$. \emph{The Plan\-che\-rel measure on the set of infinite standard
    Young tableaux} is defined as the probability distribution of $Q(\xi)$.
Any sequence with the same
probability distribution as \eqref{eq:PGP0} with 
\begin{equation}
    \label{eq:lambda-rsk}
    \lambda^{(n)}= \RSK( \xi_1,\dots, \xi_n) 
\end{equation}
will be called \emph{the Plancherel growth process} \cite{Kerov1999}. It turns
out that the Plancherel growth process is a Markov chain \cite[Sections~2.2 and
2.4]{Kerov1999}. For a~more systematic introduction to this topic we recommend
the monograph \cite[Section~1.19]{Romik2015a}.

\subfile{figures-Poisson/big_simulation/random_SYT.tex}

\subsection{Bumping route}

The~\emph{bumping route} consists of the~boxes the~entries of which were
changed by the~action of Schensted insertion, 
including the~last, newly created box, see \cref{fig:RSKb,fig:RSKc}. 
The~bumping route will be denoted by
$\tableau{\leftsquigarrow a}$ or by $\tableau_{\leftsquigarrow a}$ depending on
current typographic needs. In any row $y\in\N_0$ there is at most one box from
the~bumping route $\tableau{\leftsquigarrow a}$; we denote by
$\tableau_{\leftsquigarrow a}(y)$ its $x$-coordinate. We leave
$\tableau_{\leftsquigarrow a}(y)$ undefined if such a~box does not exist. In
this way
\begin{equation}
\label{eq:bumping_points} 
\tableau{\leftsquigarrow a}~= 
\Big\{ \big( \tableau_{\leftsquigarrow a}(y),y\big) : y\in\N_0\Big\}.
\end{equation}
For example, for the~tableau $\tableau$ from \cref{fig:RSKa} and $a=18$ we have
\[ \tableau_{\leftsquigarrow a}(y) = 
\begin{cases}
1 & \text{for } y\in\{0,1\},  \\
0 & \text{for } y\in\{2,3\}. 
\end{cases}
\]

The~bumping route $\tableau{\leftsquigarrow a}$ can~be visualized either as a
collection of its boxes or as a~plot of the~function
\[ x(y) = \tableau_{\leftsquigarrow a}(\lfloor y \rfloor), \qquad y\in\R_+, \]
cf.~the~thick red line on \cref{fig:tableauFrench}.

\subsection{Bumping routes for infinite tableaux} 

Any bumping route which corresponds to an~insertion to a~\emph{finite} tableau
is, well,  also finite. This is disadvantageous when one aims at the
asymptotics of such a~bumping route in a~row of index $y$ in the~limit
$y\to\infty$. For such problems it would be preferable to work in a~setup in
which the~bumping routes are infinite; we present the~details in the~following.

Let us fix the~value of an~integer $m\in\N_0$.
Now, for an~integer $n\geq m$ we consider a~real number $0<\alpha_n<1$ and a
finite sequence $\xi=(\xi_1,\dots,\xi_n)$ of independent, identically distributed
random variables with the~uniform distribution $U(0,1)$ on the~unit interval
$[0,1]$. In order to remove some randomness from this picture we will condition
the~choice of $\xi$ in such a~way that there are exactly $m$ entries of $\xi$ which
are smaller than~$\alpha_n$; heuristically this situation is similar to a
scenario without conditioning, for the~choice of
\begin{equation}
\label{eq:alpha}
\alpha_n=\frac{m}{n}. 
\end{equation}
We will study the~bumping route
\begin{equation}
\label{eq:bumping-1}
P(\xi_1,\dots,\xi_n) \leftsquigarrow \alpha_n 
\end{equation}
in the~limit as $n\to\infty$ and $m$ is fixed.

Without loss of generality we may assume that the~entries of the~sequence~$\xi$
are all different. Let $\pi\inrv\Sym{n}$ be the~unique permutation which encodes
the~relative order of the~entries in the~sequence $\xi$, that is
\[ \left( \pi_i < \pi_j \right) \iff \left( X_i < X_j \right) \]
for any $1\leq i,j\leq n$. Since the~algorithm behind the Robinson--Schensted--Knuth
correspondence depends only on the~relative order of the~involved numbers and
not their exact values, it follows that the~bumping route \eqref{eq:bumping-1}
coincides with the~bumping route
\begin{equation}
\label{eq:bumping-2}
 P(\pi_1,\dots,\pi_n) \leftsquigarrow m+\nicefrac{1}{2}.
\end{equation}

The~probability distribution of $\pi$ is the~uniform measure on $\Sym{n}$; it
follows that the~probability distribution of the~tableau $P(\pi_1,\dots,\pi_n)$
which appears in \eqref{eq:bumping-2} is the~Plancherel measure $\Plancherel_n$
on the~set of standard Young tableaux with $n$ boxes. Since such a
Plancherel-distributed random tableau with $n$~boxes can~be viewed as a
truncation of an~infinite standard Young tableau~$\tableau$ with the Plancherel
distribution, the~bumping routes \eqref{eq:bumping-1} and \eqref{eq:bumping-2}
can~be viewed as truncations of the~infinite bumping route
\begin{equation}
\label{eq:bumping-3}
 \tableau\leftsquigarrow m + \nicefrac{1}{2}, 
 \end{equation}
see \cref{fig:tableauFrench} for an~example. 

\subsection{The~main problem: asymptotics of infinite bumping routes}

The~aim of the~current paper is to investigate the~asymptotics of the
\emph{infinite} bumping route \eqref{eq:bumping-3} in the~limit $m\to\infty$.

Heuristically, this corresponds to investigation of the~asymptotics of 
the~\emph{finite} bumping routes \eqref{eq:bumping-1} in the~simplified setting
when we do not condition over some additional properties of $\xi$, in the~scaling
in which $\alpha_n$ does not tend to zero too fast (so that $\lim_{n\to\infty}
\alpha_n n =\infty$, cf.~\eqref{eq:alpha}), but on the~other hand $\alpha_n$
should tend to zero fast enough so that the~bumping route is long enough that
our asymptotic questions are well defined. We will not pursue in this direction
and we will stick to the investigation of the~\emph{infinite} bumping route~\eqref{eq:bumping-3}.

\medskip

Even though Romik and {\Sniady} \cite{RomikSniadyBumping} considered the
asymptotics of \emph{finite} bumping routes, their discussion is nevertheless
applicable in our context. It shows that in the~\emph{balanced scaling} when we
focus on the~part of the~bumping route with the~Cartesian~coordinates $(x,y)$ of
magnitude $x,y=O \! \left( \sqrt{m} \right)$, the~shape of the~bumping route
(scaled down by the~factor $\frac{1}{\sqrt{m}}$) converges in probability
towards an~explicit curve, which we refer to as \emph{the~limit bumping curve}, see
\cref{fig:simulated-linear} for an~illustration.

\subfile{figures-Poisson/multiple_bumping_routes/multiple_routes-linear.tex}

In the~current paper we go beyond the~scaling used by Romik and {\Sniady} and
investigate the~part of the~bumping route with the~Cartesian~coordinates of order
$x=O(1)$ and $y\gg \sqrt{m}$. This part of the~bumping curves was not visible
on \cref{fig:simulated-linear}; in order to reveal it one can~use the
semi-logarithmic plot, cf.~\cref{fig:simulated-log,fig:simulated-log-stretch-2}.

\subsection{The~naive hyperbola}
\label{sec:naive}

The~first step in this direction would be to stretch the
validity of the~results of Romik and {\Sniady} \cite{RomikSniadyBumping} beyond
their limitations and to expect that the~limit bumping curve describes the
asymptotics of the~bumping routes also in this new scaling. 
This would correspond to the~investigation of the~asymptotics of the
(non-rescaled) limit bumping curve $\big(x(y), y\big)$ in the~regime
$y\to\infty$. The~latter analysis was performed by {\Marciniak}
\cite{Marciniak2020}; one of his results is that
\[ \lim_{y\to\infty} x(y) y = 2;\]
in other words, for $y\to \infty$ the~\emph{non-rescaled} limit bumping curve
can~be approximated by the~hyperbola~$x y = 2$ while its \emph{rescaled}
counterpart which we consider in the~current paper by the~hyperbola
\begin{equation}
\label{eq:hyperbola}
x y = 2m
\end{equation} 
which is shown on \cref{fig:simulated-linear} as the~dashed line. At the very end of
\cref{sec:naive-hyperbola} we will discuss the~extent to which this naive
approach manages to confront the~reality.

\subfile{figures-Poisson/multiple_bumping_routes/multiple_routes-log.tex}

\subfile{figures-Poisson/multiple_bumping_routes/multiple_routes-log-stretch-2.tex}

\subsection{In which row a~bumping route reaches a~given column?}
\label{sec:in-which-row}

Let us fix some (preferably infinite) standard Young tableau $\tableau$.
The~bumping route in each step jumps to the~next row, directly up or to the
left to the~original column; in other words  
\[ \tableau_{\leftsquigarrow m+\nicefrac{1}{2}} (0) \geq  
\tableau_{\leftsquigarrow m+\nicefrac{1}{2}}(1) \geq \cdots \]
is a~weakly decreasing sequence of non-negative integers.

For $x,m\in\N_0$ we denote by 
\begin{equation}
\label{eq:Y} 
Y_x^{[m]} = Y_x = \min\left\{ y\in\N_0 :  
                \tableau_{\leftsquigarrow m+\nicefrac{1}{2}}(y) \leq x \right\}
\end{equation}
the~index of the~first row in which the~bumping route 
$\tableau {\leftsquigarrow m+\nicefrac{1}{2}}$ reaches the~column with
the index $x$ (or less, if the~bumping route skips the~column $x$ completely). 
For example, for the~tableau $\tableau$ from \cref{fig:tableauFrench} we have
\[ Y_0^{[3]}= 4, \quad Y_1^{[3]}= 2, \quad Y_2^{[3]}= 1, \quad 
Y_3^{[3]}=Y_4^{[3]}=\cdots= 0.
\]
If such a~row does not exist we set $Y_x=\infty$; the~following result shows
that we do not have to worry about such a~scenario.

\begin{proposition}
    \label{prop:is-finite}
    For a~random infinite standard Young tableau $\tableau$ with the~Plancherel
distribution
    \[ Y_x^{[m]} < \infty \qquad \text{for all } x,m\in\N_0\]
    holds true almost surely.
\end{proposition}

The~proof is postponed to \cref{sec:proof-prop:is-finite}.
For a~sketch of the~proof of an~equivalent result see the~work of Vershik 
\russian%
{}%
{\cite[Predlozhenie 4]{Vershik2020}}
who uses different methods.

\medskip

\begin{theorem}[The~main result]
    \label{thm:poisson-point} 
    
    Assume that $\tableau$ is an~infinite standard
Young tableau with the~Plancherel distribution. With the~above notations,
the~random set
    \[    \left( \frac{2 m}{Y^{[m]}_0}, \frac{2 m}{Y^{[m]}_1},  \dots \right) \]
    converges in distribution, as $m\to\infty$, to the Poisson point process on
$\R_+$ with the~unit intensity.
\end{theorem}

The~proof is postponed to \cref{sec:proof:thm:poisson-point}.

\begin{remark}
    \label{rem:poisson}
The~Poisson point process \cite[Section 4]{Kingman1993}
\begin{equation}
\label{eq:Poisson-point-process} 
\left( 0< \xi_0 < \xi_1 < \cdots \right) 
\end{equation}
on $\R_+$ can~be viewed concretely as the~sequence of partial sums
\begin{align*} 
\xi_0 &= \psi_0,\\
\xi_1 &= \psi_0+\psi_1, \\
\xi_2 &= \psi_0+\psi_1+\psi_2, \\
\vdots
\end{align*}
for a~sequence $(\psi_i)$ of independent, identically distributed random
variables with the~exponential distribution $\operatorname{Exp}(1)$. Thus a
concrete way to express the~convergence in \cref{thm:poisson-point} is to say
that for each $l\in\N_0$ the~joint distribution of the~\emph{finite} tuple of random
variables
 \[    \left( \frac{2 m}{Y^{[m]}_0},\dots, \frac{2 m}{Y^{[m]}_l} \right) \]
converges, as $m\to\infty$, to the~joint distribution of the~sequence of partial sums
 \[ \left( \psi_0,\;\;\; \psi_0+\psi_1,\;\;\; \dots, \;\;\; \psi_0+\psi_1+\cdots+\psi_l \right).\]
\end{remark}

\begin{corollary}
    \label{cor:first-column} 
    For each $x\in\N_0$ the~random variable $\frac{Y_x^{[m]}}{2m}$
    converges in distribution, as $m\to\infty$, to the~reciprocal of the Erlang distribution $\operatorname{Erlang}(x+1,1)$.
\end{corollary}

In particular, for $x=0$ it follows that the~random variable
$\frac{Y_0^{[m]}}{2m}$ which measures the~(rescaled) number of steps of the
bumping route to reach the~leftmost column converges in
distribution, as $m\to\infty$, to the~Fréchet distribution of shape parameter
$\alpha=1$:
\begin{equation} 
\label{eq:frechet}
\lim_{m \to\infty} \PP\left( \frac{Y_0^{[m]}}{2m} \leq u \right) =
                       e^{- \frac{1}{u} } \qquad \text{for any $u\in\R_+$.} 
\end{equation}
The~Fréchet distribution has a~heavy tail; in particular its first moment is
infinite which provides a~theoretical explanation for a~bad time complexity
of some of our Monte Carlo simulations.

Equivalently, the~random variable $e^{ -\frac{2m}{Y_0}}$ converges in
distribution, as \mbox{$m\to\infty$}, to the~uniform distribution $U(0,1)$ on the~unit
interval. \cref{fig:cdf} presents the~results of Monte Carlo simulations which
illustrate this result.

\subfile{figures-Poisson/CDF/cdf.tex}

\subsection{Projective convention for drawing Young diagrams}
\label{sec:naive-hyperbola}

Usually in order to draw a~Young diagram we use the~French convention and the
$Oxy$ coordinate system, cf.~\cref{fig:tableauFrench}.
For our purposes it will be more convenient to change the~parametrization of
the~coordinate $y$ by setting
\[ z= z(y)=  \frac{2 m}{y}.\]
This convention allows us to show an~infinite number of
rows of a~given tableau on a~finite piece of paper, 
cf.~\cref{fig:tableauProjection}. We will refer to this
way of drawing Young tableaux as \emph{the~projective convention}; it is somewhat
reminiscent of the~\emph{English convention} in the~sense that the~numbers in
the~tableau increase along the~columns \emph{from top to bottom}.

\medskip

\subfile{figures-Poisson/multiple_bumping_routes/multiple_routes-projective.tex}

In the~projective convention the~bumping route can~be seen as the~plot  of the~function
\begin{equation}
    \label{eq:projective-coordinates}
     x^{\projective}_{\tableau,m}(z)= 
     \tableau_{\leftsquigarrow m+\nicefrac{1}{2}}\left(\left\lfloor \frac{2m}{z} \right\rfloor\right)
\qquad \text{for } z\in\R_+ 
\end{equation}
shown on \cref{fig:tableauProjection} as the~thick red line.

\medskip

With these notations \cref{thm:poisson-point} allows the~following convenient
reformulation.

\begin{theorem}[The~main result, reformulated]
    \label{thm:main-poisson} 
    
Let $\tableau$ be a~random infinite standard Young
tableau with the~Plancherel distribution. For $m\to\infty$ the~stochastic
process 
\begin{equation}\label{eq:projective-process}
    \left\{ x^{\projective}_{\tableau,m}(z),\; z>0\right\} 
\end{equation}
converges
in distribution to the~standard Poisson counting process \mbox{$\{ N(z),\; z>0 \}$} 
with the~unit intensity.
\end{theorem}
For an~illustration see \cref{fig:projective}.

\begin{remark}
In \cref{thm:main-poisson} above, the~convergence in distribution for 
stochastic processes is understood as follows: for any \emph{finite} collection
$z_1,\dots,z_l>0$ we claim that the~joint distribution of the~tuple of random
variables
\[ \left( x^{\projective}_{\tableau,m}(z_1), \dots,
 x^{\projective}_{\tableau,m}(z_l)
\right)\]
converges in the~weak topology of probability measures, as $m\to\infty$, to 
the~joint distribution of the~tuple of random variables
\[ \big( N(z_1), \dots, N(z_l) \big).\]
\end{remark}

\begin{proof}[Proof of \cref{thm:main-poisson}]
    The process \eqref{eq:projective-process} is a counting process. 
    By the definition \eqref{eq:projective-coordinates}, the~time of its $k$-th jump (for an integer $k\geq 1$)
    \[ \inf\left\{ z>0:\; x^{\projective}_{\tableau,m}(z) = k \right\} = \frac{2m}{Y_{k-1}^{[m]}} \]
    is directly related to the number of the row in which the bumping route
    reaches the column with the index $k-1$. By \cref{thm:poisson-point} the joint distribution of the times of the jumps converges to the Poisson point process;
    it follows therefore that \eqref{eq:projective-process} converges to the Poisson counting process, as required.
\end{proof}

The~plot of the~mean~value of the~standard Poisson process $ z \mapsto
\mathbb{E} N(z) $ is the~straight line $x=z$ which is shown on
\cref{fig:projective} as the~dashed line. Somewhat surprisingly it coincides
with the~hyperbola~\eqref{eq:hyperbola} shown in the~projective coordinate
system; a~posteriori this gives some justification to the~naive discussion from
\cref{sec:naive}.

\subsection{The~main result with the~right-to-left approach}

\cref{thm:poisson-point} was formulated in a~compact way which may
obscure the~true nature of this result. Our criticism is focused on the
\emph{left-to-right approach} from \cref{rem:poisson} which might give a~false
impression that the~underlying mechanism for generating 
the~random variable~$\frac{2m}{Y_{x+1}^{[m]}}$ 
describing the~`time of arrival' of the~bumping
route to the~column number $x+1$ is based on generating first the~random
variable~$\frac{2m}{Y_x^{[m]}}$ related to the~previous column (that is 
the~column directly to the~left), and adding some `waiting time' for the
transition. In fact, such a~mechanism is not possible without the time travel
because the~chronological order of the events is opposite: the~bumping route first
visits the~column $x+1$ and \emph{then} lands in the~column $x$. In the
following we shall present an~alternative, \emph{right-to-left} viewpoint which
explains better the~true nature of \cref{thm:poisson-point}.

\medskip

For the~Poisson point process \eqref{eq:Poisson-point-process} and an~integer
$l\geq 1$ we consider the~collection of random variables
\begin{equation} 
\label{eq:ratios}
\xi_l, R_0, R_1, \dots, R_{l-1}
\end{equation}
which consists of $\xi_l$ and the~ratios  
\[ R_i := \frac{\xi_{i+1}}{\xi_i}\]
of consecutive entries of $(\xi_i)$.
Then \eqref{eq:ratios} are independent random variables with the~distributions
that can~be found easily. This observation can~be used to \emph{define}
$\xi_0,\dots,\xi_l$ from the~Poisson point process by setting
\[ \xi_i = \xi_l\  \frac{1}{R_{l-1}} \frac{1}{R_{l-2}} \cdots \frac{1}{R_i} 
                                      \qquad \text{for } 0\leq i \leq l.\] 
With this in mind we may reformulate \cref{thm:poisson-point} as follows.

\begin{theorem}[The~main result, reformulated]
    \label{thm:multiplicative} For any integer $l\geq 0$ the~joint distribution
of the~tuple of random variables
\begin{align}
\label{eq:chrono-plus}
& \bigg( \frac{Y_l^{[m]}}{2m}, & \hspace*{2\arraycolsep} &\frac{Y_{l-1}^{[m]}}{2m}, &\hspace*{2\arraycolsep}
 &  \frac{Y_{l-2}^{[m]}}{2m},  & \hspace*{2\arraycolsep} & \dots, & \hspace*{2\arraycolsep} & 
\frac{Y_{0}^{[m]}}{2m} \bigg) \\ 
\intertext{converges, as $m\to\infty$, to the~joint distribution of the~random variables}
\label{eq:funny-products}
& \bigg( \frac{1}{\xi_l}, & & \frac{1}{\xi_l} R_{l-1}, & & \frac{1}{\xi_l} R_{l-1} R_{l-2}, & &
\dots,  & & \frac{1}{\xi_l} R_{l-1}  \cdots R_0 \bigg),
\end{align}
where $\xi_l,R_{l-1},\dots,R_0$ are independent random variables, the
distribution of $\xi_l$ is equal to $\operatorname{Erlang}(l+1,1)$, 
and for each
$i\geq 0$ the~distribution of the~ratio $R_i$ is supported on $[1,\infty)$ with
the~power law
\begin{equation} 
\label{eq:power-law}
\PP(R_i > u) = \frac{1}{u^{i+1}} \qquad \text{for } u\geq 1.
\end{equation}
\end{theorem}

The~order of the~random variables in \eqref{eq:chrono-plus} reflects the
chronological order of the~events, from left to right. Heuristically,
\eqref{eq:funny-products} states that the~transition of the~bumping route from
the column $x+1$ to the column $x$ gives a~\emph{multiplicative} factor~$R_x$ to 
the~total waiting time, with the~factors $R_0,R_1,\dots$ independent.

It is more common in mathematical and physical models that the~total waiting
time for some event arises as a~\emph{sum} of some independent summands, so the
\emph{multiplicative} structure in \cref{thm:multiplicative} comes as a~small
surprise. We believe that this phenomenon can~be explained heuristically as
follows: when we study the~transition of the~bumping route from row $y$ to the
next row $y+1$, the~probability of the~transition from column $x+1$ to column
$x$ seems asymptotically to be equal to
\[ \frac{x+1}{y}+o\left(\frac{1}{y}\right) 
\qquad\text{for fixed value of $x$, and for }y\to\infty.\] 
This kind of decay would explain both the~multiplicative structure (`if 
a~bumping route arrives to a~given column very late, it will stay in this column
even longer') as well as the~power law \eqref{eq:power-law}. We are tempted
therefore to state the~following conjecture which might explain the
aforementioned transition probabilities of the~bumping routes.
\begin{conjecture}
For a~Plancherel-distributed random infinite standard Young tableau $\tableau$ 
\begin{align*}
\PP\left\{ \tableau_{x-1,y+1} < \tableau_{x,y} \right\} &= \frac{x}{y} + o\left(\frac{1}{y}\right) & 
     & \text{for fixed $x\geq 1$ and $y\to\infty$}, 
\\
\PP\left\{ \tableau_{x-2,y+1} < \tableau_{x,y} \right\} &=  o\left(\frac{1}{y}\right) & 
        & \text{for fixed $x\geq 2$ and $y\to\infty$}.  
\end{align*}

Furthermore, for each $x\in\{1,2,\dots\}$ 
the~set of points
\begin{equation} 
\label{eq:mysterious-set}
\left\{ \log \frac{y}{c} : y\in\{c,c+1,\dots\} \text{ and } \tableau_{x-1,y+1} < \tableau_{x,y}  \right\}
\end{equation}
converges, as $c\to\infty$, to Poisson point process on $\R_+$ with the
constant intensity equal to $x$.
\end{conjecture}

Numerical experiments are not conclusive and indicate interesting clustering
phenomena~for the~random set \eqref{eq:mysterious-set}.

\subsection{Asymptotics of fixed $m$}

The~previous results concerned the~bumping routes $\tableau \leftsquigarrow
m+\frac{1}{2}$ in the~limit $m\to\infty$ as the~inserted number tends to
infinity. In the~following we concentrate on another class of asymptotic
problems which concern the~fixed value of $m$.

The~following result shows that \eqref{eq:frechet} gives asymptotically a~very
good approximation for the~distribution tail of $Y_0^{[m]}$ in the
scaling when $m$ is fixed and the~number of the~row $y\to\infty$ tends to
infinity.

\begin{proposition}
\label{prop:m=1}
For each integer $m\geq 1$
\[ \lim_{y\to\infty} y\  \PP\left\{  Y_0^{[m]} \geq y\right\} = 2m. \]
\end{proposition}

This result is illustrated on \cref{fig:cdf} in the~behavior of each of the
cumulative distribution functions in the~neighborhood of $u=1$.
The~proof is postponed to \cref{sec:proof:prop:m=1}.

\begin{question}
What can~we say about the~other columns, that is 
the~tail asymptotics of $\PP\left\{  Y_x^{[m]} \geq y\right\}$ 
for fixed values of $x\in\N_0$ and
$m\geq 1$, in~the~limit $y\to\infty$?
\end{question}

\subsection{More open problems}

Let $\tableau$ be a~random Plancherel-distributed infinite standard Young
tableau. We consider the~\emph{bumping tree} \cite{Duzhin2019}
which is defined as the~collection
of all possible bumping routes for this tableau
\[ \left( \tableau\leftsquigarrow m + \nicefrac{1}{2} \: : \: m\in \N_0 \right) ,\]
which can~be visualized, for example, as on~\cref{fig:all-bumping}. 
Computer simulations suggest that the~set of boxes which
can~be reached by \emph{some} bumping route for a~given tableau $\tableau$ is
relatively `small'. It would be interesting to state this vague observation in
a~meaningful way. We conjecture that the~pictures such as
\cref{fig:all-bumping} which use the~logarithmic scale for the~$y$ coordinate
converge (in the~scaling when $x=O(1)$ is bounded and $y\to\infty$) to some
meaningful particle jump-and-coalescence process.

\subfile{figures-Poisson/zlepianie/zlep-log.tex}

\subsection{Overview of the paper. Sketch of the proof of \cref{thm:poisson-point}}

As we already mentioned, the detailed proof of \cref{thm:poisson-point} is
postponed to \cref{sec:proof:thm:poisson-point}. In the following we present an
overview of the paper and a rough sketch of the proof.

\subsubsection{Trajectory of infinity. Lazy parametrization of the bumping route}
\label{sec:intro-toi}

Without loss of generality we may assume that the Plancherel-distributed infinite
tableau $\tableau$ from the statement of \cref{thm:poisson-point} is of the form $\tableau=Q(\xi_1,\xi_2,\dots)$ for a sequence
$\xi_1,\xi_2,\dots$ of independent, identically distributed random variables with
the uniform distribution $U(0,1)$.

We will iteratively apply Schensted row insertion to 
the~entries of the~infinite sequence
\begin{equation}
    \label{eq:inftyinfty}
     \xi_1,\dots,\xi_m, \infty, \xi_{m+1}, \xi_{m+2}, \dots 
\end{equation}which is the~initial sequence $\xi$ with our favorite symbol $\infty$ inserted at the
position~$m+1$. At step $m+1$ the symbol $\infty$ is inserted at the end of the bottom
row; as further elements of the sequence \eqref{eq:inftyinfty} are inserted, the
symbol~$\infty$ stays put or is being bumped to the next row, higher and higher.

In \cref{prop:lazy-traj-correspondence} we will show that the trajectory of
$\infty$ in this infinite sequence of Schensted row insertions 
\begin{equation}
    \label{eq:inftyrowinsertions}
      \Big( \big( P(\xi_1,\dots,\xi_m, \infty)  \leftarrow \xi_{m+1} \big) \leftarrow  \xi_{m+2} \Big) \leftarrow \cdots 
\end{equation}
coincides with the bumping route $\tableau\leftsquigarrow m + \nicefrac{1}{2} $.
Thus our main problem is equivalent to studying the time evolution of the
position of $\infty$ in the infinite sequence of row insertions
\eqref{eq:inftyrowinsertions}. This time evolution also provides a convenient alternative
parametrization of the bumping route, called \emph{lazy parametrization}. 

\subsubsection{Augmented Young diagrams}
\label{sec:intro-augmented}

\newcommand{\mytableau}[1]{\mathcal{T}^{(#1)}}

For $t\geq m$ we consider the insertion tableau 
\begin{equation}
    \label{eq:mytableau}
    \mytableau{t}=  P(\xi_1,\dots,\xi_m, \infty,\xi_{m+1},\dots,\xi_t) 
\end{equation}
which appears at an intermediate step in \eqref{eq:inftyrowinsertions} after some
finite number of row insertions was performed. By removing the information about
the entries of the tableau $\mytableau{t}$ we obtain the \emph{shape} of $\mytableau{t}$, denoted by
$\operatorname{sh} \mytableau{t}$, which is a Young diagram with $t+1$ boxes. In the
following we will explain how to modify the notion of \emph{the shape of a tableau} so
that it better fits our needs.

Let us remove from the tableau
$\mytableau{t}$ the numbers $\xi_1,\dots,\xi_t$ and let us keep only the information about
the position of the box which contains the symbol~$\infty$. The resulting object, called 
\emph{augmented Young diagram} 
(see \cref{fig:augmented-yd} for an illustration),
can be regarded as a pair $\Lambda^{(t)}=(\lambda,\Box)$ which consists of:
\begin{itemize}
    \item 
the Young diagram $\lambda$ with $t$ boxes which keeps track of the positions of the boxes with the entries $\xi_i, i \in \{1, \dots, t\}$, in $\mytableau{t}$;
\item the outer corner $\Box$ of $\lambda$ which is the position of the box with $\infty$ in~$\mytableau{t}$.
\end{itemize}
We will say that $\operatorname{sh}^* \mytableau{t}=\Lambda^{(t)}$ is the \emph{augmented shape} of $\mytableau{t}$.

The set of augmented Young diagrams, denoted $\Augmented$, has a structure of an
oriented graph which is directly related to Schensted row insertion, as follows.
For a pair of augmented Young diagrams $\Lambda,\widetilde{\Lambda}\in\Augmented$ we say that $\Lambda\nearrow
\widetilde{\Lambda}$ if there exists a tableau $\tableau$ (which contains exactly one
entry equal to $\infty$) such that $\Lambda=\operatorname{sh}^* \tableau$ and there exists some
number $x$ such that $\widetilde{\Lambda}=\operatorname{sh}^* (\tableau\leftarrow
x)$, see \cref{fig:augmented-young-graph} and
\cref{sec:augmented-Young-graph} for more details.

With these notations the time evolution of the position of $\infty$ in the
sequence of row insertions \eqref{eq:inftyrowinsertions} can be extracted from
the sequence of the corresponding augmented shapes
\begin{equation}
    \label{eq:aPgpro}
     \Lambda^{(m)} \nearrow \Lambda^{(m+1)} \nearrow \cdots .
\end{equation}

\subsubsection{Augmented Plancherel growth processes} 

The random sequence~\eqref{eq:aPgpro} is called \emph{the augmented Plancherel
    growth process initiated at time $m$}; in~\cref{sec:aPgp} we will show that it is
a Markov chain with dynamics closely related to the usual (i.e., non-augmented)
Plancherel growth process. 
Since we have a freedom of choosing the value of the integer $m\in\N_0$, 
we get a whole family of augmented Plancherel growth processes. 
It turns out that the transition probabilities for these Markov chains
do not depend on the value of $m$.

Our strategy is to use the Markov property of
augmented Plancherel growth processes combined with the following two pieces of information.
\begin{itemize}
    \item \label{item:component-A}
    \emph{Probability distribution at a given time $t$.} In \cref{prop:distribution-fixed-time} we give an asymptotic
description of the  probability distribution of
$\Lambda^{(t)}$
in the scaling when $m,t\to\infty$ in such a way that $t=\Theta(m^2)$.

    \item \label{item:component-B} \emph{Averaged transition probabilities.} 
		In \cref{prop:transition-augmented} we give an asymptotic description of the
transition probabilities for the augmented Plancherel growth processes
between two moments of time $n$ and $n'$ (with $n<n'$) in the scaling when
$n,n'\to\infty$.

\end{itemize}
Thanks to these results we will prove \cref{thm:lazy-poisson} which gives an
asymptotic description of the probability distribution of the trajectory of the
symbol $\infty$ or, equivalently, the bumping route in the lazy parametrization.

Finally, in \cref{sec:removing-laziness} we explain how to translate this result
to the non-lazy parametrization of the bumping route in which the boxes of the
bumping route are parametrized by the index of the row; this completes the proof
of \cref{thm:poisson-point}.

\medskip

The main difficulty lies in the proofs of the aforementioned 
\cref{prop:distribution-fixed-time} and \cref{prop:transition-augmented};
in the following we sketch their proofs.

\subsubsection{Probability distribution of the augmented Plancherel growth process at a
    given time.}
\label{sec:at-given-time}

In order to prove the aforementioned \cref{prop:distribution-fixed-time} we need
to understand the probability distribution of the augmented shape of the
insertion tableau $\mytableau{t}$ given by \eqref{eq:mytableau} in the scaling
when $m=O  \big( \sqrt{t} \big)$.
Thanks to some symmetries of the RSK algorithm,
the tableau $\mytableau{t}$ is equal to the transpose of the insertion tableau
\begin{equation}
    \label{eq:transposed} 
    P(\underbrace{\xi_t, \xi_{t-1}, \dots,\xi_{m+1}}_{\text{$t-m$ entries}}, \infty,\xi_{m},\dots,\xi_1) 
\end{equation}
which corresponds to the original sequence read backwards.
Since the probability distribution of the sequence $\xi$ is invariant under permutations,
the augmented shape of the tableau \eqref{eq:transposed} can be viewed as
the value at time~$t$ of the augmented Plancherel growth process initiated at time $m':=t-m$.

The remaining difficulty is therefore to understand the probability distribution
of the augmented Plancherel growth process initiated at time $m'$, after
additional $m$ steps of Schensted row insertion were performed. We are interested in the asymptotic setting when
$m'\to\infty$ and the number of additional steps $m=O\big( \sqrt{m'} \big)$ is
relatively small. This is precisely the setting which was considered in our
recent paper about the Poisson limit theorem for the Plancherel growth process
\cite{MMS-Poisson2020v2}. We summarize these results in
\cref{sec:application-thm-poisson}; based on them we
prove in \cref{prop:distribution-fixed-time-A} that the index of the row of the
symbol $\infty$ in the tableau \eqref{eq:transposed} is asymptotically given by
the Poisson distribution.

By taking the transpose of the augmented Young diagrams we recover \cref{prop:distribution-fixed-time}, as desired.

\subsubsection{Averaged transition probabilities}

We will sketch the proof of the aforementioned \cref{prop:transition-augmented}
which concerns an augmented Plancherel growth process 
\begin{equation}
    \label{eq:pgp-atcolumnk}
     \Lambda^{(n)} \nearrow \Lambda^{(n+1)} \nearrow \cdots 
\end{equation}
for which the initial probability distribution at time $n$ is given by
$\Lambda^{(n)}= \big(\lambda^{(n)},\Box^{(n)} \big)$, where $\lambda^{(n)}$ is a random
Young diagram with $n$ boxes distributed (approximately) according to the
Plancherel measure and $\Box^{(n)}$ is its outer corner located in the column
with the fixed index $k$. Our goal is to describe the probability distribution of
this augmented Plancherel growth process at some later time $n'$, asymptotically as $n,n'\to\infty$.

Our first step in this direction is to approximate the probability distribution
of the Markov process \eqref{eq:pgp-atcolumnk} by a certain linear combination
(with real, positive and negative, coefficients) of the probability distributions
of augmented Plancherel growth processes initiated at time $m$. This linear combination is
taken over the values of $m$ which are of order $O\!\left( \sqrt{n} \right)$.
Finding such a linear combination required the results which we discussed above
in \cref{sec:at-given-time}, namely a~good understanding of the probability
distribution at time $n$ of the augmented Plancherel growth process initiated at some specified time $m=O\!\left( \sqrt{n} \right)$.

The probability distribution of $\Lambda^{(n')}$ is then approximately equal to
the aforementioned linear combination of the laws (this time evaluated  at
time~$n'$) of the augmented Plancherel growth processes initiated at some specific
times~$m$. This linear combination is straightforward to analyze because for each
individual summand the results from \cref{sec:at-given-time} are applicable. This
completes the sketch of the proof of \cref{prop:transition-augmented}.

\section{Growth of the~bottom rows}
\label{sec:application-thm-poisson}

In the~current section we will gather some results and some notations from our
recent paper \cite[Section~2]{MMS-Poisson2020v2} which will be necessary for 
the~purposes of the~current work.

\subsection{Total variation distance}

Suppose that $\mu$ and $\nu$ are (signed) measures on the~same discrete set $S$. Such
measures can~be identified with real-valued functions on $S$. 
We define the~\emph{total
    variation distance} between the~measures $\mu$ and $\nu$
\begin{equation} 
\label{eq:TVD1}
\delta( \mu, \nu) := \frac{1}{2} \left\| \mu-\nu \right\|_{\ell^1} 
\end{equation}
as half of their $\ell^1$ distance as functions. 
If $X$ and $Y$ are two random variables with values in the~same discrete set
$S$, we define their total variation distance $\delta(X,Y)$ as the~total
variation distance between their probability distributions
(which are probability measures on $S$).

Usually in the~literature the~total variation distance is defined only for
probability measures. In such a~setup the~total variation distance can~be expressed as
\begin{equation} 
\label{eq:TVD2}
 \delta( \mu, \nu) = \max_{X\subset S} \left| \mu(X) - \nu(X) \right| =
\left\| (\mu-\nu)^+ \right\|_{\ell^1}.  
\end{equation}
In the~current paper we will occasionally use the~notion of the~total variation
distance also for signed measures for which \eqref{eq:TVD1} and \eqref{eq:TVD2}
are \emph{not} equivalent.

\subsection{Growth of rows in Plancherel growth process}

\label{sec:plancherel-growth-process-independent}

Let \mbox{$\lambda^{(0)}\nearrow \lambda^{(1)} \nearrow \cdots$} be the~Plancherel growth
process. For integers $n\geq 1$ and $r\in \N_0$ we denote by $E^{(n)}_r$ the
random event which occurs if the~unique box of 
the skew diagram $\lambda^{(n)} / \lambda^{(n-1)}$ 
is located in the~row with the~index $r$.

The following result was proved by Okounkov \cite[Proposition 2]{Okounkov2000},
see also \cite[Proposition~2.7]{MMS-Poisson2020v2} for an alternative proof.

\begin{proposition}
    \label{prop:asymptotic-probability}
    For each $r\in \N_0$
    \[ 
    \lim_{n\to\infty} \sqrt{n}\ \PP\left(  E_r^{(n)} \right) =1 .\]    
\end{proposition}

\medskip

Let us fix an~integer $k\in\N_0$. We define $\mathcal{N}=\{0,1,\dots,k,\infty\}$.
For $n\geq 1$ we define the~random variable
$R^{(n)}\inrv \mathcal{N}$
which is given by
\[ R^{(n)} = 
\begin{cases}
r      & \text{if the~event $E^{(n)}_r$ occurs for $0\leq r\leq k$},\\
\infty  & \text{if the~event $E^{(n)}_r$ occurs for some $r>k$}.
\end{cases}
\]

\medskip

\newcommand{\nzero}{n}

Let $\ell=\ell(\nzero)$ be a~sequence of non-negative integers such that 
\[ 
\ell=O \! \left(\sqrt{\nzero} \right).
\]
For a~given integer $\nzero\geq (k+1)^2$ we focus on the
specific part of the~Plancherel growth process
\begin{equation}
\label{eq:focus} 
\lambda^{(\nzero)}\nearrow \cdots \nearrow\lambda^{(\nzero+\ell)}.
\end{equation}
We will encode some partial information about the~growths of the~rows as well
as about the~final Young diagram in \eqref{eq:focus}
by the~random vector
\begin{equation}
\label{eq:vn0}
V^{(\nzero)}= 
\left( 
R^{(\nzero+1)}, \: \dots , \: R^{(\nzero+\ell)}, \:
\lambda^{\left( \nzero+\ell\right)} \right) \in \mathcal{N}^{\ell}
\times \diagrams.   
\end{equation}

We also consider the~random vector
\begin{equation}
\label{eq:ovV-def} 
\overline{V}^{(\nzero)}= 
\left( 
\overline{R}^{(\nzero+1)}, \: \dots , \: \overline{R}^{(\nzero+\ell)}, \:
\overline{\lambda}^{\left( \nzero+\ell\right)} \right) \in \mathcal{N}^{\ell}
\times \diagrams   
\end{equation}
which is defined as a~sequence of independent random variables; the~random
variables $\overline{R}^{(\nzero+1)},\dots,\overline{R}^{(\nzero+\ell)}$ have the
same distribution given by
\begin{align*} 
\PP\left\{ \overline{R}^{(\nzero+i)}= r \right\} &= \frac{1}{\sqrt{\nzero}} 
& \text{for }r\in\{0,\dots,k\}, \; 1\leq i\leq \ell,
\\
\PP\left\{ \overline{R}^{(\nzero+i)}= \infty \right\} &= 1- \frac{k+1}{\sqrt{\nzero}} 
\end{align*}
and $\overline{\lambda}^{\left( \nzero+\ell\right)}$ is distributed according to the
Plancherel measure $\Plancherel_{\nzero+\ell}$; in particular the random
variables ${\lambda}^{\left( \nzero+\ell\right)}$ and $\overline{\lambda}^{\left(
    \nzero+\ell\right)}$ have the same distribution.

\medskip

Heuristically, the~following result states that when the Plancherel growth process
is in an~advanced stage and we observe a~relatively small number of its
additional steps, the~growths of the~bottom rows occur approximately like
independent random variables. Additionally, these growths do not affect too
much the~final shape of the~Young diagram.

\begin{theorem}[{\cite[Theorem~2.2]{MMS-Poisson2020v2}}]
\label{thm:independent-explicit} 

With the~above notations, for each fixed $k\in\N_0$ the~total variation
distance between $V^{(\nzero)}$ and $\overline{V}^{(\nzero)}$ converges to
zero, as $\nzero\to\infty$; more specifically
    \[ \delta\left(V^{(\nzero)}, \overline{V}^{(\nzero)}\right) =
             o\left(\frac{\ell}{\sqrt{\nzero}}\right). \]
\end{theorem}

\section{Augmented Plancherel growth process}
\label{sec:augmented-plancherel}

In this section we will introduce our main tool: \emph{the~augmented Plancherel
    growth process} which keeps track of the~position of a~very large number in 
		the~insertion tableau when new random numbers are inserted.

\subsection{Lazy parametrization of bumping routes} 

Our first step towards the~proof of \cref{thm:poisson-point} is to introduce 
a~more convenient parametrization of the~bumping routes. In
\eqref{eq:bumping_points} we used $y$, the~number of the~row, as the~variable
which parametrizes the~bumping route. In the~current section we will introduce
the~\emph{lazy parametrization}.

\medskip

Let us fix a~(finite or infinite) standard Young tableau $\tableau$
and an~integer $m\in\N_0$. For a~given integer $t\geq m$ we denote by
\[ \Box^{\lazy}_{\tableau,m}(t) = 
\left(x^{\lazy}_{\tableau,m}(t), y^{\lazy}_{\tableau,m}(t) \right)\] 
the~coordinates of the~first box in the~bumping route $\tableau \leftsquigarrow
m+\nicefrac{1}{2}$ which contains an~entry of $\tableau$ which is bigger than $t$. 
If such a box does not exists, this means that the bumping route is finite, and  
all boxes of the tableau $\tableau$ which belong to the bumping route are $\leq t$. If this is the case we define 
$ \Box^{\lazy}_{\tableau,m}(t)$ to be the last box of the bumping route, i.e.~the box of the bumping route which lies outside of $\tableau$. 
We will refer to
\begin{equation} 
\label{eq:lazy}
t\mapsto \big(x^{\lazy}_{\tableau,m}(t), y^{\lazy}_{\tableau,m}(t) \big)
\end{equation} 
as the~\emph{lazy parametrization of the~bumping route}.

For example, for the~infinite tableau $\tableau$ from \cref{fig:tableauFrench} 
and $m=3$ the~usual parametrization of the~bumping route is given by
\begin{align*} \tableau_{\leftsquigarrow m+\nicefrac{1}{2}}(y) &= 
\begin{cases}
3 & \text{for } y=0, \\
2 & \text{for } y=1, \\
1 & \text{for } y=2, \\
1 & \text{for } y=3, \\
0 & \text{for } y\geq 4,
\end{cases}
\\
\intertext{while its lazy counterpart is given by}
\Box^{\lazy}_{\tableau,m}(t)=\left(x^{\lazy}_{\tableau,m}(t), y^{\lazy}_{\tableau,m}(t) \right) & = 
\begin{cases}
(3,0)            & \text{for } t\in\{3,4,5\}, \\
(2,1)            & \text{for } t\in\{6,7,8\}, \\
(1,2)            & \text{for } t=9, \\
(1,3)            & \text{for } t\in\{10,11\}, \\
(0,4)            & \text{for } t=12, \\
(0,5)            & \text{for } t=13, \\
(0,6)            & \text{for } t\in\{14,15,16\}, \\
\;\;\;\vdots           & 
\end{cases}
\end{align*}

Clearly, the~set of values of the~function \eqref{eq:lazy} coincides with the
bumping route understood in the~traditional way \eqref{eq:bumping_points}.

We denote by $\tableau|_{\leq t}$ the outcome of keeping only these boxes of 
$\tableau$ which are at most $t$. Note that the element of the bumping route
\begin{equation}
    \label{eq:bumping-explicit}
     \Box^{\lazy}_{\tableau,m}(t) = 
\operatorname{sh} \big( \tableau|_{\leq t} \leftarrow m+\nicefrac{1}{2} \big)\; / \; \operatorname{sh} \left( \tableau|_{\leq t} \right)
\end{equation}
is the unique box of the difference of two Young diagrams on the right-hand side. 

\subsection{Trajectory of $\infty$}
\label{sec:trajectory}

Let $\xi=(\xi_1,\xi_2,\dots)$ be a~sequence of independent, identically distributed
random variables with the~uniform distribution $U(0,1)$ on the~unit interval
$[0,1]$ and let $m\geq 0$ be a~fixed integer. 
We will iteratively apply Schensted row insertion to 
the~entries of the~infinite sequence
\[ \xi_1,\dots,\xi_m, \infty, \xi_{m+1}, \xi_{m+2}, \dots\]
which is the~initial sequence $\xi$ with our favorite symbol $\infty$
inserted at position~$m+1$. 
(The~Readers who are afraid of infinity may replace
it by any number which is strictly bigger than~all of the~entries of the
sequence $\xi$.)
Our goal is to investigate the~position of the~box containing $\infty$ as a
function of the~number of iterations. More specifically, for an~integer $t\geq m$ we
define
\begin{equation}
\label{eq:trajectory} 
\Box^{\traj}_m(t)=
\Pos_\infty \! \big( P(\xi_1,\dots,\xi_m, \infty, \xi_{m+1}, \dots, \xi_t ) \big)
\end{equation}
to be the~position of the~box containing $\infty$ in the~appropriate insertion
tableau.
This problem was formulated by Duzhin \cite{Duzhin2019}; the~first
asymptotic results in the~scaling in which $m\to\infty$ and $t=O(m)$ were found
by \Marciniak~\cite{Marciniak2020}. In the~current paper we go beyond this
scaling and consider \mbox{$m\to\infty$} and $t=O \! \left(m^2\right)$;
the~answer for this problem is essentially contained in 
\cref{thm:lazy-poisson}.

The~following result shows a~direct link between the~above
problem and the~asymptotics of bumping routes. This result also shows an
interesting link between the~papers \cite{RomikSniadyBumping} and
\cite{Marciniak2020}.

\begin{proposition}
    \label{prop:lazy-traj-correspondence}
    Let $\xi_1,\xi_2,\dots$ be a~(non-random or random) sequence and
$\tableau=Q(\xi_1,\xi_2,\dots)$ be the~corresponding recording tableau.
Then for each $m\in\N$ the~bumping route in the~lazy parametrization
coincides with the~trajectory of $\infty$ as defined in
\eqref{eq:trajectory}:
    \begin{equation} 
    \label{eq:link}
    \Box^{\lazy}_{\tableau,m}(t) = \Box^{\traj}_m(t) \qquad \text{for each integer } t\geq m.
    \end{equation}
\end{proposition}
We will provide two proofs of \cref{prop:lazy-traj-correspondence}. The first one
is based on the following classic result of Sch\"utzengerger.

\begin{fact}[{\cite{Schuetzenberger1963}}]
    \label{fact}
    For any permutation $\sigma$ the
                insertion tableau $P(\sigma)$ and the recording
   tableau $Q(\sigma^{-1})$, which corresponds to the inverse of $\sigma$, 
                are equal. 
       
    \end{fact}

\begin{proof}[The first proof of \cref{prop:lazy-traj-correspondence}.]
Let $\pi=(\pi_1,  \dots, \pi_t)\in\Sym{t}$ 
be the \emph{permutation generated by the sequence 
$(\xi_1, \dots,\xi_t)$}, that is the unique
permutation such that for any choice of indices 
$i<j$ the condition $\pi_i<\pi_j$ holds true if 
and only if $\xi_i \leq \xi_j$. 
Let $\pi^{-1}=(\pi_1^{-1}, \dots,\pi^{-1}_t)$ 
be the inverse of $\pi$. 
Since RSK depends only on the relative order of entries, 
the restricted tableau $\tableau|_{\leq t}$ is equal to
 \begin{equation}
\label{eq:Marciniak-equality}
\tableau|_{\leq t} =Q(\xi_1,\dots, \xi_t)=Q(\pi)=P(\pi^{-1}).
\end{equation}

By \eqref{eq:bumping-explicit}, \eqref{eq:Marciniak-equality} and \cref{fact}
it follows that
    \begin{align*}
        \Box^{\lazy}_{\tableau,m}(t)&=
        \operatorname{sh} \left( P(\pi^{-1}) \leftarrow m+\nicefrac{1}{2} \right)\;  /\; \operatorname{sh} P(\pi^{-1}) \\
        &=
        \operatorname{sh} P\left(\pi^{-1}_1,\dots,\pi^{-1}_t,m+\nicefrac{1}{2} \right)
				\;  /\; 
				\operatorname{sh} P\left(\pi^{-1}_1,\dots,\pi^{-1}_t\right) \\
        &=\Pos_{t+1}\left(Q\left(\pi_1^{-1},  \dots, \pi_t^{-1}, m+\nicefrac{1}{2}\right)\right)\\
        &\stackrel{\text{\cref{fact}}}{=}\Pos_{t+1} \left(P(\pi_1,  \dots, \pi_{m}, t+1, \pi_{m+1}, \dots, \pi_t)\right) \\
        &=\Pos_{\infty} \left( P(\xi_1, \dots, \xi_{m}, \infty, \xi_{m+1}, \dots, \xi_t) \right) \\
        &=\Box^{\traj}_m(t)
    \end{align*}
    since the permutation $(\pi_1, \dots, \pi_{m}, t+1, \pi_{m+1},
\dots, \pi_t)$ is the inverse of the permutation generated by the sequence
$(\pi_1^{-1},  \dots, \pi_t^{-1}, m+\nicefrac{1}{2})$.
\end{proof}

The above proof has an advantage of being short and abstract.
The following alternative proof highlights the `dynamic' aspects 
of the bumping routes and the trajectory of infinity.

\begin{proof}[The second proof of \cref{prop:lazy-traj-correspondence}]
        
    We use induction over the~variable $t$.
    
    \medskip
    
The~induction base $t=m$ is quite easy: $\Box^{\lazy}_{\tableau,m}(m)$ is
the~leftmost box in the~bottom row of $\tableau$ which contains a~number which is
bigger than~$m$. This box is the~first to the~right of the~last box in the~bottom
row in the~tableau $Q(\xi_1, \ldots, \xi_m)$. On the other hand, since this
recording tableau has the~same shape as the insertion tableau $P(\xi_1, \ldots,
\xi_m)$, it follows that $\Box^{\traj}_m(m) = \Box^{\lazy}_{\tableau,m}(m)$ and
the proof of the induction base is completed.
    
    \medskip
    
    We start with an~observation that $\infty$ is bumped in the~process of
calculating the~row insertion
    \begin{equation}
    \label{eq:indution}
    P(\xi_1,\dots,\xi_m, \infty, \xi_{m+1}, \dots, \xi_t ) \leftarrow \xi_{t+1} 
    \end{equation}
    if and only if the~position of $\infty$ at time $t$, 
		that is $\Box^{\traj}_m(t)$,
    is the~unique box which belongs to the skew diagram
    \[ 
    \RSK(\xi_1,\dots, \xi_{t+1}) / 
    \RSK(\xi_1,\dots, \xi_t ). \]
    The~latter condition holds true if and only if
    the~entry of $\tableau$ located in the~box
    $ \Box^{\traj}_m(t)$ fulfills 
    \[ \tableau_{\Box^{\traj}_m(t)} = t+1.\]
		
   \medskip
    
    In order to make the~induction step we assume that the~equality
\eqref{eq:link} holds true for some $t\geq m$. There are the~following two
cases.
    
    \smallskip
    
    \emph{Case 1. Assume that 
		the~entry of $\tableau$ located in the~box
        $\Box^{\lazy}_{\tableau,m}(t)$ is strictly bigger than~$t+1$.}
    In this case the~lazy bumping route stays put and 
    \[ \Box^{\lazy}_{\tableau,m}(t+1) = \Box^{\lazy}_{\tableau,m}(t).\]
    
    By the~induction hypothesis, the~entry of $\tableau$ located in the~box $
\Box^{\traj}_m(t)=  \Box^{\lazy}_{\tableau,m}(t)$ is bigger than~$t+1$. By
the~previous discussion, $\infty$ is not bumped in the~process of
calculating the~row insertion \eqref{eq:indution} hence
    \[ \Box^{\traj}_m(t+1)  = \Box^{\traj}_m(t) \]
    and the~inductive step holds true.
    
    \smallskip
    
    \emph{Case 2. Assume that 
    the~entry of $\tableau$ located in the~box $\Box^{\lazy}_{\tableau,m}(t)$ is
    equal to $t+1$.} In this case the~lazy bumping route moves to the~next row.
It~follows that $\Box^{\lazy}_{\tableau,m}(t+1)$ is the~leftmost box of $\tableau$ in the
row above $\Box^{\lazy}_{\tableau,m}(t)$ which contains a~number which is
bigger than~$\tableau_{\Box^{\lazy}_{\tableau,m}(t)} = t+1$.
        
By the~induction hypothesis, $\tableau_{\Box^{\traj}_m(t)} =
\tableau_{\Box^{\lazy}_{\tableau,m}(t)} = t+1$, so $\infty$ is bumped in the
process of calculating the~row insertion \eqref{eq:indution} to the~next row~$r$. 
The~box $\Box^{\traj}_m(t+1)$ is the~first to the~right of the~last box in
the~row~$r$ in $\RSK(\xi_1, \dots, \xi_t, \xi_{t+1})$. Clearly, this is the~box
in the~row $r$ of $\tableau$ which has the~least entry among those which are
bigger than~$t+1$, so it is the~same as~$\Box^{\lazy}_{\tableau,m}(t+1)$. 
\end{proof}

\subsection{Augmented Young diagrams. Augmented shape of a tableau}
\label{sec:augmented-Young-diagrams}

For the motivations and heuristics behind the notion of augmented Young diagrams see \cref{sec:intro-augmented}.

A pair $\Lambda=(\lambda,\Box)$ will be called an \emph{augmented Young diagram} if $\lambda$ is a~Young diagram and $\Box$ is one of its outer corners,
see \cref{fig:tableau-after}. We will say that $\lambda$ is \emph{the regular part of $\Lambda$} and that $\Box$ is \emph{the special box of $\Lambda$}.

The~set of augmented Young diagrams will be denoted 
by $\Augmented$ and for $n\in\N_0$
we will denote by $\Augmented_n$ the set of 
augmented Young diagrams~$(\lambda,\Box)$ with the~additional
property that $\lambda$ has $n$ boxes 
(which we will shortly denote by $|\lambda|=n$).

\medskip

Suppose $\tableau$ is a tableau with the property that exactly one of its entries is equal to $\infty$. We define the \emph{augmented shape of $\tableau$}
\[ \operatorname{sh}^* \tableau = \left( \operatorname{sh} \left( \tableau \setminus \{ \infty\} \right), \Pos_{\infty} \tableau \right) \in \Augmented,\]
as the pair which consists of (a) the shape of $\tableau$ after removal of the box with~$\infty$, and (b) the location of the box with~$\infty$ in~$\tableau$, 
see \cref{fig:augmented-yd}.

\subsection{Augmented Young graph}
\label{sec:augmented-Young-graph}

The~set $\Augmented$ can~be equipped with a~structure of an~oriented graph,
called \emph{augmented Young graph}. 
We declare that a~pair 
$\Lambda\nearrow \widetilde{\Lambda}$ forms an~oriented edge 
(with $\Lambda=(\lambda,\Box)\in\Augmented$ and $\widetilde{\Lambda}=(\widetilde{\lambda},\widetilde{\Box})\in\Augmented$)
if the~following two conditions hold true:
\begin{equation}
\label{eq:augmented-edge} 
\lambda\nearrow\widetilde{\lambda} \ \ \text{and} \ \
\widetilde{\Box} = \begin{cases} 
\text{the~outer corner of $\widetilde{\lambda}$} & \\  \quad \text{which is in the~row above $\Box$}    
                            &  \text{if }\widetilde{\lambda}/\lambda= \left\{ \Box \right\},   \\[1.5ex] 
\Box &  \text{otherwise,}
\end{cases}
\end{equation}
see \cref{fig:augmented-yd} for an~illustration. If $\Lambda\nearrow
\widetilde{\Lambda}$ (with $\Lambda=(\lambda,\Box)\in\Augmented$ and
$\widetilde{\Lambda}=(\widetilde{\lambda},\widetilde{\Box})\in\Augmented$) are
such that $\Box\neq \widetilde{\Box}$ (which corresponds to the first case on the
right-hand side of \eqref{eq:augmented-edge}), we will say that \emph{the
    edge $\Lambda\nearrow\widetilde{\Lambda}$ is a bump}.

\subfile{figures-Poisson/augmented-Young-diagram.tex}

The above definition was specifically tailored so that the following simple lemma holds true.

\begin{lemma}
    \label{lem:why-the-bump}
Assume that $\tableau$ is a tableau which has exactly one entry equal to $\infty$ and let $x$ be some finite number. Then
\[ \left( \operatorname{sh}^* \tableau \right) \nearrow  \left( \operatorname{sh}^* \! \left( \tableau \leftarrow x\right) \right). \]
\end{lemma}
\begin{proof}
Let $\tableau':=\tableau/\{\infty\}$ be the tableau $\tableau$ with the box containing $\infty$ removed.
Denote $(\lambda,\Box)=\operatorname{sh}^* \tableau$ and 
$(\widetilde{\lambda},\widetilde{\Box})=   
\operatorname{sh}^* \! \left( \tableau \leftarrow x\right)$;
their regular parts 
\begin{align*} 
    \lambda &=\operatorname{sh} \tableau' , \\
    \widetilde{\lambda} &=\operatorname{sh} \left( \tableau' \leftarrow x \right) 
\end{align*}
clearly fulfill $\lambda\nearrow\widetilde{\lambda}$.

The position $\widetilde{\Box}$ of $\infty$ in $\tableau\leftarrow x$ is either: 
\begin{itemize}
    \item in the row immediately
above the position $\Box$ of $\infty$ in $\tableau$ (this happens exactly if
$\infty$ was bumped in the insertion $\tableau\leftarrow x$; equivalently if $\widetilde{\lambda}/\lambda=\{
\Box\}$), or
    \item the same as the position $\Box$ of $\infty$ in $\tableau$ (this happens
exactly when $\infty$ was not bumped; equivalently if
$\widetilde{\lambda}/\lambda\neq\{ \Box\}$).
\end{itemize}    
Clearly these two cases correspond to the second condition in
\eqref{eq:augmented-edge} which completes the proof.
\end{proof}

\newcommand{\beg}{m}

\subsection{Lifting of paths}
\label{sec:lifting}

We consider the~\emph{`covering map'} $p:\Augmented\to\diagrams$ given by taking
the regular part
\[ \Augmented \ni (\lambda, \Box) \xmapsto{p} \lambda\in \diagrams. \]

\begin{lemma}
    \label{lem:lifting}
For any $\Lambda^{(\beg)}\in\Augmented$ and 
any path in the~Young graph
\begin{equation} 
\label{eq:path}
\lambda^{(\beg)}\nearrow \lambda^{(\beg+1)} \nearrow \cdots \in \diagrams
\end{equation}
with a~specified initial element $\lambda^{(\beg)}=p\big( \Lambda^{(\beg)} \big) $ 
there exists the~unique \emph{lifted path} 
\[ \Lambda^{(\beg)}\nearrow \Lambda^{(\beg+1)} \nearrow \cdots \in \Augmented \]
in the~augmented Young graph 
with the~specified initial element $\Lambda^{(\beg)}$, and 
such that
$\lambda^{(t)}=p \big( \Lambda^{(t)} \big) $ 
holds true for each $t\in\{\beg,\beg+1,\dots\}$.
\end{lemma}
\begin{proof}
    From \eqref{eq:augmented-edge}  it follows that for each
$(\lambda,\Box)\in\Augmented$ and each $\widetilde{\lambda}$ such that
$\lambda\nearrow \widetilde{\lambda}$ there exists a unique
$\widetilde{\Box}$ such that
$(\lambda,\Box)\nearrow(\widetilde{\lambda},\widetilde{\Box})$. This shows
that, given $\Lambda^{(i)}$, the value of $\Lambda^{(i+1)}$ is determined
uniquely. This observation implies that the lemma can be proved by a
straightforward induction.
\end{proof}

\subfile{figures-Poisson/augmented-young-graph.tex}

\subsection{Augmented Plancherel growth process}
\label{sec:aPgp}

We keep the~notations from the beginning of \cref{sec:trajectory}, i.e., we
assume that $\xi=(\xi_1,\xi_2,\dots)$ is a~sequence of independent, identically
distributed random variables with the~uniform distribution $U(0,1)$ on the~unit
interval $[0,1]$ and $m\geq 0$ is a~fixed integer. We consider a~path in
the~augmented Young graph
\begin{equation}
\label{eq:Markov}
\Lambda_m^{(m)}\nearrow \Lambda_m^{(m+1)} \nearrow \cdots 
\end{equation}
given by
\[ \Lambda_m^{(t)} = \operatorname{sh}^* P(\xi_1,\dots,\xi_m,\infty,\xi_{m+1},\dots,\xi_t) \qquad \text{for any integer $t\geq m$}\]
(\cref{lem:why-the-bump} shows that \eqref{eq:Markov} is indeed a path in $\Augmented$).
We will call \eqref{eq:Markov} \emph{the augmented Plancherel growth process initiated at time~$m$}.
The coordinates of the special box of $\Lambda_m^{(t)}=\big( \lambda^{(t)},\Box_m^{(t)}\big)$  will be denoted by
\[ 
\Box_m^{(t)}=\big( x_m^{(t)},y_m^{(t)} \big).
\]

\begin{theorem}
The augmented Plancherel growth process initiated at time~$m$ is a Markov chain with the transition probabilities given for any $t\geq m$ by
\begin{multline}
    \label{eq:transition-probability-augmented} 
    \PPcond{ \Lambda_m^{(t+1)} =\widetilde{\Lambda} }{ \Lambda_m^{(t)} = \Lambda } =\\
    \begin{cases}
        \PPcond{ \lambda^{(t+1)}=\widetilde{\lambda} }{ \lambda^{(t)}=\lambda~}
        & \text{if }  \Lambda \nearrow  \widetilde{\Lambda}, 
        \\
        0 & \text{otherwise}
    \end{cases}
\end{multline}
for any $\Lambda,\widetilde{\Lambda}\in\Augmented$, where $\lambda$ is the
regular part of $\Lambda$ and $\widetilde{\lambda}$ is the regular part of
$\widetilde{\Lambda}$. These transition probabilities do not depend on the choice
of $m$. The conditional probability on the right-hand side is the transition
probability for the Plancherel growth process
$\lambda^{(0)}\nearrow\lambda^{(1)}\nearrow\cdots$.
\end{theorem}
\begin{proof}
The~path \eqref{eq:Markov}
is the unique lifting (cf.~\cref{lem:lifting}) of the~sequence of the~regular parts
\begin{equation}
    \label{eq:lambdas}
    \lambda^{(m)}\nearrow \lambda^{(m+1)} \nearrow \cdots 
\end{equation}
with the~initial condition that the special box $\Box_m^{(m)}$ is the outer
corner of $\lambda^{(m)}$ which is located in the bottom row. 
It follows that for any augmented Young diagrams $\Sigma_{m},\dots,\Sigma_{t+1}\in\Augmented$ with the regular parts $\sigma_m,\dots,\sigma_{t+1}\in \diagrams$
\begin{multline}
    \label{eq:probability-prefix}
   \PP\left( \Lambda_m^{(m)} = \Sigma_{m}, \; \dots, \; \Lambda_m^{(t+1)}= \Sigma_{t+1} \right) = \\
  = \begin{cases}
       \PP\left( \lambda^{(m)} = \sigma_m, \; \dots ,\; \lambda^{(t+1)}=\sigma_{t+1} \right) & \text{if } \Sigma_{m}\nearrow \cdots \nearrow \Sigma_{t+1} \text{ and} \\ & \text{\quad the special box of $\Sigma_m$} \\ & \quad \text{is in the bottom row}, \\[1.5ex]
       0 & \text{otherwise}. 
   \end{cases}
\end{multline}

The sequence of the regular parts \eqref{eq:lambdas}
forms the~usual Plancherel growth process (with the~first
$m$ entries truncated) hence it is a~Markov chain (the proof that the usual
Plancherel growth process is a Markov chain can be found in \cite[Sections~2.2
and 2.4]{Kerov1999}). 
It follows that the probability on the top of the
right-hand side of \eqref{eq:probability-prefix} can be written in the product
form in terms of the probability distribution of $\lambda^{(m)}$ and the
transition probabilities for the Plancherel growth process.

We compare \eqref{eq:probability-prefix} with its counterpart for $t:=t-1$;
this shows that the conditional probability 
\[\PPcond{\Lambda_{m}^{(t+1)}=\Sigma_{t+1}}{\Lambda_m^{(m)} = \Sigma_{m}, \; \dots, \; \Lambda_m^{(t)}= \Sigma_{t}} \]
is equal to the right-hand side of \eqref{eq:transition-probability-augmented} for $\widetilde{\Lambda}:=\Sigma_{t+1}$ and $\Lambda:=\Sigma_t$. 
In particular, this conditional probability does not depend on the values of $\Sigma_{m},\dots,\Sigma_{t-1}$ and the Markov property follows.
\end{proof}

The special box in the augmented Plancherel growth process can be thought of as a
\emph{test particle} which provides some information about the local behavior of the
usual Plancherel growth process. From this point of view it is reminiscent of the
\emph{second class particle} in the theory of interacting particle systems or
\emph{jeu de taquin trajectory} for infinite tableaux \cite{RomikSniady-AnnPro}.

\subsection{Probability distribution of the~augmented Plancherel growth process}

\cref{prop:distribution-fixed-time-A} and \cref{prop:distribution-fixed-time}
below provide information about the~probability distribution of the~augmented
Plancherel growth process at time $t$ for $t\to\infty$ in two distinct
asymptotic regimes: very soon after the~augmented Plancherel process was
initiated (that is when 
$t=m+O(\sqrt{m})$, cf.~\cref{prop:distribution-fixed-time-A}) and after a
very long time after the~augmented Plancherel process was initiated 
(that is~when $t=\Theta(m^2) \gg m$, cf.~\cref{prop:distribution-fixed-time}).

\begin{proposition}
    \label{prop:distribution-fixed-time-A} Let $z>0$ be a~fixed positive number and
let $t=t(m)$ be a~sequence of positive integers such that $t(m) \geq m$ and 
with the~property that 
    \[ \lim_{m\to\infty} \frac{t-m}{\sqrt{t}} = z. \]
Let $\Lambda_m^{(m)}\nearrow \Lambda_m^{(m+1)} \nearrow \cdots$ 
be the augmented Plancherel growth process initiated at time~$m$. 
We denote $\Lambda_m^{(t)}=\big(\lambda^{(t)}, \Box_m^{(t)}\big)$; 
let $\Box_m^{(t)}=\big(x_m^{(t)}, y_m^{(t)}\big)$ be the
coordinates of the special box of $\Lambda_m^{(t)}$. 
    
\begin{enumerate}[label=\emph{\alph*}),itemsep=1ex]
    \item \label{item:distribution-fixed-time-A-a} 
    The~probability distribution
of $y_{m}^{(t)}$ converges, as $m\to\infty$, to the Poisson distribution $\operatorname{Pois}(z)$
with parameter $z$.

    \item \label{item:distribution-fixed-time-A-b}
    For each $k\in\N_0$ the~total
    variation distance between
    \begin{itemize}
        \item the~conditional probability distribution of $\lambda^{(t)}$ under
the~condition that  $y_{m}^{(t)}=k$, and \item the~Plancherel measure
$\Plancherel_t$
    \end{itemize}
    converges to $0$, as $m\to\infty$.

    \item \label{item:distribution-fixed-time-A-c}
    The~total variation distance between
    \begin{itemize}
        \item 
        the~probability distribution of the~random vector  
        \begin{equation} 
        \label{eq:ymlambda}
        \left(   \lambda^{(t)}, y_{m}^{(t)}    \right) \in  \diagrams \times \N_0
        \end{equation}
        and 
        \item 
        the~product measure
        \[  \Plancherel_t  \times \operatorname{Pois}(z) \]   
    \end{itemize}
    converges to $0$, as $m\to\infty$. 
\end{enumerate}
\end{proposition}

Let us fix an~integer $k\geq 0$. 
We use the~notations from
\cref{sec:plancherel-growth-process-independent} for $\nzero:=m$ and $\ell=t-m$ so
that $\nzero+\ell=t$; we assume that $m$ is big enough so that $m\geq (k+1)^2$. Our general strategy is to read the~required information
from the~vector $V^{(\nzero)}$ given by \eqref{eq:vn0} and to apply
\cref{thm:independent-explicit}. Before the proof of \cref{prop:distribution-fixed-time-A} we start with the following auxiliary result.

\medskip

For $s\geq m$ we define the random variable $y_{m}^{(s)}\!\downarrow_\mathcal{N} \: \inrv \! \mathcal{N}$ by
\begin{align*} y_{m}^{(s)}\!\downarrow_\mathcal{N} &=
\begin{cases}
    y_{m}^{(s)} & \text{if }  y_{m}^{(s)}\in\{0,1,\dots,k\}, \\
    \infty    & \text{otherwise.}
\end{cases}\\
\intertext{We also define the random variable $F_{s}\inrv\{0,\dots,k\}$  by}
F_{s} &= 
\begin{cases}
y_{m}^{(s)}\!\downarrow_\mathcal{N} & \text{if } y_{m}^{(s)}\!\downarrow_\mathcal{N} \neq \infty,\\
\text{arbitrary element of $\{0,\dots,k\}$} & \text{otherwise}.  
\end{cases}
\end{align*}

\begin{lemma}
    \label{lem:past}
For each $s\geq m$ the value of $F_s$ can~be
expressed as an~explicit function of the~entries of the~sequence $R$ related to the~past, that is
\[    
R^{(m+1)}, \: \dots, \: R^{(s)}.  
\]    

\smallskip

For any integer $p\in\{0,\dots,k\}$
the equality $y_m^{(s)}= p$ holds true if and only if 
there are exactly $p$ values of the~index 
$u\in\{m+1,\dots,s\}$ with the~property that 
\begin{equation} 
    \label{eq:Bernoulli}
    R^{(u)}= F_{u-1}.
\end{equation}

The inequality $y_m^{(s)}>k$ holds if and only if there are at least $k+1$ values of the index \mbox{$u
\in\{m+1,\dots,s\}$} with this property.
\end{lemma}
\begin{proof}
There are \emph{exactly} $y_m^{(s)}$ edges which are bumps in the path
\begin{equation}\label{eq:golden-path}
     \Lambda^{(m)} \nearrow \cdots \nearrow \Lambda^{(s)} 
\end{equation}
because each bump increases the $y$-coordinate of the special box by $1$.
Note that an edge $\Lambda^{(u-1)}\nearrow \Lambda^{(u)}$ in this path is a bump if and only if 
\begin{equation}\label{eq:event}
    \text{the event $E^{(u)}_r$ occurs for the row $r=y_m^{(u-1)}$.}
\end{equation}

If $y_m^{(s)}\leq k$ then for any $u\in\{m+1,\dots,s\}$ the equality $F_u=y_m^{(u)}$
holds true; furthermore the event \eqref{eq:event} occurs if and only if
$R^{(u)}=y_m^{(u-1)}$. It~follows that there are \emph{exactly} $y_m^{(s)}$ values of the~index \mbox{$u
    \in\{m+1,\dots,s\}$} such that 
    \eqref{eq:Bernoulli} holds true.

On the~other hand, if $y_m^{(s)}>k$ we can apply the above reasoning to the
truncation of the path \eqref{eq:golden-path} until after the $(k+1)$-st bump occurs. It follows that in this case there are \emph{at least} $k+1$
values of the~index $u\in\{m+1,\dots,s\}$ with the~property~\eqref{eq:Bernoulli}.
In this way we proved the second part of the lemma.

\medskip

By the second part of the lemma, the~value of $y_m^{(s)}\!\downarrow_\mathcal{N}$ can~be
expressed as an~explicit function of both (i) the~previous values
\begin{equation}
\label{eq:previous}
y_m^{(m)}\!\downarrow_\mathcal{N},\: \dots, \: y_m^{(s-1)}\!\downarrow_\mathcal{N},
\end{equation}
\emph{and} (ii) the~entries of the~sequence $R$ related to the~past, that is
\begin{equation}
\label{eq:previousR}
R^{(m+1)}, \: \dots, \: R^{(s)}. 
\end{equation}
By iteratively applying this observation to the~previous values
\eqref{eq:previous} it is possible to express the~value of 
$y_m^{(s)}\!\downarrow_\mathcal{N}$ \emph{purely}
in terms of \eqref{eq:previousR}.
Also the~value of 
\[F_{s} = F_{s}\left(
R^{(m+1)}, \: \dots, \:  R^{(s)}\right)\]
can~be expressed as a~function of the~entries of the~sequence $R$ related to~the~past, as required. 
\end{proof}		
 
 \newcommand{\event}{A}
        
\begin{proof}[Proof of \cref{prop:distribution-fixed-time-A}]
\cref{lem:past} shows that the~event $y_{m}^{(t)}=k$ can~be expressed in
terms of the~vector $V^{(\nzero)}$ given by \eqref{eq:vn0}. We apply
\cref{thm:independent-explicit}; it follows that the~probability $\PP\left\{
y_{m}^{(t)}=k \right\}$ is equal, up to an~additive error term $o(1)$, 
to~the~probability that there are exactly
$k$ values of the~index \mbox{$u\in\{m+1,\dots,t\}$} with the~property that
\begin{equation} 
\label{eq:Bernoulli2}
\overline{R}^{(u)}= F_{u-1}\left(
\overline{R}^{(m+1)},\dots,\overline{R}^{(u-1)}\right).
\end{equation}
We denote by $\event_u$ the random event that the equality \eqref{eq:Bernoulli2} holds true.

\medskip

Let $i_1<\dots<i_{l}$ be an increasing sequence of integers from the set $\{m+1, \dots, t\}$ for $l\geq 1$.
We will show that
\begin{equation}
    \label{eq:independence-day}
     \PP\left( \event_{i_1} \cap \cdots \cap \event_{i_{l}} \right) =  \frac{1}{\sqrt{m}}\ \PP\left( \event_{i_1} \cap \cdots \cap \event_{i_{l-1}} \right).
\end{equation}
Indeed, by \cref{lem:past}, the event $\event_{i_1} \cap \cdots \cap \event_{i_{l-1}}$ is a disjoint finite union of some random events
of the form 
\[ B_{r_{m+1},\dots,r_{j}}=
    \left\{  \overline{R}^{(m+1)}=r_{m+1},\; \overline{R}^{(m+2)}=r_{m+2},\; \dots ,\; \overline{R}^{(j)}=r_{j} \right\} \]
over some choices of $r_{m+1},r_{m+2},\dots,r_{j}\in \mathcal{N}$, where $j:=i_l-1$.
Since the random variables $\big( \overline{R}^{(i)} \big)$ are independent, it follows that
\begin{equation}
\label{eq:problem0101} 
\PP \left( B_{r_{m+1},\dots,r_{j}} \cap \event_{i_l} \right) = \frac{1}{\sqrt{m}}\ \PP \left( B_{r_{m+1},\dots,r_{j}}  \right).
\end{equation}
By summing over the appropriate values of $r_{m+1},\dots,r_{j}\in
\mathcal{N}$ the equality~\eqref{eq:independence-day} follows.

By iterating \eqref{eq:independence-day} it follows that the events $\event_{m+1},\dots, \event_{t}$ are independent
and each has equal probability $\frac{1}{\sqrt{m}}$.

\medskip

By the Poisson limit theorem \cite[Theorem~3.6.1]{Durrett2010}
the~probability of $k$ successes in $\ell$ Bernoulli
trials as above converges to the~probability of the~atom $k$ in the Poisson
distribution with the~intensity parameter equal to
\[ \lim_{m \to\infty} \frac{\ell}{\sqrt{m}} = 
\lim_{m \to\infty} \frac{t-m}{\sqrt{t}}=
z\]
which concludes
the~proof of part \ref{item:distribution-fixed-time-A-a}.

\medskip

The~above discussion also shows that the~conditional probability distribution
considered in point \ref{item:distribution-fixed-time-A-b} is equal to the
conditional probability distribution of the~last coordinate $\lambda^{(t)}$ of
the~vector $V^{(\nzero)}$ under certain condition which is expressed in terms of
the~coordinates $R^{(m+1)},\dots,R^{(t)}$. 
By~\cref{thm:independent-explicit} this conditional probability distribution is in
the~distance $o(1)$ (with respect to the~total variation distance) to its
counterpart for the~random vector $\overline{V}^{(\nzero)}$. The~latter
conditional probability distribution, due to the~independence of the
coordinates of $\overline{V}^{(\nzero)}$, is equal to the~Plancherel measure
$\Plancherel_t$, which concludes the~proof
of~\ref{item:distribution-fixed-time-A-b}.

\medskip

Part \ref{item:distribution-fixed-time-A-c} is a~direct consequence of parts 
\ref{item:distribution-fixed-time-A-a} and
\ref{item:distribution-fixed-time-A-b}. 
\end{proof}

For an augmented Young diagram $\Lambda=\big(\lambda, (x,y) \big)$ we define its
\emph{transpose} $\Lambda^T=\big( \lambda^T, (y,x) \big)$.
\begin{lemma}
    \label{lem:transpose}
    For any integers $m,m'\geq 0$ the probability distributions at time $t=m+m'$ of the augmented Plancherel growth processes initiated at times $m$ and $m'$ respectively are related by
    \[ \Lambda_m^{(t)} \overset{d}{=} \left[ \Lambda_{m'}^{(t)}  \right]^T. \]
\end{lemma}
\begin{proof}
    Without loss of generality we may assume that the random variables
$\xi_1,\dots,\xi_t$ are \emph{distinct} real numbers. An application of
Greene's theorem \cite[Theorem~3.1]{Greene1974} shows that the insertion tableaux
which correspond to a given sequence of distinct numbers and this sequence
read backwards
   \begin{multline*} P(\xi_1,\dots,\xi_m,\infty,\xi_{m+1},\dots,\xi_t) = \\
\left[ P( \xi_t,\xi_{t-1},\dots,\xi_{m+1}, \infty, \xi_{m},\xi_{m-1}, \dots, \xi_1 )   \right]^T 
\end{multline*}
are transposes of one another. It follows that also the augmented shapes are transposes of one another:
\begin{multline*} 
\Lambda_m^{(t)}=    
    \operatorname{sh}^* P(\xi_1,\dots,\xi_m,\infty,\xi_{m+1},\dots,\xi_t) = \\
\Big[ \operatorname{sh}^*  P( \underbrace{ \xi_t,\xi_{t-1},\dots,\xi_{m+1}}_{\text{$m'$ entries}}, \infty, \xi_{m},\xi_{m-1}, \dots, \xi_1 )     \Big]^T. 
\end{multline*} 
Since the sequence $(\xi_i)$ and its any permutation $\left(\xi_{\sigma(i)} \right)$ have the same distributions, the right-hand side has
the same probability distribution as $\left[ \Lambda_{m'}^{(t)}  \right]^T$, as
required.
\end{proof}

\begin{proposition}
    \label{prop:distribution-fixed-time}
    Let $z>0$ be a~fixed real number. 
Let $t=t(m)$ be a~sequence of positive integers such that $t(m) \geq m$ and
with the~property that
\[ \lim_{m\to\infty} \frac{m}{\sqrt{t}} = z. \]
Let $\Lambda_m^{(m)}\nearrow \Lambda_m^{(m+1)} \nearrow \cdots$ be the augmented Plancherel growth process initiated at time~$m$. 
We denote $\Lambda_m^{(t)}=\big(\lambda^{(t)},
\Box_m^{(t)}\big)$; let $\Box_m^{(t)}=\big(x_m^{(t)}, y_m^{(t)}\big)$ be the
coordinates of the special box at time $t$.
    
    The~total variation distance between
   \begin{itemize}
    \item 
   the~probability distribution of the~random vector  
    \begin{equation} 
    \label{eq:xmlambda}
    \left(  x_m^{(t)}, \lambda^{(t)}     \right) \in \N_0\times \diagrams 
    \end{equation}
    and 
    \item 
     the~product measure
			\[ \operatorname{Pois}(z) \times \Plancherel_t   \]   
     \end{itemize}
 converges to $0$, as $m\to\infty$. 
\end{proposition}

\begin{proof}
By \cref{lem:transpose} the probability distribution of \eqref{eq:xmlambda} coincides with the probability distribution of \begin{equation}
    \label{eq:magictranspose}
    \left( y_{m'}^{(t)}, \big[ \lambda^{(t)} \big]^T \right)
\end{equation} for $m':=t-m$.
The random vector \eqref{eq:magictranspose} can be viewed as the image of the vector $ \big( y_{m'}^{(t)},  \lambda^{(t)}  \big)$ under the bijection 
\[\operatorname{id}\times T: (y,\lambda)\mapsto (y,\lambda^T).\]
By \cref{prop:distribution-fixed-time-A} it follows that
the total variation distance 
between \eqref{eq:magictranspose} and the push-forward measure
\[ \left(\operatorname{id}\times T \right)\left( \operatorname{Pois}(z) \times \Plancherel_t \right) =  \operatorname{Pois}(z) \times \Plancherel_t  \]
converges to zero as $m\to\infty$;
the last equality holds since the Plancherel measure is invariant under transposition.
 \end{proof}

\subsection{Lazy version of \cref{prop:m=1}. Proof of \cref{prop:is-finite}}
\label{sec:proof-prop:is-finite}

In \cref{sec:in-which-row} we parametrized the~shape of the bumping route by
the~sequence $Y_0,Y_1,\dots$ which gives \emph{the~number of the~row} in which the
bumping route reaches a~specified column, cf.~\eqref{eq:Y}. 
With the help of \cref{prop:lazy-traj-correspondence} we can define 
the~lazy counterpart of these quantities:
for $x,m\in\N_0$ we denote by 
\[ T_x^{[m]}=T_x= \min\left\{ t : x_m^{(t)}\leq x \right\} \]
the~\emph{time} it takes for the~bumping route (in the~lazy parametrization) to
reach the~specified column.

The~following result is the~lazy version of \cref{prop:m=1}.
\begin{lemma}
    \label{lem:lazy:m=1}
    For each integer $m\geq 1$
    \[ \lim_{u\to\infty} 
    \sqrt{u} \  \PP\left\{ T_0^{[m]} > u\right\}  = m. 
    \]   
\end{lemma}

\begin{proof}
	By \cref{lem:transpose}, for any $u\in \N_0$
	  \[ \PP\left\{ T_0^{[m]} > u\right\}= 
		\PP \left\{ x_m^{(u)} \geq 1 \right\} = 
		   \PP\left\{ y^{(u)}_{u-m} \geq 1 \right\}.\]
   
\medskip
   
	In the~special case $m=1$ the~proof is particularly easy: the~right-hand
side is equal to $\PP \! \left( E^{(u)}_0 \right)$ and
\cref{prop:asymptotic-probability} provides the~necessary asymptotics.
    
\medskip

    For the~general case $m\geq 1$ we use the~notations from
\cref{sec:plancherel-growth-process-independent} for $k=0$, and $\nzero=u-m$,
and $\ell=m$. The~event $y^{(u)}_{u-m} \geq 1$ occurs if and only if at
least one of the~numbers
$R^{(\nzero+1)},\dots,R^{(\nzero+\ell)}$ is equal to $0$.
We apply \cref{thm:independent-explicit}; it follows that the~probability
of the~latter event is equal, up to an~additive error term of the~order
$o \! \left(\frac{m}{\sqrt{u-m}}\right) = o \! \left(\frac{1}{\sqrt{u}}\right)$,
to the~probability that in $m$~Bernoulli
trials with success probability $\frac{1}{\sqrt{\nzero}}$ there is at least
one success. In~this way we proved that
\[  \PP\left\{ y^{(u)}_{u-m} \geq 1 \right\} = 
       \frac{m}{\sqrt{u}} + o\left( \frac{1}{\sqrt{u}} \right),\]
as desired.
\end{proof}

\begin{proof}[Proof of \cref{prop:is-finite}]
Since $Y_0^{[m]}\geq Y_1^{[m]}\geq \cdots $ is a~weakly decreasing sequence, it is enough
to consider the~case $x=0$. We apply \cref{lem:lazy:m=1} in the~limit
$u\to\infty$. It follows that the~probability that the~bumping route
$\tableau\leftsquigarrow m+\nicefrac{1}{2}$ does not reach the~column with the
index $0$ is equal to
  \[ \lim_{u \to\infty}  \PP\left\{ T_0^{[m]} \geq u\right\}  = 0, 
\]     
as required.
\end{proof}

\section{Transition probabilities for the~augmented Plancherel growth process}

Our main result in this section is \cref{thm:lazy-poisson}. It will be the~key
tool for proving the~main results of the~current paper.

\subsection{Approximating Bernoulli distributions by linear combinations of
    Poisson distributions}

\newcommand{\measurePoisson}{\nu}

The~following \cref{lem:poisson-binom} is a~technical result which will be
necessary later in the~proof of \cref{prop:transition-augmented}. Roughly
speaking, it gives a~positive answer to the~following question: \emph{for a
    given value of $k\in\N_0$, can the point measure $\delta_k$ be approximated by
    a~linear combination of the Poisson distributions in some explicit, constructive
    way?} 
		A~naive approach to this problem would be to consider a~scalar multiple
of the Poisson distribution $e^z \operatorname{Pois}(z)$ which corresponds to the
sequence of weights
\[  \N_0\ni m \mapsto \frac{1}{m!} z^m\]
and then to consider its $k$-th derivative with respect to the~parameter $z$
for~\mbox{$z=0$}. This is not exactly a~solution to the~original question (the
derivative is not a~linear combination), but since the~derivative can~be
approximated by the~forward difference operator, this naive approach gives a
hint that an~expression such as \eqref{eq:defmukph} in the~special case $p=1$
might be, in fact, a~good answer.

\begin{lemma}
    \label{lem:poisson-binom}
Let us fix an~integer $k\geq 0$ and a~real number $0\leq p\leq 1$. 
For~each~$h>0$ the~linear combination of the Poisson distributions
\begin{equation} 
\label{eq:defmukph}
\measurePoisson_{k,p,h}:=  \frac{1}{\left( e^h -1 \right)^k} \sum_{0\leq j \leq k} 
(-1)^{k-j} \binom{k}{j} e^{j h} \operatorname{Pois}(p j h)
\end{equation}
is a~probability measure on $\N_0$.

As $h\to 0$, the~measure $\measurePoisson_{k,p,h}$ converges (in the~sense of total
variation distance) to the~binomial  distribution $\operatorname{Binom}(k,p)$.
\end{lemma}

\begin{proof}

\emph{The~special case $p=1$.}
For a~function $f$ on the~real line 
we consider its \emph{forward difference} function $\Delta[f]$ 
 given by
\[ \Delta[f] (x) = f(x+1)- f(x).\]
It follows that the~iterated forward difference is given by
\[ \Delta^k[f] (x) = \sum_{0\leq j\leq k} (-1)^j
\binom{k}{j} 
 f(x+k-j).\]
  
A~priori, $\measurePoisson_{k,1,h}$ is a~signed measure with the~total mass equal to
\begin{equation}  
\label{eq:total-mass}
\frac{1}{\left( e^h -1 \right)^k} \sum_{0\leq j \leq k} 
(-1)^{k-j} \binom{k}{j} e^{j h} = 
 \frac{1}{\left( e^h -1 \right)^k} \Delta^k \left[e^{hx}\right](0).
\end{equation}
The~right-hand side of \eqref{eq:total-mass} is equal to $1$,
since the~forward difference of an~exponential function is again an~exponential:
\[ 
\Delta \left[e^{hx}\right] = \left( e^h -1 \right) e^{h x}. 
\]

The~atom of $\measurePoisson_{k,1,h}$ at an~integer $m\geq 0$ is equal to
\begin{multline*} 
\measurePoisson_{k,1,h}(m)= 
    \frac{1}{\left( e^h -1 \right)^k m!} \sum_{0\leq j \leq k} (-1)^{k-j} \binom{k}{j} (j h)^m = \\
\frac{h^m}{\left( e^h -1 \right)^k m!} \Delta^k \left[ x^m \right](0).
\end{multline*}
Note that the~monomial $x^m$ can~be expressed in terms of the~falling
factorials~$x^{\underline{p}}$ with the~coefficients given by the Stirling numbers
of the~second kind:
\[ x^m = \sum_{0\leq p\leq m} \stirling{m}{p} x^{\underline{p}} \: , \]
hence
\[ \Delta^k \left[ x^m\right] = \sum_{p} \stirling{m}{p} \Delta^k \left[ x^{\underline{p}} \right] =
\sum_{p \geq k} \stirling{m}{p} p^{\underline{k}}\  x^{\underline{p-k}}.
\]
When we evaluate the~above expression at $x=0$, there is only one non-zero
summand
\[ \Delta^k \left[ x^m\right] (0) =
 \stirling{m}{k} k!.
\]
Thus 
\[
\measurePoisson_{k,1,h}(m)= 
\frac{h^m k!}{\left( e^h -1 \right)^k m!} 
\stirling{m}{k}\geq 0,\]
and the~above expression is non-zero only for $m\geq k$.
All in all, $\measurePoisson_{k,1,h}$ is a~probability measure on $\N_0$, as required. 

It follows that the~total variation distance between $\operatorname{Binom}(k,1)=\delta_k$ and
$\measurePoisson_{k,1,h}$ is equal to
\[  \sum_{i=0}^\infty \left[ \delta_k(i) - \measurePoisson_{k,1,h}(i) \right]^+ = 
1 - \measurePoisson_{k,1,h}(k) =
1 - \frac{h^k}{\left( e^h -1 \right)^k} \xrightarrow{h\to 0} 0, \]
as required.

\medskip

\emph{The~general case.}
For a~signed measure $\mu$ which is supported on $\N_0$ and $0\leq p\leq 1$ we define
the~signed measure $C_p[\mu]$ on $\N_0$ by
\[ C_p[\mu](k) = \sum_{j\geq k} \mu(j) \binom{j}{k} p^k (1-p)^{j-k}.\]
In the~case when $\mu$ is a~probability measure, $C_p[\mu]$ has a~natural
interpretation as the~probability distribution of a~compound binomial random
variable $\operatorname{Binom}(M,p)$, where $M$ is a~random variable with
the probability distribution given by $\mu$.
    
It is easy to check that for any $0\leq q \leq 1$ the~image of a~binomial distribution
\[ C_p[ \operatorname{Binom}(n,q) ] = \operatorname{Binom}(n, pq) \]    
is again a~binomial distribution, and for any $\lambda\geq 0$ the~image of a
Poisson distribution
\[ C_p [ \operatorname{Pois}(\lambda) ] = \operatorname{Pois}(p\lambda)\]
is again a~Poisson distribution.
Since $C_p$ is a~linear map, by the~very definition \eqref{eq:defmukph} it
follows that
\begin{equation} 
\label{eq:cp}
C_p [ \measurePoisson_{k,1,h} ] = \measurePoisson_{k,p,h}; 
\end{equation}
in particular the~latter is a~probability measure, as required. By considering
the~limit $h\to 0$ of \eqref{eq:cp} we get
\[ \lim_{h\to 0} \measurePoisson_{k,p,h} 
= C_p \left[ \operatorname{Binom}(k,1) \right] 
= \operatorname{Binom}(k,p)
\]
in the~sense of~total variation distance, 
as required.
\end{proof}

\subsection{The inclusion $\Augmented\subset  \N_0 \times \diagrams$ }
\label{sec:inclusion-x}

We will extend the meaning of the
notations from \cref{sec:augmented-Young-diagrams} to a larger set. 
The~map    
\begin{equation} 
    \label{eq:augmented-x}
    \Augmented\ni (\lambda,\Box) \mapsto (x_\Box, \lambda~) \in \N_0 \times \diagrams,  
\end{equation}
where $\Box=(x_\Box,y_\Box)$, allows us to identify $\Augmented$ with a~subset of
$\N_0 \times \diagrams$. 
For a pair $(x,\lambda)\in\N_0\times\diagrams$ 
we will say that $\lambda$ is its \emph{regular part}.

We define the~edges in this larger set \mbox{$\N_0\times \diagrams\supset \Augmented$} as follows:
we declare that $(x,\lambda)\nearrow (\widetilde{x},\widetilde{\lambda})$
if the~following two conditions hold true:
\begin{equation}
    \label{eq:augmented-edge-B} 
    \lambda\nearrow\widetilde{\lambda} \quad \text{and} \quad    
    \widetilde{x} = \begin{cases} 
        \max\left\{  \widetilde{\lambda}_i : 
        \widetilde{\lambda}_i \leq  x  \right\}   & \text{if the~unique box of $\widetilde{\lambda}/\lambda$} \\[-1ex] 
        & \quad \text{is located in the column $x$},\\[1ex]
        x   & \text{otherwise.}  
    \end{cases}
\end{equation}
In this way the oriented graph $\Augmented$ is a subgraph of $\N_0\times\diagrams$.

An~analogous lifting
property as in \cref{lem:lifting} remains valid if we assume that the~initial element $\Lambda^{(\beg)}\in\N_0\times \diagrams$
and the~elements of the~lifted path
\[ \Lambda^{(\beg)}\nearrow \Lambda^{(\beg+1)} \nearrow \cdots \in \N_0\times \diagrams \]
are allowed to be taken from this larger oriented graph.

With these definitions the transition probabilities
\eqref{eq:transition-probability-augmented} also make sense if
$\Lambda,\widetilde{\Lambda}\in\N_0\times \diagrams$ are taken from this larger
oriented graph and can be used to define Markov chains valued in $\N_0\times
\diagrams$.

\subsection{Transition probabilities for augmented Plancherel growth processes}
\label{sec:augmented-plancherel-growth-proccess}

For the~purposes of the~current section we will view $\Augmented$
as a subset of $\N_0\times \diagrams$, cf.~\eqref{eq:augmented-x}. In this way the augmented Plancherel growth process
initiated at time $m$, cf.~\eqref{eq:Markov}, can be viewed as the aforementioned Markov chain 
\begin{equation} 
    \label{eq:augmented-process-x}
    \Big( \big( x_m^{(t)}, \lambda^{(t)} \big) \Big)_{t \geq m}
\end{equation}
valued in $\N_0\times\diagrams$.

Let us fix some integer $n\in\N_0$. For each integer $m\in\{0,\dots,n\}$
we may remove some initial entries of the~sequence
\eqref{eq:augmented-process-x} 
and consider the~Markov chain
\begin{equation} 
\label{eq:Markov-truncated}
\Big( \big( x_m^{(t)}, \lambda^{(t)} \big) \Big)_{t \geq n} 
\end{equation}
which is indexed by the~time parameter $t\geq n$. 
In this way we obtain a~whole family of 
Markov chains \eqref{eq:Markov-truncated} indexed by an~integer
$m\in\{0,\dots,n\}$ which have the~same transition
probabilities \eqref{eq:transition-probability-augmented}.

The latter encourages us to consider a~general class of Markov chains
\begin{equation} 
\label{eq:Markov-truncated2}
\Big( \big( x^{(t)}, \lambda^{(t)} \big) \Big)_{t \geq n} 
\end{equation}
valued in $\N_0\times \diagrams\supset \Augmented$, for which the~transition
probabilities are given by \eqref{eq:transition-probability-augmented} and for
which the~initial probability distribution of $ \big( x^{(n)},
\lambda^{(n)}\big)$ can~be arbitrary. We will refer to each such a~Markov
chain as \emph{augmented Plancherel growth process}.

\begin{proposition}
    \label{prop:transition-augmented}

Let an~integer $k\in\N_0$ and a~real number $0<p<1$ be fixed,
and let $n'=n'(n)$ be a~sequence of integers such that $n'\geq n$ and 
\[ \lim_{n\to\infty} \sqrt{\frac{n}{n'}}= p.\]
For a~given integer $n\geq 0$ let \eqref{eq:Markov-truncated2} be an~augmented
Plancherel growth process with the~initial probability distribution at time $n$
given by
\[ \delta_k \times \Plancherel_n .\]

Then the~total variation
distance
\begin{equation}
\label{eq:tvd}
\delta\left\{ \left(x^{(n')},\lambda^{(n')} \right), 
    \; \operatorname{Binom}\left(k, p \right)  \times \Plancherel_{n'}\right\} 
    \end{equation}
converges to $0$, as $n\to\infty$.
\end{proposition}
\begin{proof}
Let $\epsilon>0$ be given. By \cref{lem:poisson-binom} there exists some $h>0$
with the~property that for each $q \in \{1, p\}$ the~total variation distance
between the~measure $\measurePoisson_{k,q,h}$ defined in~\eqref{eq:defmukph} 
and the~binomial
distribution $\operatorname{Binom}(k,q)$ is bounded from above by $\epsilon$.

\medskip

Let $T$ be a~map defined on the set of 
probability measures on $\N_0 \times \diagrams_n$
in the~following way. 
For a~probability measure $\mu$ on $\N_0\times \diagrams_n$ consider
the~augmented Plancherel growth process \eqref{eq:Markov-truncated2} with the
initial probability distribution at time $n$ given by $\mu$ and
define $T \mu$ to be the~probability measure on $\N_0\times \diagrams_{n'}$ which gives the
probability distribution of $\big(x^{(n')},\lambda^{(n')} \big)$ at time~$n'$.

It is easy to extend the~map $T$ so that it becomes a~linear map between
the~vector space of \emph{signed} measures on $\N_0\times \diagrams_n$ and the
vector space of \emph{signed} measures on $\N_0\times \diagrams_{n'}$. We equip
both vector spaces with a~metric which corresponds to the~total variation
distance. Then $T$ is a~contraction because of 
Markovianity of the~augmented Plancherel growth process.

\medskip

For $m\in\{0,\dots,n\}$ and $t\geq n$ we denote by $\mu_m(t)$ the~probability
measure on $\N_0\times \diagrams$, defined by the~probability distribution at time $t$
of the~augmented Plancherel growth process $\big(x_m^{(t)},\lambda^{(t)}
\big)$ initiated at time $m$. For the~aforementioned value of $h>0$ we consider the
signed measure on $\N_0\times \diagrams_t$ given by 
the~linear combination
\[ \PP(t):=\frac{1}{\left( e^h -1 \right)^k} \sum_{0\leq j \leq k} 
(-1)^{k-j} \binom{k}{j} e^{j h} \mu_{\left\lfloor j h \sqrt{n} \right\rfloor}(t) \]
(which is well-defined for sufficiently big values of $n$ which assure that $k
h \sqrt{n}< n\leq t$). 

We apply \cref{prop:distribution-fixed-time}; 
it follows that for any $j \in \{0, \dots, k\}$ 
the~total variation distance between 
$\mu_{\left\lfloor j h \sqrt{n} \right\rfloor}(n)$ 
and the~product measure
\[
\operatorname{Pois}(jh) \times \Plancherel_n 
\]
converges to $0$, as $n\to\infty$; it follows that
the~total
variation distance between $\PP(n)$ and the~product measure
\begin{equation} 
\label{eq:almost-delta}
\measurePoisson_{k,1,h} \times \Plancherel_{n}
\end{equation}
converges to $0$, as $n\to\infty$.
On the~other hand, the~value of $h>0$ was
selected in such a~way that the~total variation distance between the
probability measure \eqref{eq:almost-delta} and the~product measure
\begin{equation}
\label{eq:deltatimesplan}
 \delta_k \times \Plancherel_n 
\end{equation}
is smaller than~$\epsilon$. In this way we proved that 
\[ \limsup_{n\to\infty} \, \delta\Big\{\, \PP(n), \  \delta_k \times \Plancherel_n  \Big\} \leq \epsilon.\]
An~analogous reasoning shows that
\[ \limsup_{n\to\infty} \, \delta\Big\{\, \PP(n'), \  \operatorname{Binom}(k,p) \times \Plancherel_{n'}  \Big\} \leq \epsilon.\]

The~image of $\PP(n)$ under the~map $T$ can~be calculated by linearity of $T$:
\[ \PP(n') =  T \PP(n).\]

By the triangle inequality and the~observation that the~map $T$ is a~contraction,
\begin{multline*} \eqref{eq:tvd} \leq 
\delta\left\{ \left(x^{(n')},\lambda^{(n')} \right), 
\;  \PP(n') \right\} + \epsilon \leq \\ 
\delta\left\{ \left(x^{(n)},\lambda^{(n)} \right), 
\;  \PP(n) \right\} + \epsilon \leq 2\epsilon 
\end{multline*}
holds true for sufficiently big values of $n$, as required.
\end{proof}

\subsection{Bumping route in the~lazy parametrization converges to the Poisson
    process}

Let $\left( N(t) : t\geq 0 \right)$ denote the Poisson counting process which is independent
from the Plancherel growth process $\lambda^{(0)}\nearrow\lambda^{(1)}\nearrow
\cdots$. The following result is the lazy version of \cref{thm:main-poisson}.

\begin{theorem}
    \label{thm:lazy-poisson}
    
    Let $l\geq 1$ be a~fixed integer, and $z_1>\cdots> z_l$ be a~fixed sequence
of positive real numbers.

Let $\Lambda_m^{(m)}\nearrow \Lambda_m^{(m+1)} \nearrow \cdots$ be the augmented Plancherel growth process initiated at time~$m$. 
We denote $\Lambda_m^{(t)}=\big(\lambda^{(t)},
\Box_m^{(t)}\big)$; let $\Box_m^{(t)}=\big(x_m^{(t)}, y_m^{(t)}\big)$ be the
coordinates of the special box at time $t$.
        
    For each $1\leq i\leq l$ let $t_i=t_i(m)$ be a~sequence of positive
integers such that
    \[ \lim_{m\to\infty} \frac{m}{\sqrt{t_i}} = z_i.\]
    We assume that $t_1\leq \cdots \leq t_l$.
    Then the~total variation distance between
\begin{itemize}
    \item the~probability distribution of the~vector
    \begin{equation} 
    \label{eq:RANDOM-A}
    \left( x_m^{(t_1)}, \dots , x_m^{(t_l)}, \lambda^{(t_l)} \right), 
    \end{equation}
     and 
       
    \item the~probability distribution of the~vector
    \begin{equation} 
    \label{eq:RANDOM-B} 
    \left(  N(z_1), \dots,
        N(z_l), \lambda^{(t_l)} \right) 
    \end{equation}			
\end{itemize}
converges to $0$, as $m\to\infty$.   

\end{theorem}
\begin{proof}
We will perform the proof by induction over $l$. 
Its main idea is that the~collection of the~random vectors
\eqref{eq:RANDOM-A} over $l\in\{1,2,\dots\}$ forms a~Markov chain;
the~same holds true for the~analogous collection of 
the~random vectors~\eqref{eq:RANDOM-B}. 
We will compare their initial probability
distributions (thanks to \cref{prop:distribution-fixed-time}) and --- in 
a~very specific sense --- we will compare the~kernels of these Markov chains
(with \cref{prop:transition-augmented}). We present the~details below.

\medskip

    The~induction base $l=1$ coincides with
\cref{prop:distribution-fixed-time}. 
    
    \medskip

    We will prove now the~induction step. We start with the~probability
distribution of the~vector \eqref{eq:RANDOM-A} (with the~substitution
$l:=l+1$). Markovianity of the augmented Plancherel growth process implies
that this probability distribution is given by
    \begin{multline}
    \label{eq:tvda}
    \PP\left\{  \left( x_m^{(t_1)}, \dots , x_m^{(t_{l+1})}, \lambda^{(t_{l+1})} \right) = 
                               (x_1,\dots,x_{l+1},\lambda) \right\} =\\
 \shoveleft{   \sum_{\mu\in\diagrams_{t_l}}
    \PP\left\{  \left( x_m^{(t_1)}, \dots , x_m^{(t_{l})}, \lambda^{(t_{l})} \right) = 
                              (x_1,\dots,x_{l},\mu) \right\} \times }\\
    \times \PPcondCurly{  \left( x_m^{(t_{l+1})}, \lambda^{(t_{l+1})} \right) =
                  \left( x_{l+1}, \lambda\right)  }{ 
                   \left( x_m^{(t_{l})}, \lambda^{(t_{l})} \right) =
                  \left( x_{l}, \mu \right) }                
      \end{multline}
for any $x_1,\dots,x_{l+1}\in\N_0$ and $\lambda\in\diagrams$.
We define the~probability measure $\mathbb{Q}$ 
on $\N_0^{l+1}\times \diagrams$ which to
a~tuple $\left(  x_1,\dots,x_{l+1},\lambda~\right)$ assigns the~probability
   \begin{multline}
\label{eq:tvdb}
\mathbb{Q}\left(  x_1,\dots,x_{l+1},\lambda~\right) :=\\
\shoveleft{
\sum_{\mu\in\diagrams_{t_l}}
\PP\left\{  \left( N(z_1),\dots, N(z_l)  \right) = 
(x_1,\dots,x_{l}) \right\}  \times  
 \Plancherel_{t_l}(\mu) \times} \\
\times \PPcondCurly{   \left( x_m^{(t_{l+1})}, \lambda^{(t_{l+1})} \right) =
\left( x_{l+1}, \lambda\right)  }{ 
\left( x_m^{(t_{l})}, \lambda^{(t_{l})} \right)  =
\left( x_{l}, \mu \right)     }          
. \end{multline}
In the~light of the~general definition~\eqref{eq:Markov-truncated2} of the
augmented Plancherel growth process, the~measures \eqref{eq:tvda} and
\eqref{eq:tvdb} on $\N_0^{l+1}\times \diagrams$ can~be viewed as applications
of the~same Markov kernel (which correspond to the last factors on the right-hand side of \eqref{eq:tvda} and \eqref{eq:tvdb})
\[
\PPcondCurly{   \left( x_m^{(t_{l+1})}, \lambda^{(t_{l+1})} \right) =
\left( x_{l+1}, \lambda\right)  }{ 
\left( x_m^{(t_{l})}, \lambda^{(t_{l})} \right)  =
\left( x_{l}, \mu \right)     }, 
\quad 
\mu\in\diagrams_{t_l},  
\]
to two specific initial probability distributions.
Since such an~application of a~Markov kernel is a~contraction (with respect to
the~total variation distance), we proved in this way that the~total variation
distance between \eqref{eq:tvda} and \eqref{eq:tvdb} is bounded from above by
the~total variation distance between the~initial distributions, that is the~random
vectors \eqref{eq:RANDOM-A} and~\eqref{eq:RANDOM-B}. 
By the~inductive hypothesis 
the total variation distance between the measures $\PP$ and $\mathbb{Q}$ 
converges to zero as $m\to\infty$.
The remaining difficulty is to understand the asymptotic behavior of the measure $\mathbb{Q}$.

\smallskip

Observe that the~sum on the~right hand side of \eqref{eq:tvdb}
\begin{multline}
\sum_{\mu\in\diagrams_{t_l}} \Plancherel_{t_l}(\mu) \times \\
\times \PPcondCurly{   \left( x_m^{(t_{l+1})}, \lambda^{(t_{l+1})} \right) =
\left( x_{l+1}, \lambda\right)  }{ 
\left( x_m^{(t_{l})}, \lambda^{(t_{l})} \right) =
\left( x_{l}, \mu \right) } = \\[1ex]
= \PPcondCurly{  \left(x^{(n')},\lambda^{(n')}\right) = 
      \left(x_{l+1}, \lambda\right)  }{  \left(x^{(n)},\lambda^{(n)}\right)  \overset{d}{=}  \delta_k \times \Plancherel_n  }            
\end{multline}
is the~probability distribution of the~random vector
$\big(x^{(n')},\lambda^{(n')}\big)$ which appears in
\cref{prop:transition-augmented} with $n'=t_{l+1}$, and $n=t_l$, and $p=
\frac{z_{l+1}}{z_{l}}$, and $k=x_l$.
Therefore we proved that the~measure $\mathbb{Q}$ is in 
an~$o(1)$-neighborhood of the~following probability measure
\begin{multline}
\label{eq:prob-Q} 
\mathbb{Q}'\left(  x_1,\dots,x_{l+1},\lambda~\right) :=\\
{
\PP\left\{  \big( N(z_1),\dots, N(z_l)  \big) = 
(x_1,\dots,x_{l}) \right\}   \times} \\ 
\times \Plancherel_{n'}(\lambda) \ 
{\operatorname{Binom}\left( x_l, \frac{z_{l+1}}{z_l} \right)(x_{l+1})}.
\end{multline} 

It is easy to check that 
\[ 
\PPcond{N(z_{l+1})=x_{l+1}}{N(z_{l})=x_{l}}
= \operatorname{Binom}\left( x_l, \frac{z_{l+1}}{z_l} \right)(x_{l+1}).
\]
Hence the~probability of the~binomial distribution 
which appears as the~last factor
on the~right-hand side of \eqref{eq:prob-Q}
can~be interpreted as the~conditional probability distribution of the Poisson
process in the~past, given its value in the~future. 

We show that the Poisson counting process with the~reversed time is also a~Markov process.
Since the Poisson counting process has independent increments, the probability of the event
\[ 
\big( N(z_1),\dots, N(z_l) \big) = (x_1,\dots,x_l) 
\]
can be written as a product; an analogous observation is valid for
$l:=l+1$. Due to cancellations of the factors which contribute to the numerator
and the denominator, the following conditional probability can be simplified:
\begin{multline*}
  \PPcond{N(z_{l+1})=x_{l+1}}{\big( N(z_1),\dots, N(z_l) \big) = (x_1,\dots,x_l)}
 = \\[1ex]
  = \frac{ \PP\left\{ \big( N(z_1),\dots, N(z_{l+1}) \big) = (x_1,\dots,x_{l+1}) \right\}
  }{ 
\PP\left\{ \big( N(z_1),\dots, N(z_{l}) \big) = (x_1,\dots,x_{l}) \right\}
} 
= \\[1ex]
  = \frac{\PP\big( N(z_{l})=x_{l} \: \wedge \: N(z_{l+1})=x_{l+1} \big) }{ \PP\big(  N(z_{l})=x_{l} \big)} = \\[1ex] =
  \PPcond{N(z_{l+1})=x_{l+1}}{N(z_{l})=x_{l}}.
\end{multline*}

By combining the above observations with \eqref{eq:prob-Q} it follows that
\begin{multline*}
\mathbb{Q}'\left(  x_1,\dots,x_{l+1},\lambda \right)=\\
\PP\left\{  \left( N(z_1),\dots, N(z_{l+1})  \right) = 
(x_1,\dots,x_{l+1}) \right\}  \ \Plancherel_{t_{l+1}}(\lambda) 
\end{multline*}
is the~probability distribution of \eqref{eq:RANDOM-B} (with the~obvious
substitution  $l:=l+1$) which completes the~inductive step. 
\end{proof}

\subsection{Lazy version of \cref{rem:poisson}}

The~special case $l=0$ of the~following result seems to be closely related to a
very recent work of Azangulov and Ovechkin \cite{Azangulov2020} who used
different methods.

\begin{proposition}
    \label{prop:lazy-remark}
Let $(\psi_i)$ be a~sequence of independent, identically distributed random
variables with the~exponential distribution $\operatorname{Exp}(1)$.
    
For each $l\in\N_0$ the~joint distribution of the~finite tuple of random
variables
\begin{equation}
\label{eq:vector}
   \left( \frac{m}{\sqrt{T^{[m]}_0}},\dots, \frac{m}{\sqrt{T^{[m]}_l}} \right) 
\end{equation}
converges, as $m\to\infty$, to the~joint distribution of the~sequence of partial sums
\[ \left( \psi_0,\;\;\; \psi_0+\psi_1,\;\;\; \dots, \;\;\; \psi_0+\psi_1+\cdots+\psi_l \right).\]    
\end{proposition}
\begin{proof}
For any $s_0,\dots,s_l>0$ the~cumulative distribution function of the~random
vector \eqref{eq:vector}
\begin{multline} 
\label{eq:CDF}
\PP\left( \frac{m}{\sqrt{T^{[m]}_0}} < s_0, \quad \dots, \quad
\frac{m}{\sqrt{T^{[m]}_l}} < s_l
  \right)  =\\
  \PP\left( x_m^{(t_0)} > 0, \quad x_m^{(t_1)} > 1, \quad \dots, \quad x_m^{(t_l)} > l \right)
\end{multline}
can~be expressed directly in terms of the~cumulative distribution of
the~random vector
$\left(  x_m^{(t_0)}, \dots,  x_m^{(t_l)} \right) $
with
\[ t_i = t_i(m)= \left\lfloor \left(\frac{m}{s_i}\right)^2\right\rfloor. \]
\cref{thm:lazy-poisson} shows that the~right-hand side of \eqref{eq:CDF}
converges to
\begin{multline*} \PP\Big( N(s_0)>0, \quad N(s_1)>1, \quad \ldots, \;\;\;  N(s_l)>l \Big) = \\
\PP\Big( \psi_0 \leq s_0, \quad \psi_0+\psi_1 \leq  s_1, \quad \dots, \quad \psi_0+\cdots+\psi_l\leq s_l \Big),  
\end{multline*}
where 
\[ 
\psi_i= 
\inf\Big\{ t : N(t)\geq i+1 \Big\} -  \inf\Big\{ t : N(t)\geq i \Big\}, 
\ \ \ i \in \N_0,
\]
denote the~time between the~jumps of the Poisson process. Since
$(\psi_0,\psi_1,\dots)$ form a~sequence of independent random variables with
the~exponential distribution, this concludes the~proof.
\end{proof}

\subsection{Conjectural generalization}

We revisit \cref{sec:trajectory} with some changes. This time let
\[ 
\xi=(\ldots,\xi_{-2},\xi_{-1},\xi_{0},\xi_1,\dots)
\] 
be a~\emph{doubly} infinite
sequence of independent, identically distributed random variables with the
uniform distribution $U(0,1)$ on the~unit interval $[0,1]$.
Let us fix $m\in\R_+$. For $s,t\in\R_+$ we define 
\begin{multline*} \big( x_m(s,t), y_m(s,t) \big) = \\
\Pos_\infty \left(  P\left(\xi_{-\left\lfloor m s \right\rfloor},\dots,\xi_{-2},\xi_{-1}, \infty, \xi_{1}, \xi_{2}, \dots, \xi_{\left\lfloor \frac{m^2}{t^2} \right\rfloor} \right) \right).
\end{multline*}

Let $\mathcal{N}$ denote the~Poisson point process with the~uniform unit
intensity on~$\R_+^2$. For $s,t\in\R_+$ we denote by
\[ \mathcal{N}_{s,t}=\mathcal{N}\left( [0,s] \times [0,t] \right) \]
the~number of sampled points in the~specified rectangle.

\begin{conjecture}
The~random function 
\begin{align}
\label{eq:2dPoisson}
\R_+^2 \ni (s,t) &\mapsto x_m(s,t) \\
\intertext{converges in distribution to Poisson point process}
\label{eq:2dPoisson-true}
 \R_+^2 \ni (s,t) &\mapsto \mathcal{N}_{s,t} 
\end{align}
in the~limit as $m\to\infty$.   
\end{conjecture}
Note that the~results of the~current paper show the~convergence of 
the marginals which correspond to (a) fixed value of $s$ and all values of $t>0$
(cf.~\cref{thm:lazy-poisson}), or (b) fixed value of $t$ and all values of
$s>0$ (this is a~corollary from the~proof of
\cref{prop:distribution-fixed-time}).

It is a~bit discouraging that the~contour curves obtained in computer
experiments (see \cref{fig:2dPoisson}) \emph{do not seem} to be counting the
number of points from some set which belong to a~specified rectangle,
see \cref{fig:2dPoissont} for comparison.
 On the
other hand, maybe the~value of $m$ used in our experiments was not big enough
to reveal the~asymptotic behavior of these curves.

\subfile{figures-Poisson/2D-Poisson/2D-Poisson.tex}

\section{Removing laziness}
\label{sec:removing-laziness}

Most of the~considerations above concerned the~lazy parametrization of the
bumping routes. In this section we will show how to pass to the~parametrization
by the~row number and, in this way, to prove the~remaining claims from
\cref{sec:preliminaries} (that is \cref{thm:poisson-point,prop:m=1}).

\subsection{Proof of \cref{prop:m=1}}
\label{sec:proof:prop:m=1}

\newcommand{\ysmall}{y}

Our general strategy in this proof is to use \cref{lem:lazy:m=1} and to use the
observation that a Plancherel-distributed random Young diagram with $n$ boxes has
approximately $2\sqrt{n}$ columns in the scaling when $n\to\infty$.

\begin{proof}[Proof of \cref{prop:m=1}]

We denote by $c^{(n)}$ the~number of rows (or, equivalently, the~length of the
leftmost column) of the~Young diagram $\lambda^{(n)}$. Our~proof will be based
on an~observation (recall~\cref{prop:lazy-traj-correspondence}) that
\[ Y_0^{[m]}  = c^{\left( T_0^{[m]} \right)}. \]

Let $\epsilon>0$ be fixed. 
Since $c^{(n)}$ has the~same distribution as the
length of the~bottom row of a~Plancherel-distributed random Young diagram with
$n$ boxes, the~large deviation results \cite{Deuschel1999,Seppaelaeinen1998}
show that there exists a~constant $C_\epsilon>0$ such that
\begin{multline}
\label{eq:LD}
\PP\left( \sup_{n\geq n_0} \left| \frac{c^{(n)}}{\sqrt{n}}-2 \right| > \epsilon\right) \leq 
\sum_{n\geq n_0} \PP\left(  \left| \frac{c^{(n)}}{\sqrt{n}}-2 \right| > \epsilon\right)  \leq \\
\sum_{n\geq n_0} e^{-C_\epsilon \sqrt{n}}  
= O \! \left( e^{-C_\epsilon \sqrt{n_0}} \right) \leq 
o\left( \frac{1}{n_0} \right)
\end{multline}
in the~limit as $n_0 \to\infty$.

\medskip

Consider an arbitrary integer $y\geq 1$. Assume that (i) the~event on the~left-hand side of \eqref{eq:LD} \emph{does not} hold true for
$n_0:=y$, \emph{and} (ii) $Y_0^{[m]}\geq y$. 
Since $T_0^{[m]}\geq Y_0^{[m]}\geq
y$  it follows that
\[
\left| \frac{Y^{[m]}_0}{\sqrt{T_0^{[m]}}}-2 \right|
 = \left| \frac{c^{\left(T_0^{[m]}\right)}}{\sqrt{T_0^{[m]}}}-2 \right| \leq  \epsilon \]
hence
\begin{equation}
\label{eq:CV} 
T_0^{[m]} \geq \left( \frac{y}{2+\epsilon} \right)^2. 
\end{equation}

By considering two possibilities: either the~event on the~left-hand side 
of~\eqref{eq:LD} holds true for $n_0:=y$ or not, it follows that
\[ \PP\left\{  Y_0^{[m]} \geq y\right\} \leq 
o\left( \frac{1}{y} \right) +
\PP\left\{  T_0^{[m]} \geq \left( \frac{y}{2+\epsilon} \right)^2  \right\}. 
\]
\cref{lem:lazy:m=1} implies therefore that
\[ \PP\left\{  Y_0^{[m]} \geq y\right\} \leq  
\frac{(2+\epsilon) m}{y} +
o\left( \frac{1}{y} \right)
\]
which completes the~proof of the~upper bound.

\medskip

For the~lower bound, assume that (i) the~event on the~left-hand side of
\eqref{eq:LD} \emph{does not} hold true for 
\[n_0:= \left\lceil \left( \frac{y}{2-\epsilon}\right)^2 \right\rceil\] 
\emph{and} (ii) $T_0^{[m]}\geq n_0$. 
In an~analogous way as in the~proof of \eqref{eq:CV} it follows that
\[ Y_0^{[m]} \geq (2-\epsilon) \sqrt{T_0^{[m]}} \geq  
 (2-\epsilon)  \sqrt{n_0} \geq y .\]

By considering two possibilities: either the~event on the~left-hand side 
of~\eqref{eq:LD} holds true or not, it follows that that
\[ \PP\left\{  T_0^{[m]} \geq n_0 \right\} \leq  o\left( \frac{1}{y} \right) +
\PP\left\{  Y_0^{[m]} \geq y \right\}.
\]
\cref{lem:lazy:m=1} implies therefore that
\[ \PP\left\{  Y_0^{[m]} \geq y\right\} \geq  
\frac{(2-\epsilon) m}{y} +
o\left( \frac{1}{y} \right)
\]
which completes the~proof of the~lower bound.
\end{proof}

\subsection{Lazy parametrization versus row parametrization}

\begin{proposition}
    \label{prop:lazy-nonlazy}
    For each $x\in\N_0$
\[ \lim_{m \to\infty}  \frac{Y^{[m]}_x }{\sqrt{T^{[m]}_x} } = 2 \]   
holds true in probability.
\end{proposition}

Regretfully, the~ideas used in the~proof of \cref{prop:m=1}
(cf.~\cref{sec:proof:prop:m=1} above) are not directly applicable for the~proof
of \cref{prop:lazy-nonlazy} when $x\geq 1$ because we are not aware of suitable
large deviation results for the~lower tail of the~distribution of a~specific row a
Plancherel-distributed Young diagram, other than~the~bottom row.

\medskip

Our general strategy in this proof is to study the length $\mu^{(t)}_x$ of the
column with the fixed index $x$ in the Plancherel growth process $\lambda^{(t)}$,
as $t\to\infty$. Since we are unable to get asymptotic \emph{uniform} bounds for
\begin{equation}
    \label{eq:myratio} 
    \left| \frac{\mu_x^{(t)}}{\sqrt{t}} - 2 \right| 
\end{equation}
over \emph{all} integers $t$ such that $\frac{t}{m^2}$ belongs to some compact
subset of $(0,\infty)$ in the limit $m\to\infty$, as a substitute we consider a
\emph{finite} subset of $(0,\infty)$ of the form
\[ 
\left\{ c (1+\epsilon),\; \dots,\; c(1+\epsilon)^l \right\} 
\]
for arbitrarily small values of $c,\epsilon>0$ and arbitrarily large integer
$l\geq 0$ and prove the appropriate bounds for the integers $t_i(m)$ for which
\scalebox{0.95}{$\displaystyle \frac{t_i}{m^2}$} are approximately elements of this finite set.
We will use monotonicity in order to get some information about \eqref{eq:myratio} also for the integers $t$ which are between the numbers $\{ t_i(m) \}$.

\begin{proof}    
Let $\epsilon>0$ be fixed. Let $\delta>0$ be arbitrary. By
\cref{prop:lazy-remark} the~law of the~random variable
$\frac{m}{\sqrt{T^{[m]}_x} }$ converges to the Erlang distribution which is
supported on $\R_+$ and has no atom in $0$. Let $W$ be a~random variable with
the~latter probability distribution; in this way the law of $\frac{T_x^{[m]}}{m^2}$ converges to the law of $W^{-2}$.
 Let $c>0$ be a~sufficiently small number
such that
\[ \PP\left( W^{-2} < c \right) < \delta. \]
Now, let $l\in\N_0$ be a~sufficiently big integer so that
\[  \PP\left(   c (1+\epsilon)^l < W^{-2} \right) < \delta. \]

We define
\[ t_i = t_i (m) = \left\lfloor  m^2 c (1+\epsilon)^i \right\rfloor \qquad \text{for } i\in\{0,\dots,l\}.\]
With these notations there exists some $m_1$ with the~property that for each $m\geq m_1$
\begin{equation}
\label{eq:goodlimits}
\PP\left( t_0 < T_x^{[m]} \leq t_l \right) > 1-2\delta.
\end{equation}

\medskip

Let $\mu^{(n)}=\big[ \lambda^{(n)}\big]^T$ be the~transpose of
$\lambda^{(n)}$; in this way $\mu^{(n)}_x$ is the~number of the~boxes of
$\tableau$ which are in the~column $x$ and contain an~entry $\leq n$. The
probability distribution of $\mu^{(n)}$ is also given by the Plancherel measure. 
The monograph of Romik \cite[Theorem~1.22]{Romik2015a} contains that proof that 
\[ \frac{\mu^{(t_i(m))}_x}{\sqrt{t_i(m)}} \xrightarrow{\PP} 2\]
holds true in the special case of the bottom row $i=0$;
it is quite straightforward to check that this proof is also valid for 
each $i\in\{0,\dots,l\}$, for the details see \cite[proof of Lemma 2.5]{MMS-Poisson2020v2}.
Hence
there exists some $m_2$ with the~property that for each $m\geq m_2$ 
the~probability of the~event
\begin{equation}
\label{eq:happy-event} 
\left| \frac{\mu^{(t_i(m))}_x}{\sqrt{t_i(m)}} -2 \right| < \epsilon 
                   \quad \text{holds \emph{for each} $i\in\{0,\dots,l\}$} 
\end{equation}
is at least $1-\delta$.

\medskip

Let us consider an~elementary event $\tableau$ with the~property that the~event
considered in \eqref{eq:goodlimits} occurred, that is~$t_0 < T_x^{[m]} \leq t_l$,
\emph{and} the~event \eqref{eq:happy-event} occurred. Since $t_0\leq \cdots
\leq t_l$ form a~weakly increasing sequence, there exists an~index
$j=j(\tableau) \in\{0,\dots,l-1\}$ such that
\[ 
t_{j} < T_x^{[m]} \leq t_{j+1}.
\]
It follows that
\[  \mu_x^{(t_j)}  < Y^{[m]}_x  \leq \mu_x^{(t_{j+1})} \]
hence 
\begin{multline}
\label{eq:goodbound}
(2-\epsilon) \frac{1}{\sqrt{1+\epsilon}+o(1)}
< 
\frac{\mu_x^{(t_{j} )}}{\sqrt{t_{j+1}}}
 \leq   \frac{Y^{[m]}_x }{\sqrt{T^{[m]}_x} }  \leq \\
  \frac{\mu_x^{(t_{j+1})}}{\sqrt{t_j}} < 
  (2+\epsilon) \left( \sqrt{1+\epsilon} + o(1) \right).
\end{multline}

In this way we proved that for each $m\geq \max(m_1,m_2)$ the~probability of
the~event \eqref{eq:goodbound} is at least $1-3\delta$, as required.
\end{proof}

\subsection{Proof of \cref{thm:poisson-point}}
\label{sec:proof:thm:poisson-point}

\begin{proof}
    For each integer $l\geq 0$ \cref{prop:lazy-remark} gives the asymptotics of
the joint probability distribution of the random variables
$T_0^{[m]},\dots,T_l^{[m]}$ which concern the shape of the bumping route in
the lazy parametrization. On the other hand, \cref{prop:lazy-nonlazy}
allows us to express asymptotically these random variables by their non-lazy counterparts
$Y_0^{[m]},\dots,Y_l^{[m]}$.
The discussion from \cref{rem:poisson} completes the proof.
\end{proof}

\section{Acknowledgments}

We thank Iskander Azangulov, Maciej Dołęga, Vadim Gorin, Piet Groeneboom, Adam
Jakubowski, Grigory Ovechkin, Timo Sepp\"{a}l\"{a}inen, and Anatoly Vershik for
discussions and bibliographic suggestions.

\section{Declarations}

\textbf{Funding.} Research supported by Narodowe Centrum Nauki, grant number 2017/26/A/ST1/00189.
Mikołaj Marciniak was additionally supported by 
Narodowe Centrum Badań i Rozwoju, grant number POWR.03.05.00-00-Z302/17-00.

\textbf{Conflicts of interest/Competing interests:} not applicable.

\textbf{Availability of data and material (data transparency):} not applicable.

\textbf{Code availability:} not applicable.

\biblio

\end{document}

%% file: figures-Poisson/big_simulation/random_SYT.tex
\begin{figure}
\centering
\subfloat[]{
\begin{tikzpicture}[scale=1]
\clip (-0.7,-0.7) rectangle (4.8,8.8); 
\fill[blue!20] (3,0) rectangle (4,1); 
\fill[blue!20] (2,1) rectangle (3,2); 
\fill[blue!20] (1,2) rectangle (2,3); 
\fill[blue!20] (1,3) rectangle (2,4); 
\fill[blue!20] (0,4) rectangle (1,5); 
\fill[blue!20] (0,5) rectangle (1,6); 
\fill[blue!20] (0,6) rectangle (1,7); 
\fill[blue!20] (0,7) rectangle (1,8); 
\fill[blue!20] (0,8) rectangle (1,9); 
\draw[black!50] (1,0) --  (1,100);  
\draw[black!50] (2,0) --  (2,100);  
\draw[black!50] (3,0) --  (3,100);  
\draw[black!50] (4,0) --  (4,100);  
\draw[black!50] (5,0) --  (5,100);  
\draw[black!50] (6,0) --  (6,100);  
\draw[black!50] (7,0) --  (7,100);  
\draw[black!50] (8,0) --  (8,100);  
\draw[black!50] (0,1) --  (20,1);  
\draw[black!50] (0,2) --  (20,2);  
\draw[black!50] (0,3) --  (20,3);  
\draw[black!50] (0,4) --  (20,4);  
\draw[black!50] (0,5) --  (20,5);  
\draw[black!50] (0,6) --  (20,6);  
\draw[black!50] (0,7) --  (20,7);  
\draw[black!50] (0,8) --  (20,8);  
\draw[black!50] (0,9) --  (20,9);  
\draw[black!50] (0,10) --  (20,10);  
\draw[black!50] (0,11) --  (20,11);  
\draw[black!50] (0,12) --  (20,12);  
\draw[black!50] (0,13) --  (20,13);  
\draw[black!50] (0,14) --  (20,14);  
\draw[black!50] (0,15) --  (20,15);  
\draw[black!50] (0,16) --  (20,16);  
\draw (0.5,0.5) node { 1 }; 
\draw (1.5,0.5) node { 2 }; 
\draw (2.5,0.5) node { 3 }; 
\draw (3.5,0.5) node { 6 }; 
\draw (4.5,0.5) node { 16 }; 
\draw (5.5,0.5) node { 23 }; 
\draw (6.5,0.5) node { 24 }; 
\draw (7.5,0.5) node { 30 }; 
\draw (8.5,0.5) node { 31 }; 
\draw (9.5,0.5) node { 45 }; 
\draw (0.5,1.5) node { 4 }; 
\draw (1.5,1.5) node { 5 }; 
\draw (2.5,1.5) node { 9 }; 
\draw (3.5,1.5) node { 11 }; 
\draw (4.5,1.5) node { 29 }; 
\draw (5.5,1.5) node { 34 }; 
\draw (6.5,1.5) node { 42 }; 
\draw (7.5,1.5) node { 52 }; 
\draw (8.5,1.5) node { 61 }; 
\draw (9.5,1.5) node { 66 }; 
\draw (0.5,2.5) node { 7 }; 
\draw (1.5,2.5) node { 10 }; 
\draw (2.5,2.5) node { 18 }; 
\draw (3.5,2.5) node { 21 }; 
\draw (4.5,2.5) node { 32 }; 
\draw (5.5,2.5) node { 36 }; 
\draw (6.5,2.5) node { 47 }; 
\draw (7.5,2.5) node { 69 }; 
\draw (8.5,2.5) node { 71 }; 
\draw (9.5,2.5) node { 79 }; 
\draw (0.5,3.5) node { 8 }; 
\draw (1.5,3.5) node { 12 }; 
\draw (2.5,3.5) node { 20 }; 
\draw (3.5,3.5) node { 22 }; 
\draw (4.5,3.5) node { 38 }; 
\draw (5.5,3.5) node { 43 }; 
\draw (6.5,3.5) node { 49 }; 
\draw (7.5,3.5) node { 78 }; 
\draw (8.5,3.5) node { 81 }; 
\draw (9.5,3.5) node { 86 }; 
\draw (0.5,4.5) node { 13 }; 
\draw (1.5,4.5) node { 15 }; 
\draw (2.5,4.5) node { 28 }; 
\draw (3.5,4.5) node { 35 }; 
\draw (4.5,4.5) node { 39 }; 
\draw (5.5,4.5) node { 48 }; 
\draw (6.5,4.5) node { 56 }; 
\draw (7.5,4.5) node { 87 }; 
\draw (8.5,4.5) node { 98 }; 
\draw (9.5,4.5) node { 101 }; 
\draw (0.5,5.5) node { 14 }; 
\draw (1.5,5.5) node { 27 }; 
\draw (2.5,5.5) node { 37 }; 
\draw (3.5,5.5) node { 50 }; 
\draw (4.5,5.5) node { 58 }; 
\draw (5.5,5.5) node { 84 }; 
\draw (6.5,5.5) node { 106 }; 
\draw (7.5,5.5) node { 113 }; 
\draw (8.5,5.5) node { 124 }; 
\draw (9.5,5.5) node { 146 }; 
\draw (0.5,6.5) node { 17 }; 
\draw (1.5,6.5) node { 33 }; 
\draw (2.5,6.5) node { 41 }; 
\draw (3.5,6.5) node { 54 }; 
\draw (4.5,6.5) node { 72 }; 
\draw (5.5,6.5) node { 109 }; 
\draw (6.5,6.5) node { 120 }; 
\draw (7.5,6.5) node { 144 }; 
\draw (8.5,6.5) node { 149 }; 
\draw (9.5,6.5) node { 151 }; 
\draw (0.5,7.5) node { 19 }; 
\draw (1.5,7.5) node { 46 }; 
\draw (2.5,7.5) node { 57 }; 
\draw (3.5,7.5) node { 63 }; 
\draw (4.5,7.5) node { 73 }; 
\draw (5.5,7.5) node { 129 }; 
\draw (6.5,7.5) node { 139 }; 
\draw (7.5,7.5) node { 150 }; 
\draw (8.5,7.5) node { 173 }; 
\draw (9.5,7.5) node { 180 }; 
\draw (0.5,8.5) node { 25 }; 
\draw (1.5,8.5) node { 51 }; 
\draw (2.5,8.5) node { 65 }; 
\draw (3.5,8.5) node { 67 }; 
\draw (4.5,8.5) node { 91 }; 
\draw (5.5,8.5) node { 130 }; 
\draw (6.5,8.5) node { 148 }; 
\draw (7.5,8.5) node { 165 }; 
\draw (8.5,8.5) node { 175 }; 
\draw (9.5,8.5) node { 231 }; 

\foreach \x in {1,...,3}
{ \draw[very thick]  (\x,0 ) +(0,5pt) -- +(0,-5pt) node[anchor=north] {$ \x  $}; }

\foreach \y in {1,...,7}
 { \draw[very thick]  (0, \y ) +(5pt,0) -- +(-5pt,0) node[anchor=east] {$ \y  $}; }

\draw [->,very thick] (0,0) -- (3.8,0) node[anchor=north west] {$x$};
\draw [->,very thick] (0,0) -- (0,7.8) node[anchor=south east] {$y$};
\draw[red,line width=1mm] (3,0) -- (3,1); 
\draw[red,line width=1mm] (2,1) -- (2,2); 
\draw[red,line width=1mm] (1,2) -- (1,3); 
\draw[red,line width=1mm] (1,3) -- (1,4); 
\draw[red,line width=1mm] (0,4) -- (0,5); 
\draw[red,line width=1mm] (0,5) -- (0,6); 
\draw[red,line width=1mm] (0,6) -- (0,7); 
\draw[red,line width=1mm] (0,7) -- (0,8); 
\draw[red,line width=1mm] (0,8) -- (0,9); 
        
\end{tikzpicture}
\label{fig:tableauFrench}
}
\hfill
\subfloat[]{
\begin{tikzpicture}[scale=1]
\clip (-1,-0.7) rectangle (4.8,8.8); 

\begin{scope}[yscale=6/7]
        \fill[blue!20] (3,30) rectangle (4,7.0); 
        \fill[blue!20] (2,7.0) rectangle (3,3.5); 
        \fill[blue!20] (1,3.5) rectangle (2,2.33333); 
        \fill[blue!20] (1,2.33333) rectangle (2,1.75); 
        \fill[blue!20] (0,1.75) rectangle (1,1.4); 
        \fill[blue!20] (0,1.4) rectangle (1,1.16667); 
        \fill[blue!20] (0,1.16667) rectangle (1,1.0); 
        \fill[blue!20] (0,1.0) rectangle (1,0.875); 
        \fill[blue!20] (0,0.875) rectangle (1,0.77778); 
        \fill[blue!20] (0,0.77778) rectangle (1,0.7); 
        \fill[blue!20] (0,0.7) rectangle (1,0.63636); 
        \fill[blue!20] (0,0.63636) rectangle (1,0.58333); 
        \fill[blue!20] (0,0.58333) rectangle (1,0.53846); 
        \fill[blue!20] (0,0.53846) rectangle (1,0.5); 
        \fill[blue!20] (0,0.5) rectangle (1,0.46667); 
        \fill[blue!20] (0,0.46667) rectangle (1,0.4375); 
        \fill[blue!20] (0,0.4375) rectangle (1,0.41176); 
        \fill[blue!20] (0,0.41176) rectangle (1,0.38889); 
        \fill[blue!20] (0,0.38889) rectangle (1,0.36842); 
        \fill[blue!20] (0,0.36842) rectangle (1,0.35); 
        \fill[blue!20] (0,0.35) rectangle (1,0.33333); 
        \fill[blue!20] (0,0.33333) rectangle (1,0.31818); 
        \fill[blue!20] (0,0.31818) rectangle (1,0.30435); 
        \fill[blue!20] (0,0.30435) rectangle (1,0.29167); 
        \fill[blue!20] (0,0.29167) rectangle (1,0.28); 
        \fill[blue!20] (0,0.28) rectangle (1,0.26923); 
        \fill[blue!20] (0,0.26923) rectangle (1,0.25926); 
        \fill[blue!20] (0,0.25926) rectangle (1,0.25); 
        \fill[blue!20] (0,0.25) rectangle (1,0.24138); 
        \fill[blue!20] (0,0.24138) rectangle (1,0.23333); 
        \fill[blue!20] (0,0.23333) rectangle (1,0.22581); 
        \fill[blue!20] (0,0.22581) rectangle (1,0.21875); 
        \fill[blue!20] (0,0.21875) rectangle (1,0.21212); 
        \fill[blue!20] (0,0.21212) rectangle (1,0.20588); 
        \fill[blue!20] (0,0.20588) rectangle (1,0.2); 
        \fill[blue!20] (0,0.2) rectangle (1,0.19444); 
        \fill[blue!20] (0,0.19444) rectangle (1,0.18919); 
        \fill[blue!20] (0,0.18919) rectangle (1,0.18421); 
        \fill[blue!20] (0,0.18421) rectangle (1,0.17949); 
        \fill[blue!20] (0,0.17949) rectangle (1,0.175); 
        \fill[blue!20] (0,0.175) rectangle (1,0.17073); 
        \fill[blue!20] (0,0.17073) rectangle (1,0.16667); 
        \fill[blue!20] (0,0.16667) rectangle (1,0.16279); 
        \fill[blue!20] (0,0.16279) rectangle (1,0.15909); 
        \fill[blue!20] (0,0.15909) rectangle (1,0.15556); 
        \fill[blue!20] (0,0.15556) rectangle (1,0.15217); 
        \fill[blue!20] (0,0.15217) rectangle (1,0.14894); 
        \fill[blue!20] (0,0.14894) rectangle (1,0.14583); 
        \fill[blue!20] (0,0.14583) rectangle (1,0.14286); 
        \fill[blue!20] (0,0.14286) rectangle (1,0.14); 
        \fill[blue!20] (0,0.14) rectangle (1,0.13725); 
        \fill[blue!20] (0,0.13725) rectangle (1,0.13462); 
        \fill[blue!20] (0,0.13462) rectangle (1,0.13208); 
        \fill[blue!20] (0,0.13208) rectangle (1,0.12963); 
        \fill[blue!20] (0,0.12963) rectangle (1,0.12727); 
        \fill[blue!20] (0,0.12727) rectangle (1,0.125); 
        \fill[blue!20] (0,0.125) rectangle (1,0.12281); 
        \fill[blue!20] (0,0) rectangle (1,0.125); 
        \draw[black!50] (1,30) --  (1,0.0007);  
        \draw[black!50] (2,30) --  (2,0.0007);  
        \draw[black!50] (3,30) --  (3,0.0007);  
        \draw[black!50] (4,30) --  (4,0.0007);  
        \draw[black!50] (5,30) --  (5,0.0007);  
        \draw[black!50] (6,30) --  (6,0.0007);  
        \draw[black!50] (7,30) --  (7,0.0007);  
        \draw[black!50] (8,30) --  (8,0.0007);  
        \draw[black!50] (0,7.0) --  (20,7.0);  
        \draw[black!50] (0,3.5) --  (20,3.5);  
        \draw[black!50] (0,2.33333) --  (20,2.33333);  
        \draw[black!50] (0,1.75) --  (20,1.75);  
        \draw[black!50] (0,1.4) --  (20,1.4);  
        \draw[black!50] (0,1.16667) --  (20,1.16667);  
        \draw[black!50] (0,1.0) --  (20,1.0);  
        \draw[black!50] (0,0.875) --  (20,0.875);  
        \draw[black!50] (0,0.77778) --  (20,0.77778);  
        \draw[black!50] (0,0.7) --  (20,0.7);  
        \draw (0.5,7.36842) node { 1 }; 
        \draw (1.5,7.36842) node { 2 }; 
        \draw (2.5,7.36842) node { 3 }; 
        \draw (3.5,7.36842) node { 6 }; 
        \draw (4.5,7.36842) node { 16 }; 
        \draw (5.5,7.36842) node { 23 }; 
        \draw (6.5,7.36842) node { 24 }; 
        \draw (7.5,7.36842) node { 30 }; 
        \draw (8.5,7.36842) node { 31 }; 
        \draw (9.5,7.36842) node { 45 }; 
        \draw (0.5,4.66667) node { 4 }; 
        \draw (1.5,4.66667) node { 5 }; 
        \draw (2.5,4.66667) node { 9 }; 
        \draw (3.5,4.66667) node { 11 }; 
        \draw (4.5,4.66667) node { 29 }; 
        \draw (5.5,4.66667) node { 34 }; 
        \draw (6.5,4.66667) node { 42 }; 
        \draw (7.5,4.66667) node { 52 }; 
        \draw (8.5,4.66667) node { 61 }; 
        \draw (9.5,4.66667) node { 66 }; 
        \draw (0.5,2.8) node { 7 }; 
        \draw (1.5,2.8) node { 10 }; 
        \draw (2.5,2.8) node { 18 }; 
        \draw (3.5,2.8) node { 21 }; 
        \draw (4.5,2.8) node { 32 }; 
        \draw (5.5,2.8) node { 36 }; 
        \draw (6.5,2.8) node { 47 }; 
        \draw (7.5,2.8) node { 69 }; 
        \draw (8.5,2.8) node { 71 }; 
        \draw (9.5,2.8) node { 79 }; 
        \draw (0.5,2.0) node { 8 }; 
        \draw (1.5,2.0) node { 12 }; 
        \draw (2.5,2.0) node { 20 }; 
        \draw (3.5,2.0) node { 22 }; 
        \draw (4.5,2.0) node { 38 }; 
        \draw (5.5,2.0) node { 43 }; 
        \draw (6.5,2.0) node { 49 }; 
        \draw (7.5,2.0) node { 78 }; 
        \draw (8.5,2.0) node { 81 }; 
        \draw (9.5,2.0) node { 86 };

        \draw[red,line width=1mm] (3,30) -- (3,7.0); 
        \draw[red,line width=1mm] (2,7.0) -- (2,3.5); 
        \draw[red,line width=1mm] (1,3.5) -- (1,2.33333); 
        \draw[red,line width=1mm] (1,2.33333) -- (1,1.75); 
        \draw[red,line width=1mm] (0,1.75) -- (0,1.4); 
        \draw[red,line width=1mm] (0,1.4) -- (0,1.16667); 
        \draw[red,line width=1mm] (0,1.16667) -- (0,1.0); 
        \draw[red,line width=1mm] (0,1.0) -- (0,0.875); 
        \draw[red,line width=1mm] (0,0.875) -- (0,0.77778); 
        \draw[red,line width=1mm] (0,0.77778) -- (0,0.7); 
        \draw[red,line width=1mm] (0,0.7) -- (0,0.63636); 
        \draw[red,line width=1mm] (0,0.63636) -- (0,0.58333); 
        \draw[red,line width=1mm] (0,0.58333) -- (0,0.53846); 
        \draw[red,line width=1mm] (0,0.53846) -- (0,0.5); 
        \draw[red,line width=1mm] (0,0.5) -- (0,0.46667); 
        \draw[red,line width=1mm] (0,0.46667) -- (0,0.4375); 
        \draw[red,line width=1mm] (0,0.4375) -- (0,0.41176); 
        \draw[red,line width=1mm] (0,0.41176) -- (0,0.38889); 
        \draw[red,line width=1mm] (0,0.38889) -- (0,0.36842); 
        \draw[red,line width=1mm] (0,0.36842) -- (0,0.35); 
        \draw[red,line width=1mm] (0,0.35) -- (0,0.33333); 
        \draw[red,line width=1mm] (0,0.33333) -- (0,0.31818); 
        \draw[red,line width=1mm] (0,0.31818) -- (0,0.30435); 
        \draw[red,line width=1mm] (0,0.30435) -- (0,0.29167); 
        \draw[red,line width=1mm] (0,0.29167) -- (0,0.28); 
        \draw[red,line width=1mm] (0,0.28) -- (0,0.26923); 
        \draw[red,line width=1mm] (0,0.26923) -- (0,0.25926); 
        \draw[red,line width=1mm] (0,0.25926) -- (0,0.25); 
        \draw[red,line width=1mm] (0,0.25) -- (0,0.24138); 
        \draw[red,line width=1mm] (0,0.24138) -- (0,0.23333); 
        \draw[red,line width=1mm] (0,0.23333) -- (0,0.22581); 
        \draw[red,line width=1mm] (0,0.22581) -- (0,0.21875); 
        \draw[red,line width=1mm] (0,0.21875) -- (0,0.21212); 
        \draw[red,line width=1mm] (0,0.21212) -- (0,0.20588); 
        \draw[red,line width=1mm] (0,0.20588) -- (0,0.2); 
        \draw[red,line width=1mm] (0,0.2) -- (0,0.19444); 
        \draw[red,line width=1mm] (0,0.19444) -- (0,0.18919); 
        \draw[red,line width=1mm] (0,0.18919) -- (0,0.18421); 
        \draw[red,line width=1mm] (0,0.18421) -- (0,0.17949); 
        \draw[red,line width=1mm] (0,0.17949) -- (0,0.175); 
        \draw[red,line width=1mm] (0,0.175) -- (0,0.17073); 
        \draw[red,line width=1mm] (0,0.17073) -- (0,0.16667); 
        \draw[red,line width=1mm] (0,0.16667) -- (0,0.16279); 
        \draw[red,line width=1mm] (0,0.16279) -- (0,0.15909); 
        \draw[red,line width=1mm] (0,0.15909) -- (0,0.15556); 
        \draw[red,line width=1mm] (0,0.15556) -- (0,0.15217); 
        \draw[red,line width=1mm] (0,0.15217) -- (0,0.14894); 
        \draw[red,line width=1mm] (0,0.14894) -- (0,0.14583); 
        \draw[red,line width=1mm] (0,0.14583) -- (0,0.14286); 
        \draw[red,line width=1mm] (0,0.14286) -- (0,0.14); 
        \draw[red,line width=1mm] (0,0.14) -- (0,0.13725); 
        \draw[red,line width=1mm] (0,0.13725) -- (0,0.13462); 
        \draw[red,line width=1mm] (0,0.13462) -- (0,0.13208); 
        \draw[red,line width=1mm] (0,0.13208) -- (0,0.12963); 
        \draw[red,line width=1mm] (0,0.12963) -- (0,0.12727); 
        \draw[red,line width=1mm] (0,0.12727) -- (0,0.125); 
        \draw[red,line width=1mm] (0,0.125) -- (0,0.12281); 
        \draw[red,line width=1mm] (0,0) -- (0,0.125); 
\end{scope}

\foreach \x in {1,...,3}
{ \draw[very thick]  (\x,0 ) +(0,5pt) -- +(0,-5pt) node[anchor=north] {$ \x  $}; }

\foreach \y in {1,...,7}
{ \draw[very thick]  (0, \y ) +(5pt,0) -- +(-5pt,0) node[anchor=east] {$ \y  $}; }

\draw [->,very thick] (0,0) -- (3.8,0) node[anchor=north west] {$x$};
\draw [->,very thick] (0,0) -- (0,7.8) node[anchor=south] {$z=\frac{2m}{y}$};

\end{tikzpicture}
\label{fig:tableauProjection} } \caption{ %
\protect\subref{fig:tableauFrench} 
The French convention for drawing tableaux.
An example of an infinite standard Young tableau $\tableau$ sampled with the Plancherel
distribution. The highlighted boxes form a~bumping route obtained by adding 
the~entry $m+\nicefrac{1}{2}$ for \mbox{$m=3$}. The thick red line is the corresponding
plot of the~function 
$x(y) = \tableau_{\leftsquigarrow m+\nicefrac{1}{2}}(\lfloor y \rfloor)$.
\linebreak 
\protect\subref{fig:tableauProjection}~The same tableau shown in the~projective
coordinates $Oxz$ with $z=\frac{2m}{y}$. The thick red line is the plot of the~function 
$x(z)= \tableau_{\leftsquigarrow m+\nicefrac{1}{2}}(\lfloor \frac{2m}{z} \rfloor)$.
}
\end{figure}

%% file: figures-Poisson/multiple_bumping_routes/multiple_routes-linear.tex
\newcommand{\dir}[1]{figures-Poisson/multiple_bumping_routes/#1}
\newcommand{\file}[1]{file {root/figures-Poisson/multiple_bumping_routes/#1}}

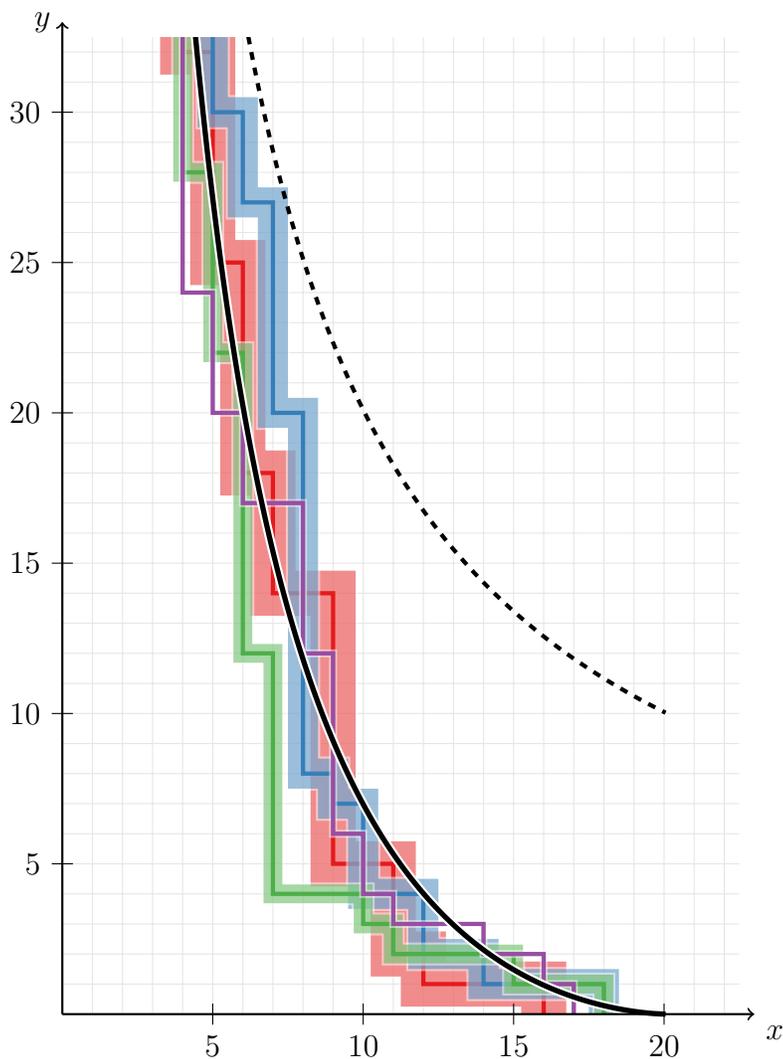
\begin{figure}
\centering
\begin{tikzpicture}[scale=0.4]
\definecolor{Set1-9-1}{RGB}{228,26,28}
\definecolor{Set1-9-A}{RGB}{228,26,28}
\definecolor{Set1-9-2}{RGB}{55,126,184}
\definecolor{Set1-9-B}{RGB}{55,126,184}
\definecolor{Set1-9-3}{RGB}{77,175,74}
\definecolor{Set1-9-C}{RGB}{77,175,74}
\definecolor{Set1-9-4}{RGB}{152,78,163}
\definecolor{Set1-9-D}{RGB}{152,78,163}
\definecolor{Set1-9-5}{RGB}{255,127,0}
\definecolor{Set1-9-E}{RGB}{255,127,0}
\definecolor{Set1-9-6}{RGB}{255,255,51}
\definecolor{Set1-9-F}{RGB}{255,255,51}
\definecolor{Set1-9-7}{RGB}{166,86,40}
\definecolor{Set1-9-G}{RGB}{166,86,40}
\definecolor{Set1-9-8}{RGB}{247,129,191}
\definecolor{Set1-9-H}{RGB}{247,129,191}
\definecolor{Set1-9-9}{RGB}{153,153,153}
\definecolor{Set1-9-I}{RGB}{153,153,153} 
     
\begin{scope}
    \clip (-2,-2) rectangle (30,32.5);
    \draw[black!10] (0,0) grid (22.5,32.5);
    \draw[white,thick,double distance=6mm,double=Set1-9-1,opacity=0.5] 
    plot[scale=1] file{\dir{100.5-linear-A.txt}};
    \draw[ultra thick,Set1-9-1] 
    plot[scale=1] file{\dir{100.5-linear-A.txt}};
    
    \draw[white,thick,double distance=4mm,double=Set1-9-2,opacity=0.5]
    plot[scale=1] file{\dir{100.5-linear-B.txt}};    
    \draw[ultra thick,Set1-9-2]
    plot[scale=1] file{\dir{100.5-linear-B.txt}};    
    
    \draw[white,thick,double distance=2.5mm,double=Set1-9-3,opacity=0.5]
    plot[scale=1] file{\dir{100.5-linear-C.txt}};    
    \draw[ultra thick,Set1-9-3]
    plot[scale=1] file{\dir{100.5-linear-C.txt}};    
    
    \draw[white,thick,double distance=0.5mm,double=Set1-9-4,opacity=0.5]
    plot[scale=1] file{\dir{100.5-linear-D.txt}};   
    \draw[ultra thick,Set1-9-4]
    plot[scale=1] file{\dir{100.5-linear-D.txt}};   
    
    \draw[white,thick,double distance=2pt,double=black] plot[smooth,scale=sqrt(100.5)] 
    file {\dir{bumpingcurve.txt}};
    
    \draw[black,ultra thick,dashed] plot[smooth,scale=sqrt(100.5)]  file {\dir{2x-res.txt}}; 
    
\end{scope}

    \draw[thick,->] (0,0) -- (23,0) node[anchor=north west]{$x$};
    \draw[thick,->] (0,0) -- (0,33)   node[anchor=east]{$y$};
    
    \foreach \x in {5,10,15,20}
    { \draw  (\x,0 ) +(0,10pt) -- +(0,-10pt) node[anchor=north] {$ \x  $}; }

    \foreach \y in {5,10,15,20,25,30}
    { \draw  (0,\y) +(10pt,0) -- +(-10pt,0) node[anchor=east] {$ \y  $}; }

\end{tikzpicture}

\caption{Four sample bumping routes corresponding to an~insertion
    $\tableau\leftarrow m+\nicefrac{1}{2}$ for $m=100$ and a random infinite standard
    Young tableau $\tableau$ which was sampled according to the Plancherel measure.
    In order to improve visibility, each bumping route is visualized as the plot of
    the~corresponding function $y\mapsto \tableau_{\leftsquigarrow
        m+\nicefrac{1}{2}}(\lfloor y \rfloor)$, cf.~\cref{fig:tableauFrench}, and not as
    a~collection of boxes. Colour and thickness were added in order to help identify
    the routes. The solid line is the (rescaled) \emph{limit bumping curve}. The
    dashed line is the~hyperbola $x y = 2m$, cf.~Equation~\eqref{eq:hyperbola}.}
\label{fig:simulated-linear}
\end{figure}

%% file: figures-Poisson/multiple_bumping_routes/multiple_routes-log.tex
\newcommand{\dir}[1]{figures-Poisson/multiple_bumping_routes/#1}

\begin{figure}
\centering
\begin{tikzpicture}[xscale=0.4,yscale=3]
\definecolor{Set1-9-1}{RGB}{228,26,28}
\definecolor{Set1-9-A}{RGB}{228,26,28}
\definecolor{Set1-9-2}{RGB}{55,126,184}
\definecolor{Set1-9-B}{RGB}{55,126,184}
\definecolor{Set1-9-3}{RGB}{77,175,74}
\definecolor{Set1-9-C}{RGB}{77,175,74}
\definecolor{Set1-9-4}{RGB}{152,78,163}
\definecolor{Set1-9-D}{RGB}{152,78,163}
\definecolor{Set1-9-5}{RGB}{255,127,0}
\definecolor{Set1-9-E}{RGB}{255,127,0}
\definecolor{Set1-9-6}{RGB}{255,255,51}
\definecolor{Set1-9-F}{RGB}{255,255,51}
\definecolor{Set1-9-7}{RGB}{166,86,40}
\definecolor{Set1-9-G}{RGB}{166,86,40}
\definecolor{Set1-9-8}{RGB}{247,129,191}
\definecolor{Set1-9-H}{RGB}{247,129,191}
\definecolor{Set1-9-9}{RGB}{153,153,153}
\definecolor{Set1-9-I}{RGB}{153,153,153}      
    \clip (-3,-0.5) rectangle (25,3.5);
    \draw[black!10] (0,-2) grid[ystep=1] (22,10);

\draw[white,thick,double distance=6mm,double=Set1-9-1,opacity=0.5] 
plot[scale=1] file {\dir{100.5-log-A.txt}};
\draw[ultra thick,Set1-9-1] 
plot[scale=1] file {\dir{100.5-log-A.txt}};

\draw[white,thick,double distance=4mm,double=Set1-9-2,opacity=0.5]
plot[scale=1] file {\dir{100.5-log-B.txt}};    
\draw[ultra thick,Set1-9-2]
plot[scale=1] file {\dir{100.5-log-B.txt}};    

\draw[white,thick,double distance=2.5mm,double=Set1-9-3,opacity=0.5]
plot[scale=1] file {\dir{100.5-log-C.txt}};    
\draw[ultra thick,Set1-9-3]
plot[scale=1] file {\dir{100.5-log-C.txt}};    

\draw[white,thick,double distance=0.5mm,double=Set1-9-4,opacity=0.5]
plot[scale=1] file {\dir{100.5-log-D.txt}};   
\draw[ultra thick,Set1-9-4]
plot[scale=1] file {\dir{100.5-log-D.txt}};

    \draw[thick,->] (0,0) -- (22.5,0) node[anchor=north west]{$x$};
    \draw[thick,->] (0,-3) -- (0,3.3)   node[anchor=east]{$y$};
    
    \foreach \x in {5,10,15,20}
    { \draw  (\x,0 ) +(0,1pt) -- +(0,-1pt) node[anchor=north] {$ \x  $}; }

    \foreach \y/\yname in {0/1,1/10,2/100,3/1000}
    { \draw  (0,\y) +(10pt,0) -- +(-10pt,0) node[anchor=east] {$ \yname  $}; }

\draw[white,thick,double distance=2pt,double=black,shift={(0, log10(sqrt(100.5)))}] plot[smooth,xscale=sqrt(100.5),yscale=1] file {\dir{bumping-log.txt}};

\draw[white,thick,double distance=2pt,double=black,dashed,shift={(0,log10(sqrt(100.5)))}] plot[smooth,xscale=sqrt(100.5),yscale=1] file {\dir{2x-log-res.txt}}; 

\end{tikzpicture}

\caption{The content of \cref{fig:simulated-linear} shown with the axis $y$ in the logarithmic scale.}
\label{fig:simulated-log}

\end{figure}
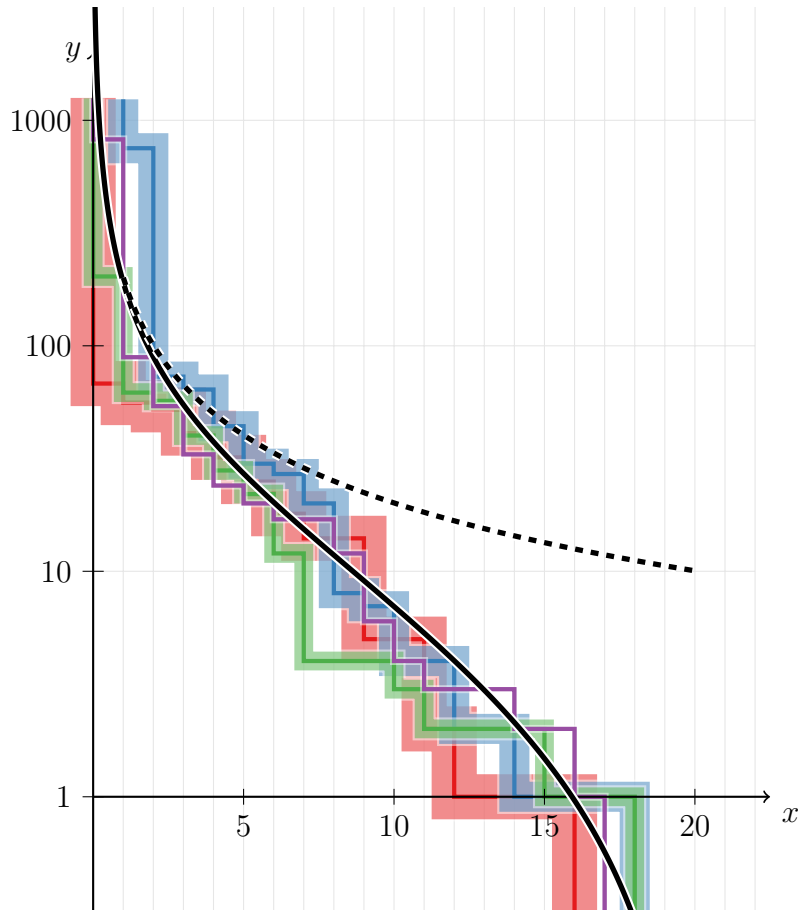

%% file: figures-Poisson/multiple_bumping_routes/multiple_routes-log-stretch-2.tex
\newcommand{\dir}[1]{figures-Poisson/multiple_bumping_routes/#1}

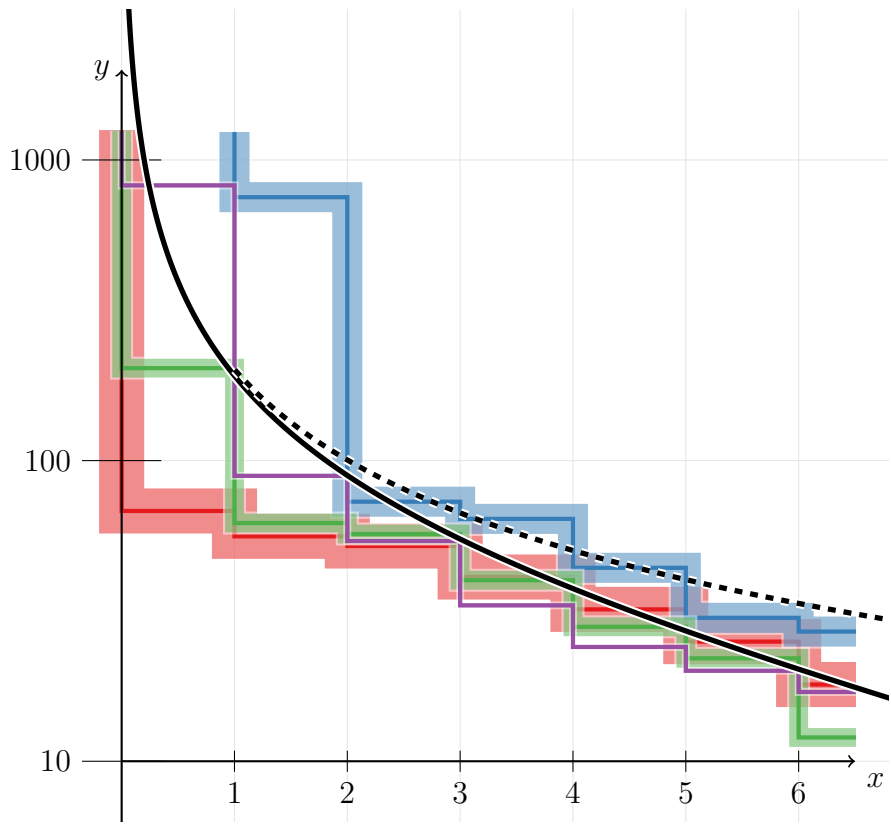
\begin{figure}
\centering
\begin{tikzpicture}[xscale=1.5,yscale=4]
\definecolor{Set1-9-1}{RGB}{228,26,28}
\definecolor{Set1-9-A}{RGB}{228,26,28}
\definecolor{Set1-9-2}{RGB}{55,126,184}
\definecolor{Set1-9-B}{RGB}{55,126,184}
\definecolor{Set1-9-3}{RGB}{77,175,74}
\definecolor{Set1-9-C}{RGB}{77,175,74}
\definecolor{Set1-9-4}{RGB}{152,78,163}
\definecolor{Set1-9-D}{RGB}{152,78,163}
\definecolor{Set1-9-5}{RGB}{255,127,0}
\definecolor{Set1-9-E}{RGB}{255,127,0}
\definecolor{Set1-9-6}{RGB}{255,255,51}
\definecolor{Set1-9-F}{RGB}{255,255,51}
\definecolor{Set1-9-7}{RGB}{166,86,40}
\definecolor{Set1-9-G}{RGB}{166,86,40}
\definecolor{Set1-9-8}{RGB}{247,129,191}
\definecolor{Set1-9-H}{RGB}{247,129,191}
\definecolor{Set1-9-9}{RGB}{153,153,153}
\definecolor{Set1-9-I}{RGB}{153,153,153}      
    \clip (-1,0.8) rectangle (6.8,3.5);
    \draw[black!10] (0,-2) grid (10,4);
    
\begin{scope}
\clip(-0.5,1) rectangle (6.5,3.3);    
\draw[white,thick,double distance=6mm,double=Set1-9-1,opacity=0.5] plot[scale=1] file {\dir{100.5-log-A.txt}};
\draw[ultra thick,Set1-9-1] plot[scale=1] file {\dir{100.5-log-A.txt}};

\draw[white,thick,double distance=4mm,double=Set1-9-2,opacity=0.5] plot[scale=1] file {\dir{100.5-log-B.txt}};    
\draw[ultra thick,Set1-9-2] plot[scale=1] file {\dir{100.5-log-B.txt}};    

\draw[white,thick,double distance=2.5mm,double=Set1-9-3,opacity=0.5] plot[scale=1] file {\dir{100.5-log-C.txt}};    
\draw[ultra thick,Set1-9-3] plot[scale=1] file {\dir{100.5-log-C.txt}};    

\draw[white,thick,double distance=0.5mm,double=Set1-9-4,opacity=0.5] plot[scale=1] file {\dir{100.5-log-D.txt}};   
\draw[ultra thick,Set1-9-4] plot[scale=1] file {\dir{100.5-log-D.txt}};   
\end{scope}

\draw[thick,->] (0,1) -- (6.5,1) node[anchor=north west]{$x$};
\draw[thick,->] (0,-3) -- (0,3.3)   node[anchor=east]{$y$};

\foreach \x in {1,...,6}
{ \draw  (\x,1) +(0,1pt) -- +(0,-1pt) node[anchor=north] {$ \x  $}; }

\foreach \y/\yname in {0/1,1/10,2/100,3/1000}
{ \draw  (0,\y) +(10pt,0) -- +(-10pt,0) node[anchor=east] {$ \yname  $}; }

\draw[white,thick,double distance=2pt,double=black,shift={(0,log10(sqrt(100.5)))}]  plot[smooth,xscale=sqrt(100.5),yscale=1] file {\dir{bumping-log.txt}};

\draw[white,thick,double distance=2pt,double=black,dashed,shift={(0, log10(sqrt(100.5)))}] plot[smooth,xscale=sqrt(100.5),yscale=1]  file {\dir{2x-log-res.txt}}; 

\end{tikzpicture}

\caption{Zoom on a part of \cref{fig:simulated-log}. In this kind of scaling
    when $x=O(1)$ and $y\gg \sqrt{m}$ the results of Romik and {\Sniady}
    \cite{RomikSniadyBumping} are \emph{not} applicable and the~rescaled limit bumping curve
    (the solid line) is shown for illustration purposes only.}
\label{fig:simulated-log-stretch-2}

\end{figure}

%% file: figures-Poisson/CDF/cdf.tex
\newcommand{\dir}[1]{figures-Poisson/CDF/#1}

\begin{figure}
    \centering
    \begin{tikzpicture}[scale=8]
    \definecolor{Set1-9-1}{RGB}{228,26,28}
    \definecolor{Set1-9-A}{RGB}{228,26,28}
    \definecolor{Set1-9-2}{RGB}{55,126,184}
    \definecolor{Set1-9-B}{RGB}{55,126,184}
    \definecolor{Set1-9-3}{RGB}{77,175,74}
    \definecolor{Set1-9-C}{RGB}{77,175,74}
    \definecolor{Set1-9-4}{RGB}{152,78,163}
    \definecolor{Set1-9-D}{RGB}{152,78,163}
    \definecolor{Set1-9-5}{RGB}{255,127,0}
    \definecolor{Set1-9-E}{RGB}{255,127,0}
    \definecolor{Set1-9-6}{RGB}{255,255,51}
    \definecolor{Set1-9-F}{RGB}{255,255,51}
    \definecolor{Set1-9-7}{RGB}{166,86,40}
    \definecolor{Set1-9-G}{RGB}{166,86,40}
    \definecolor{Set1-9-8}{RGB}{247,129,191}
    \definecolor{Set1-9-H}{RGB}{247,129,191}
    \definecolor{Set1-9-9}{RGB}{153,153,153}
    \definecolor{Set1-9-I}{RGB}{153,153,153}      

    \draw[black!10] (0,0) grid[step=0.1] (1,1);
    
    \draw[Set1-9-1,line width=2.2pt] plot[scale=1]            file {\dir{1.5-CDF-log.txt}};

    \draw[white,double=Set1-9-2,line width=0.3pt,double distance=1.4pt] plot[scale=1]            file {\dir{6.5-CDF-log.txt}};
        
    \draw[white,double=Set1-9-3,line width=0.4pt,double distance=0.4pt] plot[scale=1]            file {\dir{25.5-CDF-log.txt}};

    \draw[dashed] (0,0) -- (1,1);

    \draw[thick,->] (0,0) -- (1.1,0)    node[anchor=north west]{$u$};
    \draw[thick,->] (0,0) -- (0,1.1)    node[anchor=east]{$\mathbb{P}\left(e^{-\frac{2m}{Y_0}}\leq u\right)$};
    
    \foreach \x in {0.2,0.4,0.6,0.8,1}
{ \draw  (\x,0 ) +(0,1pt) -- +(0,-1pt) node[anchor=north] {$ \x  $}; }

\foreach \y in {0.2,0.4,0.6,0.8,1}
{ \draw  (0,\y) +(1pt,0) -- +(-1pt,0) node[anchor=east] {$ \y  $}; }

    \end{tikzpicture}
    
    \caption{Monte Carlo simulations of the cumulative probability distribution
    function for the random variable $e^{-\frac{2m}{Y_0}}$. The~thick red line
    corresponds to $m=1$ (sample size is equal to $10\ 000$); the blue
    line corresponds to $m=6$ (sample size $3\ 000$); the thin green line
    corresponds to $m=25$ (sample size $2\ 500$). The dashed line corresponds to
    the cumulative probability distribution function of the uniform
    distribution $U(0,1)$ on the unit interval. Due to constraints on
    computation time it was not possible to get Monte Carlo data for all values of $u$. 
		The staircase feature of the plots is due to the~discrete nature of the probability distribution of $Y_0$.}
    
    \label{fig:cdf}
\end{figure}
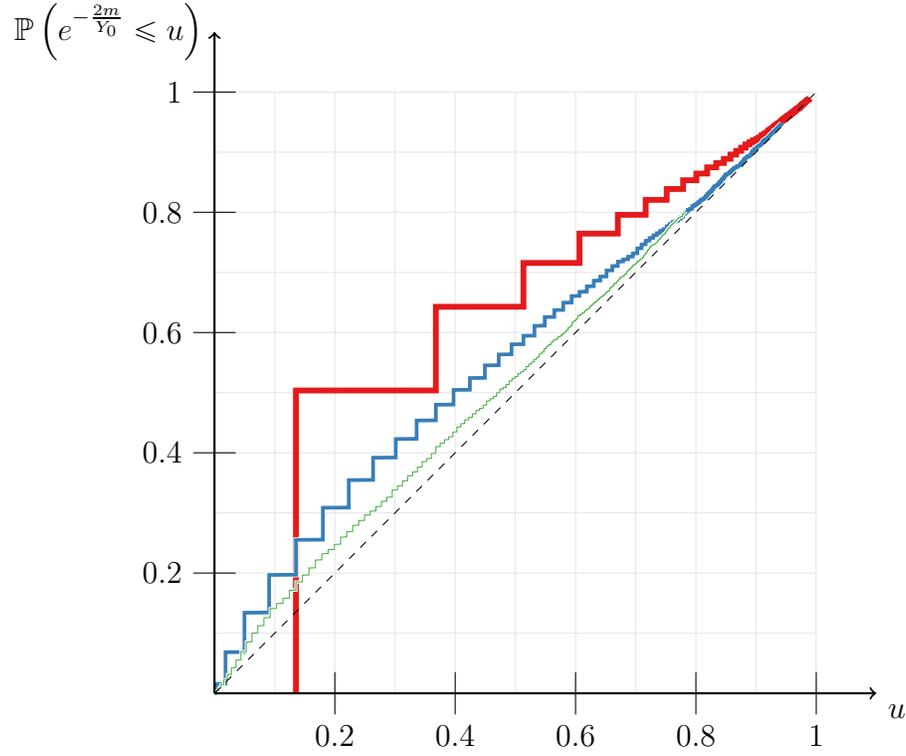

%% file: figures-Poisson/multiple_bumping_routes/multiple_routes-projective.tex
\newcommand{\dir}[1]{figures-Poisson/multiple_bumping_routes/#1}

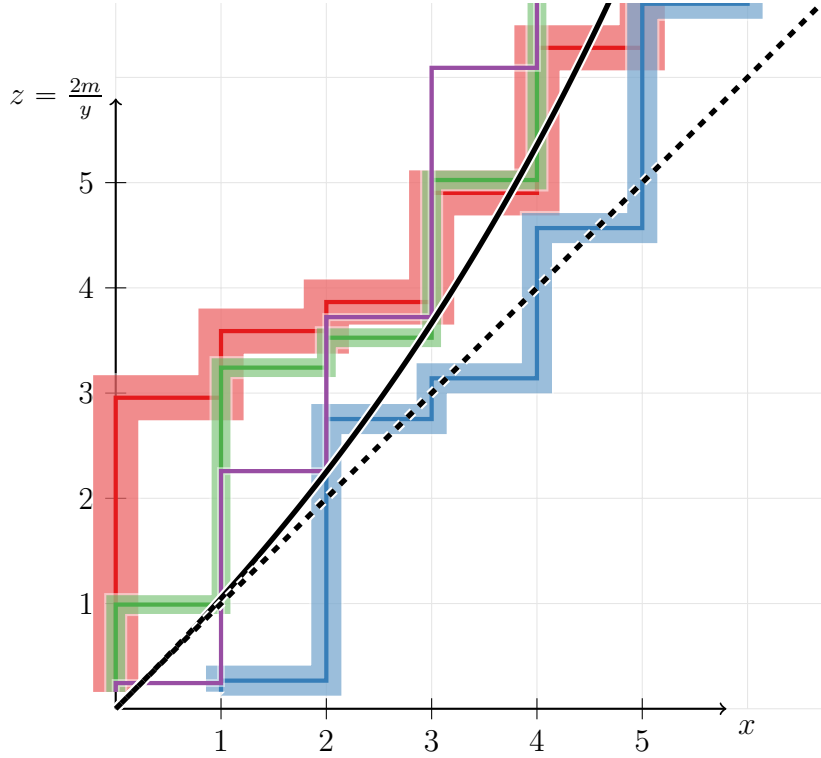
\begin{figure}
\centering
\begin{tikzpicture}[scale=1.4]
\definecolor{Set1-9-1}{RGB}{228,26,28}
\definecolor{Set1-9-A}{RGB}{228,26,28}
\definecolor{Set1-9-2}{RGB}{55,126,184}
\definecolor{Set1-9-B}{RGB}{55,126,184}
\definecolor{Set1-9-3}{RGB}{77,175,74}
\definecolor{Set1-9-C}{RGB}{77,175,74}
\definecolor{Set1-9-4}{RGB}{152,78,163}
\definecolor{Set1-9-D}{RGB}{152,78,163}
\definecolor{Set1-9-5}{RGB}{255,127,0}
\definecolor{Set1-9-E}{RGB}{255,127,0}
\definecolor{Set1-9-6}{RGB}{255,255,51}
\definecolor{Set1-9-F}{RGB}{255,255,51}
\definecolor{Set1-9-7}{RGB}{166,86,40}
\definecolor{Set1-9-G}{RGB}{166,86,40}
\definecolor{Set1-9-8}{RGB}{247,129,191}
\definecolor{Set1-9-H}{RGB}{247,129,191}
\definecolor{Set1-9-9}{RGB}{153,153,153}
\definecolor{Set1-9-I}{RGB}{153,153,153}      
    \clip (-1,-1) rectangle (6.7,6.7);
    \draw[black!10] (0,0) grid (30,10);

    \draw[thick,->] (0,0) -- (5.8,0) node[anchor=north west]{$x$};
    \draw[thick,->] (0,0) -- (0,5.8) node[anchor=east]{$z=\frac{2m}{y}$};
    \draw (1,0) +(0,0.1) -- +(0,-0.1) node[anchor=north]{$1$};
    \draw (2,0) +(0,0.1) -- +(0,-0.1) node[anchor=north]{$2$};
    \draw (3,0) +(0,0.1) -- +(0,-0.1) node[anchor=north]{$3$};
    \draw (4,0) +(0,0.1) -- +(0,-0.1) node[anchor=north]{$4$};
    \draw (5,0) +(0,0.1) -- +(0,-0.1) node[anchor=north]{$5$};
    \draw (0,1) +(0.1,0) -- +(-0.1,0) node[anchor=east]{$1$};
    \draw (0,2) +(0.1,0) -- +(-0.1,0) node[anchor=east]{$2$};
    \draw (0,3) +(0.1,0) -- +(-0.1,0) node[anchor=east]{$3$};
    \draw (0,4) +(0.1,0) -- +(-0.1,0) node[anchor=east]{$4$};
    \draw (0,5) +(0.1,0) -- +(-0.1,0) node[anchor=east]{$5$};

\draw[white,thick,double distance=6mm,double=Set1-9-1,opacity=0.5] plot[scale=1] file {\dir{100.5-projective-A.txt}};
\draw[ultra thick,Set1-9-1] plot[scale=1] file {\dir{100.5-projective-A.txt}};

\draw[white,thick,double distance=4mm,double=Set1-9-2,opacity=0.5] plot[scale=1] file {\dir{100.5-projective-B.txt}};    
\draw[ultra thick,Set1-9-2] plot[scale=1] file {\dir{100.5-projective-B.txt}};    

\draw[white,thick,double distance=2.5mm,double=Set1-9-3,opacity=0.5] plot[scale=1] file {\dir{100.5-projective-C.txt}};    
\draw[ultra thick,Set1-9-3] plot[scale=1] file {\dir{100.5-projective-C.txt}};    

\draw[white,thick,double distance=0.5mm,double=Set1-9-4,opacity=0.5] plot[scale=1] file {\dir{100.5-projective-D.txt}};   
\draw[ultra thick,Set1-9-4] plot[scale=1] file {\dir{100.5-projective-D.txt}};

\draw[white,thick,double distance=2pt,double=black] plot[smooth,scale=sqrt(100.5)] file {\dir{bumping-projective.txt}};

\draw[white,thick,double distance=2pt,double=black,dashed] 
(0,0) -- (10,10);

\end{tikzpicture}

\caption{The bumping routes from \cref{fig:simulated-linear} shown in 
		the~projective convention (see also \cref{fig:tableauProjection}). 
		The dashed line $x=z$ corresponds to the hyperbola \eqref{eq:hyperbola};
    it is tangent in $0$ to the rescaled limit bumping curve (the solid line); it is also
    the~plot of the mean value of the Poisson process $z\mapsto \mathbb{E} N(z)$.}
\label{fig:projective}

\end{figure}

%% file: figures-Poisson/zlepianie/zlep-log.tex
\newcommand{\dir}[1]{figures-Poisson/zlepianie/#1}

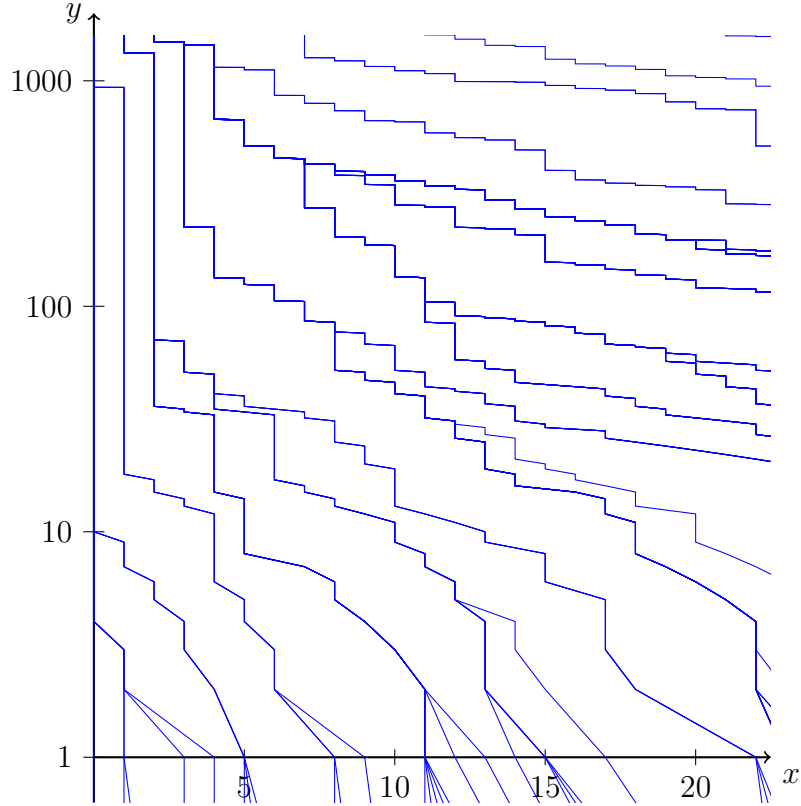
\begin{figure}
\centering
\begin{tikzpicture}[xscale=0.4,yscale=3]

    \draw[thick,->] (0,0) -- (22.5,0) node[anchor=north west]{$x$};
    \draw[thick,->] (0,-0.2) -- (0,3.3)   node[anchor=east]{$y$};
    
    \foreach \x in {5,10,15,20}
    { \draw  (\x,0 ) +(0,1pt) -- +(0,-1pt) node[anchor=north] {$ \x  $}; }

    \foreach \y/\yname in {0/1,1/10,2/100,3/1000}
    { \draw  (0,\y) +(10pt,0) -- +(-10pt,0) node[anchor=east] {$ \yname  $}; }
    
\clip (-0.1,-0.2) rectangle (22.5,3.2);

\foreach \x in {0,...,100,120,140,160,180,200,270}
{\draw[blue] plot  file {\dir{log\x.txt}}; }
\draw[blue] (0,-2) -- (0,1);

\end{tikzpicture}

\caption{All possible bumping routes (\emph{``the bumping tree''}) for a given
    Plancherel-distributed random infinite standard Young tableau. The $y$ axis is
    shown using the logarithmic scale. In order to improve visibility, each bumping
    route was drawn as a piecewise-affine function connecting the points~\eqref{eq:bumping_points} 
		and \emph{not} as a jump function as in \cref{fig:tableauFrench}. }

\label{fig:all-bumping} 

\end{figure}

%% file: figures-Poisson/augmented-Young-diagram.tex
\begin{figure}
\centering
\subfloat[]{
    \begin{tikzpicture}
    \coordinate (central) at (0,0);
    ;
    \node at (0.5,0.5) {16}; 
    \node at (1.5,0.5) {37}; 
    \node at (2.5,0.5) {41}; 
    \node at (3.5,0.5) {82}; 
    \node at (0.5,1.5) {23}; 
    \node at (1.5,1.5) {53}; 
    \node at (2.5,1.5) {$\infty$}; 
    \node at (0.5,2.5) {74}; 
     \draw[thick] (central) ++(0,0)
    {[current point is local] ++(0,0) rectangle +(1,1)}
    {[current point is local] ++(1,0) rectangle +(1,1)}
    {[current point is local] ++(2,0) rectangle +(1,1)}
    {[current point is local] ++(3,0) rectangle +(1,1)}
    {[current point is local] ++(0,1) rectangle +(1,1)}
    {[current point is local] ++(1,1) rectangle +(1,1)}
    {[current point is local] ++(0,2) rectangle +(1,1)}
    {[current point is local] ++(2,1) rectangle +(1,1)} 
    ;
\end{tikzpicture}
\label{fig:tableau-before}
}
\hfill
\subfloat[]{
  \begin{tikzpicture}
    \coordinate (central) at (0,0);
    \draw[dotted,thick] (central) ++(0,0) 
    {[current point is local] ++(2,1) rectangle +(1,1)}
    ;
    \draw[blue,opacity=0.5,thick] (central) ++(0,0) 
    ++(2,1) +(0.1,0.1) -- +(0.9,0.9) +(0.1,0.9) -- +(0.9,0.1) 
    ;
    \draw[thick] (central) ++(0,0) 
    {[current point is local] ++(0,0) rectangle +(1,1)}
    {[current point is local] ++(1,0) rectangle +(1,1)}
    {[current point is local] ++(2,0) rectangle +(1,1)}
    {[current point is local] ++(3,0) rectangle +(1,1)}
    {[current point is local] ++(0,1) rectangle +(1,1)}
    {[current point is local] ++(1,1) rectangle +(1,1)}
    {[current point is local] ++(0,2) rectangle +(1,1)}
    ;
\end{tikzpicture}
\label{fig:tableau-after} } \caption{ %
\protect\subref{fig:tableau-before} Example of a tableau $\tableau$ which has exactly
one entry equal to $\infty$. \protect\subref{fig:tableau-after} The augmented
shape $\operatorname{sh}^* \tableau=(\lambda,\Box)$ of the tableau $\tableau$. The regular part 
$\lambda=(4,2,1)$ is shown as the Young diagram drawn with solid lines, the position of the special box
$\Box=(x_\Box,y_\Box)=(2,1)$ is marked as the decorated box drawn with the dotted lines.
With the notations of \cref{sec:inclusion-x} this augmented Young diagram corresponds to $(x_\Box,\lambda)=(2,\lambda)\in \N_0\times \diagrams$.}
\label{fig:augmented-yd}
\end{figure}
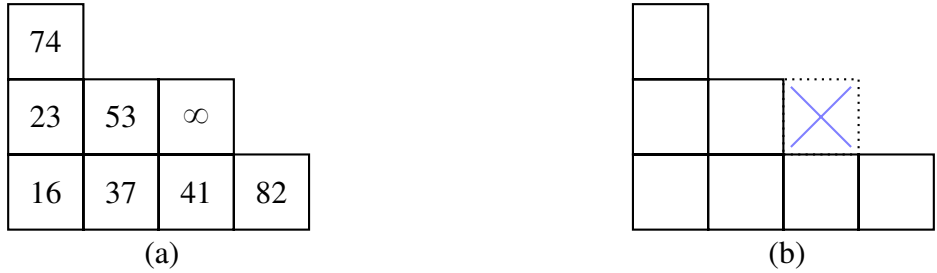

%% file: figures-Poisson/augmented-young-graph.tex
\newcommand{\SX}{15}

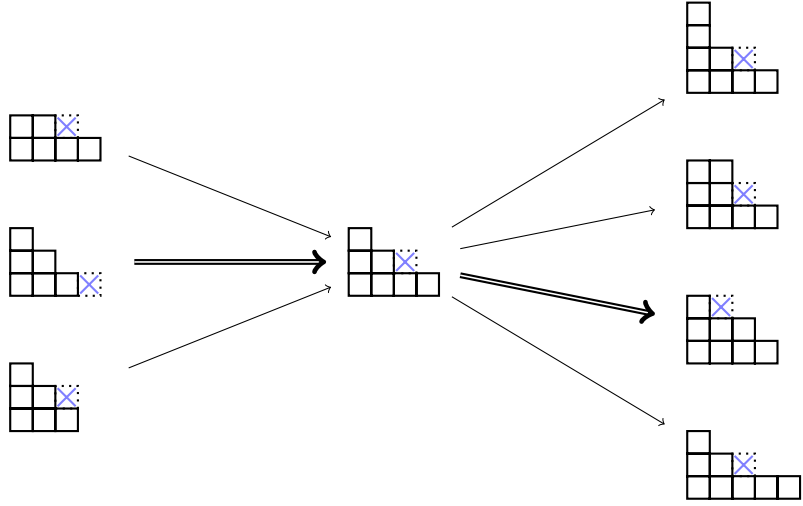
\begin{figure}
    \begin{tikzpicture}[scale=0.3]
    \coordinate (central) at (0,0);
    \draw[dotted,thick] (central) ++(-2,-1.5) 
    {[current point is local] ++(2,1) rectangle +(1,1)}
    ;
    \draw[blue,opacity=0.5,thick] (central) ++(-2,-1.5) 
    ++(2,1) +(0.1,0.1) -- +(0.9,0.9) +(0.1,0.9) -- +(0.9,0.1) 
    ;
    \draw[thick] (central) ++(-2,-1.5) 
    {[current point is local] ++(0,0) rectangle +(1,1)}
    {[current point is local] ++(1,0) rectangle +(1,1)}
    {[current point is local] ++(2,0) rectangle +(1,1)}
    {[current point is local] ++(3,0) rectangle +(1,1)}
    {[current point is local] ++(0,1) rectangle +(1,1)}
    {[current point is local] ++(1,1) rectangle +(1,1)}
    {[current point is local] ++(0,2) rectangle +(1,1)}
    ;

\coordinate (p1) at (\SX,-9);
\draw[dotted,thick] (p1) ++(-2,-1.5) 
{[current point is local] ++(2,1) rectangle +(1,1)}
;
\draw[blue,opacity=0.5,thick] (p1) ++(-2,-1.5) 
++(2,1) +(0.1,0.1) -- +(0.9,0.9) +(0.1,0.9) -- +(0.9,0.1) 
;
\draw[thick] (p1) ++(-2,-1.5) 
{[current point is local] ++(0,0) rectangle +(1,1)}
{[current point is local] ++(1,0) rectangle +(1,1)}
{[current point is local] ++(2,0) rectangle +(1,1)}
{[current point is local] ++(3,0) rectangle +(1,1)}
{[current point is local] ++(4,0) rectangle +(1,1)}
{[current point is local] ++(0,1) rectangle +(1,1)}
{[current point is local] ++(1,1) rectangle +(1,1)}
{[current point is local] ++(0,2) rectangle +(1,1)}
;

\coordinate (p2) at (\SX,-3);
\draw[dotted,thick] (p2) ++(-2,-1.5) 
{[current point is local] ++(1,2) rectangle +(1,1)}
;
\draw[blue,opacity=0.5,thick] (p2) ++(-2,-1.5) 
++(1,2) +(0.1,0.1) -- +(0.9,0.9) +(0.1,0.9) -- +(0.9,0.1) 
;
\draw[thick] (p2) ++(-2,-1.5) 
{[current point is local] ++(0,0) rectangle +(1,1)}
{[current point is local] ++(1,0) rectangle +(1,1)}
{[current point is local] ++(2,0) rectangle +(1,1)}
{[current point is local] ++(3,0) rectangle +(1,1)}
{[current point is local] ++(0,1) rectangle +(1,1)}
{[current point is local] ++(1,1) rectangle +(1,1)}
{[current point is local] ++(2,1) rectangle +(1,1)}
{[current point is local] ++(0,2) rectangle +(1,1)}
;

\coordinate (p3) at (\SX,3);
\draw[dotted,thick] (p3) ++(-2,-1.5) 
{[current point is local] ++(2,1) rectangle +(1,1)}
;
\draw[blue,opacity=0.5,thick] (p3) ++(-2,-1.5) 
++(2,1) +(0.1,0.1) -- +(0.9,0.9) +(0.1,0.9) -- +(0.9,0.1) 
;
\draw[thick] (p3) ++(-2,-1.5) 
{[current point is local] ++(0,0) rectangle +(1,1)}
{[current point is local] ++(1,0) rectangle +(1,1)}
{[current point is local] ++(2,0) rectangle +(1,1)}
{[current point is local] ++(3,0) rectangle +(1,1)}
{[current point is local] ++(0,1) rectangle +(1,1)}
{[current point is local] ++(1,1) rectangle +(1,1)}
{[current point is local] ++(0,2) rectangle +(1,1)}
{[current point is local] ++(1,2) rectangle +(1,1)}
;

\coordinate (p4) at (\SX,9);
\draw[dotted,thick] (p4) ++(-2,-1.5) 
{[current point is local] ++(2,1) rectangle +(1,1)}
;
\draw[blue,opacity=0.5,thick] (p4) ++(-2,-1.5) 
++(2,1) +(0.1,0.1) -- +(0.9,0.9) +(0.1,0.9) -- +(0.9,0.1) 
;
\draw[thick] (p4) ++(-2,-1.5) 
{[current point is local] ++(0,0) rectangle +(1,1)}
{[current point is local] ++(1,0) rectangle +(1,1)}
{[current point is local] ++(2,0) rectangle +(1,1)}
{[current point is local] ++(3,0) rectangle +(1,1)}
{[current point is local] ++(0,1) rectangle +(1,1)}
{[current point is local] ++(1,1) rectangle +(1,1)}
{[current point is local] ++(0,2) rectangle +(1,1)}
{[current point is local] ++(0,3) rectangle +(1,1)}
;

\draw[->] ($ (central)!3 cm!(p1) $) -- ($ (p1) !3.5 cm! (central)$) ;
\draw[->,double,thick] ($ (central)!3 cm!(p2) $) -- ($ (p2) !3.5 cm! (central)$) ;
\draw[->] ($ (central)!3 cm!(p3) $) -- ($ (p3) !3.5 cm! (central)$) ;
\draw[->] ($ (central)!3 cm!(p4) $) -- ($ (p4) !3.5 cm! (central)$) ;

\coordinate (q1) at (-\SX,-6);
\draw[dotted,thick] (q1) ++(-2,-1.5) 
{[current point is local] ++(2,1) rectangle +(1,1)}
;
\draw[blue,opacity=0.5,thick] (q1) ++(-2,-1.5) 
++(2,1) +(0.1,0.1) -- +(0.9,0.9) +(0.1,0.9) -- +(0.9,0.1) 
;
\draw[thick] (q1) ++(-2,-1.5) 
{[current point is local] ++(0,0) rectangle +(1,1)}
{[current point is local] ++(1,0) rectangle +(1,1)}
{[current point is local] ++(2,0) rectangle +(1,1)}
{[current point is local] ++(0,1) rectangle +(1,1)}
{[current point is local] ++(1,1) rectangle +(1,1)}
{[current point is local] ++(0,2) rectangle +(1,1)}
;

\coordinate (q2) at (-\SX,0);
\draw[dotted,thick] (q2) ++(-2,-1.5) 
{[current point is local] ++(3,0) rectangle +(1,1)}
;
\draw[blue,opacity=0.5,thick] (q2) ++(-2,-1.5) 
++(3,0) +(0.1,0.1) -- +(0.9,0.9) +(0.1,0.9) -- +(0.9,0.1) 
;
\draw[thick] (q2) ++(-2,-1.5) 
{[current point is local] ++(0,0) rectangle +(1,1)}
{[current point is local] ++(1,0) rectangle +(1,1)}
{[current point is local] ++(2,0) rectangle +(1,1)}
{[current point is local] ++(0,1) rectangle +(1,1)}
{[current point is local] ++(1,1) rectangle +(1,1)}
{[current point is local] ++(0,2) rectangle +(1,1)}
;

\coordinate (q3) at (-\SX,6);
\draw[dotted,thick] (q3) ++(-2,-1.5) 
{[current point is local] ++(2,1) rectangle +(1,1)}
;
\draw[blue,opacity=0.5,thick] (q3) ++(-2,-1.5) 
++(2,1) +(0.1,0.1) -- +(0.9,0.9) +(0.1,0.9) -- +(0.9,0.1) 
;
\draw[thick] (q3) ++(-2,-1.5) 
{[current point is local] ++(0,0) rectangle +(1,1)}
{[current point is local] ++(1,0) rectangle +(1,1)}
{[current point is local] ++(2,0) rectangle +(1,1)}
{[current point is local] ++(3,0) rectangle +(1,1)}
{[current point is local] ++(0,1) rectangle +(1,1)}
{[current point is local] ++(1,1) rectangle +(1,1)}
;

\draw[<-] ($ (central)!3 cm!(q1) $) -- ($ (q1) !3.5 cm! (central)$) ;
\draw[<-,double,thick] ($ (central)!3 cm!(q2) $) -- ($ (q2) !3.5 cm! (central)$) ;
\draw[<-] ($ (central)!3 cm!(q3) $) -- ($ (q3) !3.5 cm! (central)$) ;

    \end{tikzpicture}
    
    \caption{A part of the augmented Young graph. For each vertex
    $(\lambda,\Box)\in\Augmented$ the regular part $\lambda$ is drawn with the solid
    line and the special box $\Box$ is indicated as a decorated dotted square. For
    clarity this figure shows only the direct neighborhood of the augmented Young
    diagram $(\lambda,\Box)$ with $\lambda=(4,2,1)$ and $\Box=(2,1)$. The double
    thick arrows indicate the edges which are \emph{bumps}.}
    
    \label{fig:augmented-young-graph}
\end{figure}

%% file: figures-Poisson/2D-Poisson/2D-Poisson.tex
\newcommand{\dir}[1]{figures-Poisson/2D-Poisson/#1}

\begin{figure}
    \centering
    \begin{tikzpicture}[scale=2]
    \definecolor{Set1-9-1}{RGB}{228,26,28}
    \definecolor{Set1-9-A}{RGB}{228,26,28}
    \definecolor{Set1-9-2}{RGB}{55,126,184}
    \definecolor{Set1-9-B}{RGB}{55,126,184}
    \definecolor{Set1-9-3}{RGB}{77,175,74}
    \definecolor{Set1-9-C}{RGB}{77,175,74}
    \definecolor{Set1-9-4}{RGB}{152,78,163}
    \definecolor{Set1-9-D}{RGB}{152,78,163}
    \definecolor{Set1-9-5}{RGB}{255,127,0}
    \definecolor{Set1-9-E}{RGB}{255,127,0}
    \definecolor{Set1-9-6}{RGB}{255,255,51}
    \definecolor{Set1-9-F}{RGB}{255,255,51}
    \definecolor{Set1-9-7}{RGB}{166,86,40}
    \definecolor{Set1-9-G}{RGB}{166,86,40}
    \definecolor{Set1-9-8}{RGB}{247,129,191}
    \definecolor{Set1-9-H}{RGB}{247,129,191}
    \definecolor{Set1-9-9}{RGB}{153,153,153}
    \definecolor{Set1-9-I}{RGB}{153,153,153}   
    \clip(-0.5,-0.5) rectangle (3.55,3.55);
    \foreach \x in {1,...,3}
    { \draw[very thick]  (\x,0 ) +(0,2pt) -- +(0,-2pt) node[anchor=north] {$ \x  $}; }
    
    \foreach \y in {1,...,3}
    { \draw[very thick]  (0, \y ) +(2pt,0) -- +(-2pt,0) node[anchor=east] {$ \y  $}; }

    \draw [->,very thick] (0,0) -- (3.2,0) node[anchor=north west] {$s$};
    \draw [->,very thick] (0,0) -- (0,3.2) node[anchor=south east] {$t$};
    
    \clip (0,0.4) rectangle (50,50);

    \input{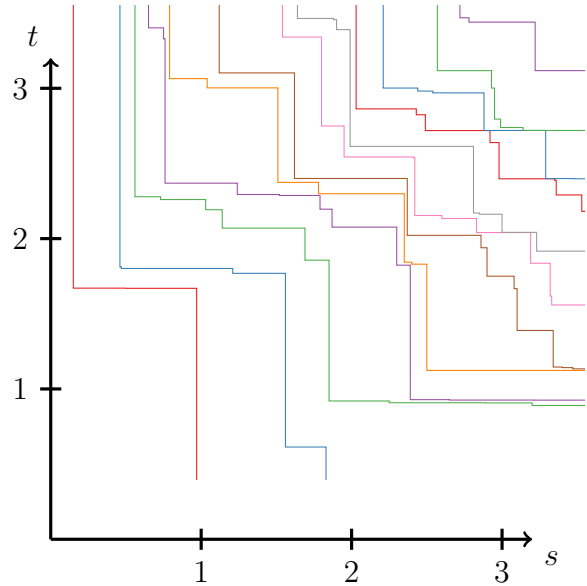}    
    \end{tikzpicture}
    
    \caption{Computer simulation of the level curves of the~function~\eqref{eq:2dPoisson} 
		for $m=100$. A part of the plot which corresponds to
    small values of $t$ was not shown due to restrictions on the computation
    time. }
    
    \label{fig:2dPoisson}
\end{figure}

\begin{figure}
    \centering
    \begin{tikzpicture}[scale=2]
    \definecolor{Set1-9-1}{RGB}{228,26,28}
    \definecolor{Set1-9-A}{RGB}{228,26,28}
    \definecolor{Set1-9-2}{RGB}{55,126,184}
    \definecolor{Set1-9-B}{RGB}{55,126,184}
    \definecolor{Set1-9-3}{RGB}{77,175,74}
    \definecolor{Set1-9-C}{RGB}{77,175,74}
    \definecolor{Set1-9-4}{RGB}{152,78,163}
    \definecolor{Set1-9-D}{RGB}{152,78,163}
    \definecolor{Set1-9-5}{RGB}{255,127,0}
    \definecolor{Set1-9-E}{RGB}{255,127,0}
    \definecolor{Set1-9-6}{RGB}{255,255,51}
    \definecolor{Set1-9-F}{RGB}{255,255,51}
    \definecolor{Set1-9-7}{RGB}{166,86,40}
    \definecolor{Set1-9-G}{RGB}{166,86,40}
    \definecolor{Set1-9-8}{RGB}{247,129,191}
    \definecolor{Set1-9-H}{RGB}{247,129,191}
    \definecolor{Set1-9-9}{RGB}{153,153,153}
    \definecolor{Set1-9-I}{RGB}{153,153,153}   
    \clip(-0.5,-0.5) rectangle (3.55,3.55);
    \foreach \x in {1,...,3}
    { \draw[very thick]  (\x,0 ) +(0,2pt) -- +(0,-2pt) node[anchor=north] {$ \x  $}; }
    
    \foreach \y in {1,...,3}
    { \draw[very thick]  (0, \y ) +(2pt,0) -- +(-2pt,0) node[anchor=east] {$ \y  $}; }

    \draw [->,very thick] (0,0) -- (3.2,0) node[anchor=north west] {$s$};
    \draw [->,very thick] (0,0) -- (0,3.2) node[anchor=south east] {$t$};
    
    \input{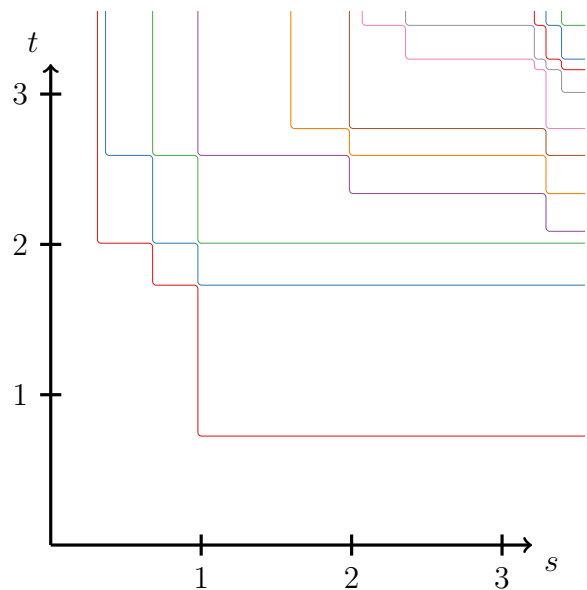}    
    \end{tikzpicture}
    
    \caption{Computer simulation of the level curves 
		of the~function~\eqref{eq:2dPoisson-true} 
		for Poisson point process.}
        \label{fig:2dPoissont}
 
\end{figure}

%% file: figures-Poisson/2D-Poisson/krzywe100.tex
\draw[Set1-9-1](0.0,50) -- (0.01,50) -- (0.01,100.0) -- (0.02,100.0) -- (0.02,70.71067811865474) -- (0.03,70.71067811865474) -- (0.03,6.726727939963125) -- (0.15,6.726727939963125) -- (0.15,1.6706158440387588) -- (0.5,1.6706158440387588) -- (0.5,1.6694514082354446) -- (0.97,1.6694514082354446) -- (0.97,0.3499971125357323) -- (3.7,0.3499971125357323); 
\draw[Set1-9-2](0.0,50) -- (0.04,50) -- (0.04,100.0) -- (0.09,100.0) -- (0.09,31.622776601683793) -- (0.12,31.622776601683793) -- (0.12,10.0) -- (0.19,10.0) -- (0.19,6.868028197434452) -- (0.41,6.868028197434452) -- (0.41,6.819943394704735) -- (0.44,6.819943394704735) -- (0.44,5.902813361009553) -- (0.46,5.902813361009553) -- (0.46,1.811606093861842) -- (0.47,1.811606093861842) -- (0.47,1.8015824844372499) -- (1.21,1.8015824844372499) -- (1.21,1.7691496414475842) -- (1.56,1.7691496414475842) -- (1.56,0.6135893161973001) -- (1.83,0.6135893161973001) -- (1.83,0.36684220189712957) -- (2.06,0.36684220189712957) -- (2.06,0.36103200473960506) -- (2.94,0.36103200473960506) -- (2.94,0.36074529043559017) -- (3.38,0.36074529043559017) -- (3.38,0.3607124324669062) -- (3.6,0.3607124324669062) -- (3.6,0.3499971125357323) -- (3.7,0.3499971125357323); 
\draw[Set1-9-3](0.0,50) -- (0.05,50) -- (0.05,100.0) -- (0.1,100.0) -- (0.1,57.73502691896258) -- (0.13,57.73502691896258) -- (0.13,31.622776601683793) -- (0.17,31.622776601683793) -- (0.17,24.253562503633297) -- (0.24,24.253562503633297) -- (0.24,14.14213562373095) -- (0.27,14.14213562373095) -- (0.27,10.0) -- (0.28,10.0) -- (0.28,9.805806756909202) -- (0.43,9.805806756909202) -- (0.43,7.53778361444409) -- (0.56,7.53778361444409) -- (0.56,2.2786197567488773) -- (0.73,2.2786197567488773) -- (0.73,2.2599230735278435) -- (1.03,2.2599230735278435) -- (1.03,2.1921181601070723) -- (1.14,2.1921181601070723) -- (1.14,2.06990099900277) -- (1.69,2.06990099900277) -- (1.69,1.8566333001716968) -- (1.85,1.8566333001716968) -- (1.85,0.9199900641609603) -- (2.25,0.9199900641609603) -- (2.25,0.9082655880772972) -- (2.89,0.9082655880772972) -- (2.89,0.9070318103901573) -- (3.2,0.9070318103901573) -- (3.2,0.8901646066550629) -- (3.7,0.8901646066550629); 
\draw[Set1-9-4](0.0,50) -- (0.06,50) -- (0.06,100.0) -- (0.14,100.0) -- (0.14,57.73502691896258) -- (0.16,57.73502691896258) -- (0.16,31.622776601683793) -- (0.21,31.622776601683793) -- (0.21,24.253562503633297) -- (0.25,24.253562503633297) -- (0.25,17.960530202677493) -- (0.26,17.960530202677493) -- (0.26,14.433756729740645) -- (0.33,14.433756729740645) -- (0.33,14.14213562373095) -- (0.37,14.14213562373095) -- (0.37,10.846522890932809) -- (0.6,10.846522890932809) -- (0.6,3.768891807222045) -- (0.65,3.768891807222045) -- (0.65,3.4020690871988584) -- (0.75,3.4020690871988584) -- (0.75,3.3296357910601344) -- (0.76,3.3296357910601344) -- (0.76,2.3688968483956714) -- (1.24,2.3688968483956714) -- (1.24,2.293553851298437) -- (1.52,2.293553851298437) -- (1.52,2.2869467687317413) -- (1.79,2.2869467687317413) -- (1.79,2.195814375455294) -- (1.87,2.195814375455294) -- (1.87,2.0765845843612776) -- (2.3,2.0765845843612776) -- (2.3,1.823312393723682) -- (2.39,1.823312393723682) -- (2.39,0.9287169073588276) -- (2.65,0.9287169073588276) -- (2.65,0.926164167516367) -- (3.21,0.926164167516367) -- (3.21,0.9255689080117763) -- (3.7,0.9255689080117763); 
\draw[Set1-9-5](0.0,50) -- (0.07,50) -- (0.07,100.0) -- (0.18,100.0) -- (0.18,57.73502691896258) -- (0.23,57.73502691896258) -- (0.23,31.622776601683793) -- (0.31,31.622776601683793) -- (0.31,17.960530202677493) -- (0.32,17.960530202677493) -- (0.32,16.43989873053573) -- (0.42,16.43989873053573) -- (0.42,11.322770341445956) -- (0.51,11.322770341445956) -- (0.51,10.91089451179962) -- (0.55,10.91089451179962) -- (0.55,10.846522890932809) -- (0.72,10.846522890932809) -- (0.72,8.42151921066519) -- (0.77,8.42151921066519) -- (0.77,8.362420100070908) -- (0.79,8.362420100070908) -- (0.79,3.064257065179478) -- (1.04,3.064257065179478) -- (1.04,3.002854067691021) -- (1.51,3.002854067691021) -- (1.51,2.374232208416312) -- (1.78,2.374232208416312) -- (1.78,2.2983951313055564) -- (2.35,2.2983951313055564) -- (2.35,1.8433375453644367) -- (2.4,1.8433375453644367) -- (2.4,1.8300168911338581) -- (2.5,1.8300168911338581) -- (2.5,1.1231701965980267) -- (3.7,1.1231701965980267); 
\draw[Set1-9-7](0.0,50) -- (0.08,50) -- (0.08,100.0) -- (0.29,100.0) -- (0.29,57.73502691896258) -- (0.34,57.73502691896258) -- (0.34,17.960530202677493) -- (0.36,17.960530202677493) -- (0.36,16.43989873053573) -- (0.49,16.43989873053573) -- (0.49,11.470786693528087) -- (0.81,11.470786693528087) -- (0.81,6.337242505244779) -- (0.82,6.337242505244779) -- (0.82,5.892556509887896) -- (0.85,5.892556509887896) -- (0.85,3.617873026462108) -- (1.09,3.617873026462108) -- (1.09,3.5944257734479472) -- (1.12,3.5944257734479472) -- (1.12,3.102360242788488) -- (1.62,3.102360242788488) -- (1.62,2.4000768036865967) -- (2.37,2.4000768036865967) -- (2.37,2.021956339963735) -- (2.86,2.021956339963735) -- (2.86,1.9392813966684825) -- (2.9,1.9392813966684825) -- (2.9,1.7500820370182728) -- (3.08,1.7500820370182728) -- (3.08,1.667361545440517) -- (3.1,1.667361545440517) -- (3.1,1.389022867463347) -- (3.34,1.389022867463347) -- (3.34,1.1458731212852642) -- (3.4,1.1458731212852642) -- (3.4,1.1419067262449563) -- (3.47,1.1419067262449563) -- (3.47,1.1335857732069365) -- (3.7,1.1335857732069365); 
\draw[Set1-9-8](0.0,50) -- (0.11,50) -- (0.11,100.0) -- (0.38,100.0) -- (0.38,28.86751345948129) -- (0.39,28.86751345948129) -- (0.39,18.56953381770519) -- (0.52,18.56953381770519) -- (0.52,17.960530202677493) -- (0.53,17.960530202677493) -- (0.53,12.216944435630522) -- (0.84,12.216944435630522) -- (0.84,5.892556509887896) -- (1.05,5.892556509887896) -- (1.05,5.547001962252291) -- (1.06,5.547001962252291) -- (1.06,5.513178464199712) -- (1.15,5.513178464199712) -- (1.15,4.402254531628119) -- (1.2,4.402254531628119) -- (1.2,4.1204282171516455) -- (1.27,4.1204282171516455) -- (1.27,4.09272754535029) -- (1.44,4.09272754535029) -- (1.44,3.971507353947688) -- (1.54,3.971507353947688) -- (1.54,3.3407655239053047) -- (1.8,3.3407655239053047) -- (1.8,2.749286996141075) -- (1.95,2.749286996141075) -- (1.95,2.5432863197073092) -- (2.41,2.5432863197073092) -- (2.41,2.5383654128340476) -- (2.42,2.5383654128340476) -- (2.42,2.15365246126974) -- (2.6,2.15365246126974) -- (2.6,2.1339479988815997) -- (2.83,2.1339479988815997) -- (2.83,2.0403914662519314) -- (3.19,2.0403914662519314) -- (3.19,1.8361765076453271) -- (3.32,1.8361765076453271) -- (3.32,1.6177502706648093) -- (3.33,1.6177502706648093) -- (3.33,1.558699215971374) -- (3.56,1.558699215971374) -- (3.56,1.1339501179510878) -- (3.7,1.1339501179510878); 
\draw[Set1-9-9](0.0,50) -- (0.2,50) -- (0.2,100.0) -- (0.4,100.0) -- (0.4,30.15113445777636) -- (0.45,30.15113445777636) -- (0.45,18.56953381770519) -- (0.54,18.56953381770519) -- (0.54,15.43033499620919) -- (0.74,15.43033499620919) -- (0.74,12.309149097933272) -- (0.9,12.309149097933272) -- (0.9,9.053574604251853) -- (0.94,9.053574604251853) -- (0.94,7.856742013183862) -- (1.07,7.856742013183862) -- (1.07,7.235746052924216) -- (1.26,7.235746052924216) -- (1.26,4.1204282171516455) -- (1.48,4.1204282171516455) -- (1.48,4.09272754535029) -- (1.59,4.09272754535029) -- (1.59,3.5623524993954825) -- (1.64,3.5623524993954825) -- (1.64,3.4647946433761856) -- (1.88,3.4647946433761856) -- (1.88,3.4565056491014174) -- (1.9,3.4565056491014174) -- (1.9,3.390317518104052) -- (1.99,3.390317518104052) -- (1.99,2.6135418674465836) -- (2.81,2.6135418674465836) -- (2.81,2.1703261485649126) -- (2.85,2.1703261485649126) -- (2.85,2.1621937503220527) -- (3.0,2.1621937503220527) -- (3.0,2.041666843894699) -- (3.23,2.041666843894699) -- (3.23,1.9167079295731937) -- (3.58,1.9167079295731937) -- (3.58,1.9083135479350184) -- (3.65,1.9083135479350184) -- (3.65,1.902778181530907) -- (3.66,1.902778181530907) -- (3.66,1.6572562308756484) -- (3.7,1.6572562308756484); 
\draw[Set1-9-1](0.0,50) -- (0.22,50) -- (0.22,100.0) -- (0.57,100.0) -- (0.57,30.15113445777636) -- (0.58,30.15113445777636) -- (0.58,22.360679774997894) -- (0.59,22.360679774997894) -- (0.59,18.56953381770519) -- (0.61,18.56953381770519) -- (0.61,16.012815380508712) -- (0.86,16.012815380508712) -- (0.86,12.309149097933272) -- (0.99,12.309149097933272) -- (0.99,10.369516947304252) -- (1.02,10.369516947304252) -- (1.02,10.101525445522107) -- (1.08,10.101525445522107) -- (1.08,9.053574604251853) -- (1.28,9.053574604251853) -- (1.28,4.170288281141495) -- (1.74,4.170288281141495) -- (1.74,4.12744171706498) -- (1.93,4.12744171706498) -- (1.93,4.1204282171516455) -- (2.03,4.1204282171516455) -- (2.03,2.8629916715693415) -- (2.43,2.8629916715693415) -- (2.43,2.823912473624525) -- (2.49,2.823912473624525) -- (2.49,2.718636239135184) -- (2.92,2.718636239135184) -- (2.92,2.6388990698462975) -- (2.98,2.6388990698462975) -- (2.98,2.3980056885558416) -- (3.35,2.3980056885558416) -- (3.35,2.38772994200364) -- (3.36,2.38772994200364) -- (3.36,2.2905435417568807) -- (3.53,2.2905435417568807) -- (3.53,2.1816595214404266) -- (3.7,2.1816595214404266); 
\draw[Set1-9-2](0.0,50) -- (0.3,50) -- (0.3,100.0) -- (0.62,100.0) -- (0.62,22.360679774997894) -- (0.66,22.360679774997894) -- (0.66,18.56953381770519) -- (0.68,18.56953381770519) -- (0.68,16.012815380508712) -- (0.88,16.012815380508712) -- (0.88,12.309149097933272) -- (1.1,12.309149097933272) -- (1.1,10.369516947304252) -- (1.32,10.369516947304252) -- (1.32,5.652334189442215) -- (1.33,5.652334189442215) -- (1.33,4.3314808182421) -- (1.35,4.3314808182421) -- (1.35,4.291410754228745) -- (1.6,4.291410754228745) -- (1.6,4.248592886620874) -- (2.09,4.248592886620874) -- (2.09,3.582871819500093) -- (2.21,3.582871819500093) -- (2.21,3.001501125938321) -- (2.44,3.001501125938321) -- (2.44,2.9814239699997196) -- (2.54,2.9814239699997196) -- (2.54,2.969569354582493) -- (2.88,2.969569354582493) -- (2.88,2.719641466102106) -- (3.29,2.719641466102106) -- (3.29,2.3993858358289817) -- (3.48,2.3993858358289817) -- (3.48,2.3980056885558416) -- (3.62,2.3980056885558416) -- (3.62,2.2905435417568807) -- (3.7,2.2905435417568807); 
\draw[Set1-9-3](0.0,50) -- (0.35,50) -- (0.35,100.0) -- (0.63,100.0) -- (0.63,27.735009811261456) -- (0.64,27.735009811261456) -- (0.64,22.360679774997894) -- (0.78,22.360679774997894) -- (0.78,16.012815380508712) -- (0.92,16.012815380508712) -- (0.92,13.483997249264842) -- (0.98,13.483997249264842) -- (0.98,12.309149097933272) -- (1.22,12.309149097933272) -- (1.22,10.369516947304252) -- (1.36,10.369516947304252) -- (1.36,8.333333333333334) -- (1.4,8.333333333333334) -- (1.4,7.832604499879574) -- (1.42,7.832604499879574) -- (1.42,5.652334189442215) -- (1.57,5.652334189442215) -- (1.57,4.583492485141057) -- (1.66,4.583492485141057) -- (1.66,4.3314808182421) -- (2.1,4.3314808182421) -- (2.1,4.291410754228745) -- (2.55,4.291410754228745) -- (2.55,4.248592886620874) -- (2.57,4.248592886620874) -- (2.57,3.117398431942748) -- (2.93,3.117398431942748) -- (2.93,3.001501125938321) -- (2.95,3.001501125938321) -- (2.95,2.795084971874737) -- (2.99,2.795084971874737) -- (2.99,2.7389551783238835) -- (3.14,2.7389551783238835) -- (3.14,2.719641466102106) -- (3.7,2.719641466102106); 
\draw[Set1-9-4](0.0,50) -- (0.48,50) -- (0.48,100.0) -- (0.71,100.0) -- (0.71,22.360679774997894) -- (0.8,22.360679774997894) -- (0.8,16.666666666666668) -- (1.11,16.666666666666668) -- (1.11,16.012815380508712) -- (1.29,16.012815380508712) -- (1.29,13.483997249264842) -- (1.45,13.483997249264842) -- (1.45,6.25) -- (1.68,6.25) -- (1.68,5.652334189442215) -- (1.72,5.652334189442215) -- (1.72,5.089865985592876) -- (1.77,5.089865985592876) -- (1.77,4.800153607373193) -- (1.96,4.800153607373193) -- (1.96,4.583492485141057) -- (2.31,4.583492485141057) -- (2.31,4.3314808182421) -- (2.46,4.3314808182421) -- (2.46,4.291410754228745) -- (2.69,4.291410754228745) -- (2.69,4.173919355648411) -- (2.7,4.173919355648411) -- (2.7,3.603749850782236) -- (2.72,3.603749850782236) -- (2.72,3.4689615656798014) -- (2.78,3.4689615656798014) -- (2.78,3.440104580768907) -- (3.22,3.440104580768907) -- (3.22,3.117398431942748) -- (3.7,3.117398431942748); 
\draw[Set1-9-5](0.0,50) -- (0.67,50) -- (0.67,100.0) -- (0.83,100.0) -- (0.83,25.0) -- (0.87,25.0) -- (0.87,22.360679774997894) -- (1.01,22.360679774997894) -- (1.01,16.666666666666668) -- (1.39,16.666666666666668) -- (1.39,16.012815380508712) -- (1.49,16.012815380508712) -- (1.49,9.205746178983233) -- (1.61,9.205746178983233) -- (1.61,7.053456158585982) -- (1.73,7.053456158585982) -- (1.73,6.286946134619315) -- (1.89,6.286946134619315) -- (1.89,5.089865985592876) -- (2.19,5.089865985592876) -- (2.19,4.672693135159978) -- (2.53,4.672693135159978) -- (2.53,4.3314808182421) -- (2.58,4.3314808182421) -- (2.58,4.291410754228745) -- (3.09,4.291410754228745) -- (3.09,3.8096965887972964) -- (3.44,3.8096965887972964) -- (3.44,3.7529331252040077) -- (3.7,3.7529331252040077); 
\draw[Set1-9-7](0.0,50) -- (0.69,50) -- (0.69,100.0) -- (0.91,100.0) -- (0.91,25.0) -- (1.0,25.0) -- (1.0,22.360679774997894) -- (1.13,22.360679774997894) -- (1.13,16.666666666666668) -- (1.46,16.666666666666668) -- (1.46,16.012815380508712) -- (1.5,16.012815380508712) -- (1.5,10.42572070285374) -- (1.71,10.42572070285374) -- (1.71,9.205746178983233) -- (1.97,9.205746178983233) -- (1.97,5.415303610738824) -- (1.98,5.415303610738824) -- (1.98,5.2631578947368425) -- (2.11,5.2631578947368425) -- (2.11,5.103103630798288) -- (2.22,5.103103630798288) -- (2.22,4.756514941544941) -- (2.26,4.756514941544941) -- (2.26,4.672693135159978) -- (2.68,4.672693135159978) -- (2.68,4.3478260869565215) -- (3.01,4.3478260869565215) -- (3.01,4.3314808182421) -- (3.46,4.3314808182421) -- (3.46,4.291410754228745) -- (3.67,4.291410754228745) -- (3.67,3.8096965887972964) -- (3.7,3.8096965887972964); 
\draw[Set1-9-8](0.0,50) -- (0.7,50) -- (0.7,100.0) -- (0.93,100.0) -- (0.93,25.0) -- (1.18,25.0) -- (1.18,16.666666666666668) -- (1.63,16.666666666666668) -- (1.63,10.660035817780521) -- (1.81,10.660035817780521) -- (1.81,10.42572070285374) -- (1.84,10.42572070285374) -- (1.84,9.407208683835972) -- (2.04,9.407208683835972) -- (2.04,7.179581586177381) -- (2.08,7.179581586177381) -- (2.08,5.2631578947368425) -- (2.29,5.2631578947368425) -- (2.29,5.103103630798288) -- (2.47,5.103103630798288) -- (2.47,5.012547071170855) -- (2.73,5.012547071170855) -- (2.73,4.756514941544941) -- (2.82,4.756514941544941) -- (2.82,4.449941594899848) -- (3.39,4.449941594899848) -- (3.39,4.3478260869565215) -- (3.63,4.3478260869565215) -- (3.63,4.3314808182421) -- (3.7,4.3314808182421); 
\draw[Set1-9-9](0.0,50) -- (0.89,50) -- (0.89,100.0) -- (0.95,100.0) -- (0.95,25.0) -- (1.23,25.0) -- (1.23,21.320071635561042) -- (1.3,21.320071635561042) -- (1.3,16.666666666666668) -- (1.65,16.666666666666668) -- (1.65,10.660035817780521) -- (2.07,10.660035817780521) -- (2.07,10.42572070285374) -- (2.13,10.42572070285374) -- (2.13,5.735393346764043) -- (2.34,5.735393346764043) -- (2.34,5.103103630798288) -- (2.97,5.103103630798288) -- (2.97,4.756514941544941) -- (3.12,4.756514941544941) -- (3.12,4.637388957601683) -- (3.54,4.637388957601683) -- (3.54,4.3478260869565215) -- (3.7,4.3478260869565215); 

%% file: figures-Poisson/2D-Poisson/true-Poisson.tex
\draw[Set1-9-1,rounded corners=0.5mm](0,10) -- (0.10186588549529918,10) -- (0.10186588549529918,4.930485835696822) -- (0.31046981370323024,4.930485835696822) -- (0.31046981370323024,2.0079185036670832) -- (0.677243921205683,2.0079185036670832) -- (0.677243921205683,1.72868731644926) -- (0.9781450992170915,1.72868731644926) -- (0.9781450992170915,0.7247421436168267) -- (4.241909247936997,0.7247421436168267) -- (4.241909247936997,0.20784268820072616) -- (10,0.20784268820072616); 
\draw[Set1-9-2,rounded corners=0.5mm](0,10) -- (0.31046981370323024,10) -- (0.31046981370323024,4.930485835696822) -- (0.3641638657491214,4.930485835696822) -- (0.3641638657491214,2.591906163925323) -- (0.677243921205683,2.591906163925323) -- (0.677243921205683,2.0079185036670832) -- (0.9781450992170915,2.0079185036670832) -- (0.9781450992170915,1.72868731644926) -- (3.5841294106779937,1.72868731644926) -- (3.5841294106779937,1.610895264026884) -- (3.8516359793366632,1.610895264026884) -- (3.8516359793366632,1.05242645217414) -- (4.241909247936997,1.05242645217414) -- (4.241909247936997,0.7247421436168267) -- (10,0.7247421436168267); 
\draw[Set1-9-3,rounded corners=0.5mm](0,10) -- (0.3641638657491214,10) -- (0.3641638657491214,4.930485835696822) -- (0.677243921205683,4.930485835696822) -- (0.677243921205683,2.591906163925323) -- (0.9781450992170915,2.591906163925323) -- (0.9781450992170915,2.0079185036670832) -- (3.5841294106779937,2.0079185036670832) -- (3.5841294106779937,1.72868731644926) -- (3.8516359793366632,1.72868731644926) -- (3.8516359793366632,1.610895264026884) -- (4.241909247936997,1.610895264026884) -- (4.241909247936997,1.05242645217414) -- (4.851952865101655,1.05242645217414) -- (4.851952865101655,0.9710104757039084) -- (10,0.9710104757039084); 
\draw[Set1-9-4,rounded corners=0.5mm](0,10) -- (0.677243921205683,10) -- (0.677243921205683,4.930485835696822) -- (0.845706128177347,4.930485835696822) -- (0.845706128177347,4.581887195463421) -- (0.9781450992170915,4.581887195463421) -- (0.9781450992170915,2.591906163925323) -- (1.9853246187176872,2.591906163925323) -- (1.9853246187176872,2.338254242453604) -- (3.2920633518763247,2.338254242453604) -- (3.2920633518763247,2.0870058925705615) -- (3.5841294106779937,2.0870058925705615) -- (3.5841294106779937,2.0079185036670832) -- (3.8516359793366632,2.0079185036670832) -- (3.8516359793366632,1.72868731644926) -- (4.241909247936997,1.72868731644926) -- (4.241909247936997,1.610895264026884) -- (4.314801458681896,1.610895264026884) -- (4.314801458681896,1.524537964542867) -- (4.851952865101655,1.524537964542867) -- (4.851952865101655,1.05242645217414) -- (10,1.05242645217414); 
\draw[Set1-9-5,rounded corners=0.5mm](0,10) -- (0.845706128177347,10) -- (0.845706128177347,4.930485835696822) -- (0.9781450992170915,4.930485835696822) -- (0.9781450992170915,4.581887195463421) -- (1.2049838560529653,4.581887195463421) -- (1.2049838560529653,3.637995142987407) -- (1.5959518120150045,3.637995142987407) -- (1.5959518120150045,2.7708381137032427) -- (1.9853246187176872,2.7708381137032427) -- (1.9853246187176872,2.591906163925323) -- (3.2920633518763247,2.591906163925323) -- (3.2920633518763247,2.338254242453604) -- (3.5841294106779937,2.338254242453604) -- (3.5841294106779937,2.0870058925705615) -- (3.8516359793366632,2.0870058925705615) -- (3.8516359793366632,2.0079185036670832) -- (4.241909247936997,2.0079185036670832) -- (4.241909247936997,1.72868731644926) -- (4.314801458681896,1.72868731644926) -- (4.314801458681896,1.610895264026884) -- (4.851952865101655,1.610895264026884) -- (4.851952865101655,1.524537964542867) -- (4.940583918257177,1.524537964542867) -- (4.940583918257177,1.1951739457940587) -- (10,1.1951739457940587); 
\draw[Set1-9-7,rounded corners=0.5mm](0,10) -- (0.9781450992170915,10) -- (0.9781450992170915,4.930485835696822) -- (1.2049838560529653,4.930485835696822) -- (1.2049838560529653,4.581887195463421) -- (1.5959518120150045,4.581887195463421) -- (1.5959518120150045,3.637995142987407) -- (1.9853246187176872,3.637995142987407) -- (1.9853246187176872,2.7708381137032427) -- (3.2920633518763247,2.7708381137032427) -- (3.2920633518763247,2.591906163925323) -- (3.5841294106779937,2.591906163925323) -- (3.5841294106779937,2.338254242453604) -- (3.8516359793366632,2.338254242453604) -- (3.8516359793366632,2.0870058925705615) -- (4.241909247936997,2.0870058925705615) -- (4.241909247936997,2.0079185036670832) -- (4.314801458681896,2.0079185036670832) -- (4.314801458681896,1.72868731644926) -- (4.794542514676822,1.72868731644926) -- (4.794542514676822,1.6755210054101533) -- (4.851952865101655,1.6755210054101533) -- (4.851952865101655,1.610895264026884) -- (4.940583918257177,1.610895264026884) -- (4.940583918257177,1.524537964542867) -- (10,1.524537964542867); 
\draw[Set1-9-8,rounded corners=0.5mm](0,10) -- (1.2049838560529653,10) -- (1.2049838560529653,4.930485835696822) -- (1.5959518120150045,4.930485835696822) -- (1.5959518120150045,4.581887195463421) -- (1.6984939544978694,4.581887195463421) -- (1.6984939544978694,4.052873291294706) -- (1.9853246187176872,4.052873291294706) -- (1.9853246187176872,3.637995142987407) -- (2.070610763770609,3.637995142987407) -- (2.070610763770609,3.4568925226434937) -- (2.3587134673935317,3.4568925226434937) -- (2.3587134673935317,3.232241320620715) -- (3.2148274864789874,3.232241320620715) -- (3.2148274864789874,3.162497075835362) -- (3.2920633518763247,3.162497075835362) -- (3.2920633518763247,2.7708381137032427) -- (3.5841294106779937,2.7708381137032427) -- (3.5841294106779937,2.591906163925323) -- (3.8516359793366632,2.591906163925323) -- (3.8516359793366632,2.338254242453604) -- (4.241909247936997,2.338254242453604) -- (4.241909247936997,2.0870058925705615) -- (4.314801458681896,2.0870058925705615) -- (4.314801458681896,2.0079185036670832) -- (4.794542514676822,2.0079185036670832) -- (4.794542514676822,1.72868731644926) -- (4.851952865101655,1.72868731644926) -- (4.851952865101655,1.6755210054101533) -- (4.940583918257177,1.6755210054101533) -- (4.940583918257177,1.610895264026884) -- (10,1.610895264026884); 
\draw[Set1-9-9,rounded corners=0.5mm](0,10) -- (1.5959518120150045,10) -- (1.5959518120150045,4.930485835696822) -- (1.6984939544978694,4.930485835696822) -- (1.6984939544978694,4.581887195463421) -- (1.9853246187176872,4.581887195463421) -- (1.9853246187176872,4.052873291294706) -- (2.070610763770609,4.052873291294706) -- (2.070610763770609,3.637995142987407) -- (2.3587134673935317,3.637995142987407) -- (2.3587134673935317,3.4568925226434937) -- (3.2148274864789874,3.4568925226434937) -- (3.2148274864789874,3.232241320620715) -- (3.2920633518763247,3.232241320620715) -- (3.2920633518763247,3.162497075835362) -- (3.3954228194168206,3.162497075835362) -- (3.3954228194168206,3.0107577698612715) -- (3.5841294106779937,3.0107577698612715) -- (3.5841294106779937,2.7708381137032427) -- (3.8516359793366632,2.7708381137032427) -- (3.8516359793366632,2.591906163925323) -- (4.241909247936997,2.591906163925323) -- (4.241909247936997,2.338254242453604) -- (4.314801458681896,2.338254242453604) -- (4.314801458681896,2.0870058925705615) -- (4.794542514676822,2.0870058925705615) -- (4.794542514676822,2.0079185036670832) -- (4.851952865101655,2.0079185036670832) -- (4.851952865101655,1.72868731644926) -- (4.940583918257177,1.72868731644926) -- (4.940583918257177,1.6755210054101533) -- (10,1.6755210054101533); 
\draw[Set1-9-1,rounded corners=0.5mm](0,10) -- (1.6984939544978694,10) -- (1.6984939544978694,4.930485835696822) -- (1.9853246187176872,4.930485835696822) -- (1.9853246187176872,4.581887195463421) -- (2.070610763770609,4.581887195463421) -- (2.070610763770609,4.052873291294706) -- (2.3587134673935317,4.052873291294706) -- (2.3587134673935317,3.637995142987407) -- (3.2148274864789874,3.637995142987407) -- (3.2148274864789874,3.4568925226434937) -- (3.2920633518763247,3.4568925226434937) -- (3.2920633518763247,3.232241320620715) -- (3.3954228194168206,3.232241320620715) -- (3.3954228194168206,3.162497075835362) -- (3.5841294106779937,3.162497075835362) -- (3.5841294106779937,3.0107577698612715) -- (3.8516359793366632,3.0107577698612715) -- (3.8516359793366632,2.7708381137032427) -- (4.241909247936997,2.7708381137032427) -- (4.241909247936997,2.591906163925323) -- (4.314801458681896,2.591906163925323) -- (4.314801458681896,2.338254242453604) -- (4.794542514676822,2.338254242453604) -- (4.794542514676822,2.0870058925705615) -- (4.851952865101655,2.0870058925705615) -- (4.851952865101655,2.0079185036670832) -- (4.940583918257177,2.0079185036670832) -- (4.940583918257177,1.72868731644926) -- (10,1.72868731644926); 
\draw[Set1-9-2,rounded corners=0.5mm](0,10) -- (1.9853246187176872,10) -- (1.9853246187176872,4.930485835696822) -- (2.070610763770609,4.930485835696822) -- (2.070610763770609,4.581887195463421) -- (2.3587134673935317,4.581887195463421) -- (2.3587134673935317,4.052873291294706) -- (3.2148274864789874,4.052873291294706) -- (3.2148274864789874,3.637995142987407) -- (3.2920633518763247,3.637995142987407) -- (3.2920633518763247,3.4568925226434937) -- (3.3954228194168206,3.4568925226434937) -- (3.3954228194168206,3.232241320620715) -- (3.5841294106779937,3.232241320620715) -- (3.5841294106779937,3.162497075835362) -- (3.8516359793366632,3.162497075835362) -- (3.8516359793366632,3.0107577698612715) -- (4.241909247936997,3.0107577698612715) -- (4.241909247936997,2.7708381137032427) -- (4.314801458681896,2.7708381137032427) -- (4.314801458681896,2.591906163925323) -- (4.794542514676822,2.591906163925323) -- (4.794542514676822,2.338254242453604) -- (4.851952865101655,2.338254242453604) -- (4.851952865101655,2.0870058925705615) -- (4.940583918257177,2.0870058925705615) -- (4.940583918257177,2.0079185036670832) -- (10,2.0079185036670832); 
\draw[Set1-9-3,rounded corners=0.5mm](0,10) -- (2.070610763770609,10) -- (2.070610763770609,4.930485835696822) -- (2.3587134673935317,4.930485835696822) -- (2.3587134673935317,4.581887195463421) -- (2.493788162337189,4.581887195463421) -- (2.493788162337189,4.178342040918265) -- (3.2148274864789874,4.178342040918265) -- (3.2148274864789874,4.052873291294706) -- (3.2920633518763247,4.052873291294706) -- (3.2920633518763247,3.637995142987407) -- (3.3954228194168206,3.637995142987407) -- (3.3954228194168206,3.4568925226434937) -- (3.5841294106779937,3.4568925226434937) -- (3.5841294106779937,3.232241320620715) -- (3.8516359793366632,3.232241320620715) -- (3.8516359793366632,3.162497075835362) -- (4.241909247936997,3.162497075835362) -- (4.241909247936997,3.0107577698612715) -- (4.314801458681896,3.0107577698612715) -- (4.314801458681896,2.7708381137032427) -- (4.794542514676822,2.7708381137032427) -- (4.794542514676822,2.591906163925323) -- (4.851952865101655,2.591906163925323) -- (4.851952865101655,2.338254242453604) -- (4.940583918257177,2.338254242453604) -- (4.940583918257177,2.0870058925705615) -- (10,2.0870058925705615); 
\draw[Set1-9-4,rounded corners=0.5mm](0,10) -- (2.3587134673935317,10) -- (2.3587134673935317,4.930485835696822) -- (2.493788162337189,4.930485835696822) -- (2.493788162337189,4.581887195463421) -- (3.2148274864789874,4.581887195463421) -- (3.2148274864789874,4.178342040918265) -- (3.2920633518763247,4.178342040918265) -- (3.2920633518763247,4.052873291294706) -- (3.3954228194168206,4.052873291294706) -- (3.3954228194168206,3.637995142987407) -- (3.5841294106779937,3.637995142987407) -- (3.5841294106779937,3.4568925226434937) -- (3.8516359793366632,3.4568925226434937) -- (3.8516359793366632,3.232241320620715) -- (4.241909247936997,3.232241320620715) -- (4.241909247936997,3.162497075835362) -- (4.314801458681896,3.162497075835362) -- (4.314801458681896,3.0107577698612715) -- (4.794542514676822,3.0107577698612715) -- (4.794542514676822,2.7708381137032427) -- (4.851952865101655,2.7708381137032427) -- (4.851952865101655,2.591906163925323) -- (4.940583918257177,2.591906163925323) -- (4.940583918257177,2.338254242453604) -- (10,2.338254242453604); 
\draw[Set1-9-5,rounded corners=0.5mm](0,10) -- (2.493788162337189,10) -- (2.493788162337189,4.930485835696822) -- (3.2148274864789874,4.930485835696822) -- (3.2148274864789874,4.581887195463421) -- (3.2920633518763247,4.581887195463421) -- (3.2920633518763247,4.178342040918265) -- (3.3954228194168206,4.178342040918265) -- (3.3954228194168206,4.052873291294706) -- (3.5841294106779937,4.052873291294706) -- (3.5841294106779937,3.637995142987407) -- (3.8516359793366632,3.637995142987407) -- (3.8516359793366632,3.4568925226434937) -- (4.241909247936997,3.4568925226434937) -- (4.241909247936997,3.232241320620715) -- (4.314801458681896,3.232241320620715) -- (4.314801458681896,3.162497075835362) -- (4.794542514676822,3.162497075835362) -- (4.794542514676822,3.0107577698612715) -- (4.851952865101655,3.0107577698612715) -- (4.851952865101655,2.7708381137032427) -- (4.940583918257177,2.7708381137032427) -- (4.940583918257177,2.591906163925323) -- (10,2.591906163925323); 
\draw[Set1-9-7,rounded corners=0.5mm](0,10) -- (3.2148274864789874,10) -- (3.2148274864789874,4.930485835696822) -- (3.2920633518763247,4.930485835696822) -- (3.2920633518763247,4.581887195463421) -- (3.3954228194168206,4.581887195463421) -- (3.3954228194168206,4.178342040918265) -- (3.5841294106779937,4.178342040918265) -- (3.5841294106779937,4.052873291294706) -- (3.8516359793366632,4.052873291294706) -- (3.8516359793366632,3.637995142987407) -- (4.241909247936997,3.637995142987407) -- (4.241909247936997,3.4568925226434937) -- (4.314801458681896,3.4568925226434937) -- (4.314801458681896,3.232241320620715) -- (4.794542514676822,3.232241320620715) -- (4.794542514676822,3.162497075835362) -- (4.851952865101655,3.162497075835362) -- (4.851952865101655,3.0107577698612715) -- (4.940583918257177,3.0107577698612715) -- (4.940583918257177,2.7708381137032427) -- (10,2.7708381137032427); 
\draw[Set1-9-8,rounded corners=0.5mm](0,10) -- (3.2920633518763247,10) -- (3.2920633518763247,4.930485835696822) -- (3.3954228194168206,4.930485835696822) -- (3.3954228194168206,4.581887195463421) -- (3.5841294106779937,4.581887195463421) -- (3.5841294106779937,4.178342040918265) -- (3.8516359793366632,4.178342040918265) -- (3.8516359793366632,4.052873291294706) -- (4.11742195460317,4.052873291294706) -- (4.11742195460317,3.909465584833342) -- (4.241909247936997,3.909465584833342) -- (4.241909247936997,3.637995142987407) -- (4.314801458681896,3.637995142987407) -- (4.314801458681896,3.4568925226434937) -- (4.794542514676822,3.4568925226434937) -- (4.794542514676822,3.232241320620715) -- (4.851952865101655,3.232241320620715) -- (4.851952865101655,3.162497075835362) -- (4.940583918257177,3.162497075835362) -- (4.940583918257177,3.0107577698612715) -- (10,3.0107577698612715); 
\draw[Set1-9-9,rounded corners=0.5mm](0,10) -- (3.3954228194168206,10) -- (3.3954228194168206,4.930485835696822) -- (3.5841294106779937,4.930485835696822) -- (3.5841294106779937,4.581887195463421) -- (3.8516359793366632,4.581887195463421) -- (3.8516359793366632,4.178342040918265) -- (4.11742195460317,4.178342040918265) -- (4.11742195460317,4.052873291294706) -- (4.241909247936997,4.052873291294706) -- (4.241909247936997,3.909465584833342) -- (4.314801458681896,3.909465584833342) -- (4.314801458681896,3.637995142987407) -- (4.794542514676822,3.637995142987407) -- (4.794542514676822,3.4568925226434937) -- (4.851952865101655,3.4568925226434937) -- (4.851952865101655,3.232241320620715) -- (4.940583918257177,3.232241320620715) -- (4.940583918257177,3.162497075835362) -- (10,3.162497075835362); 